\documentclass[reqno,12pt]{amsart}

\usepackage{amsmath,amssymb,mathrsfs,mathtools,verbatim,bbm,kpfonts,upgreek,slashed} 
\usepackage[backrefs]{amsrefs}
\usepackage[protrusion=true,babel=true]{microtype}
\usepackage[T1]{fontenc}
\usepackage[english]{babel}
\usepackage[margin=1in]{geometry}
\usepackage[onehalfspacing]{setspace}
\usepackage[pdfusetitle,pagebackref,hypertexnames=false,bookmarks=false]{hyperref}
\numberwithin{equation}{section}							
\usepackage[nameinlink,noabbrev]{cleveref}
\expandafter\def\csname ver@etex.sty\endcsname{3000/12/31}

\usepackage{autonum}										

1

\allowdisplaybreaks

\let\originalleft\left
\let\originalright\right
\renewcommand{\left}{\mathopen{}\mathclose\bgroup\originalleft}
\renewcommand{\right}{\aftergroup\egroup\originalright}

\makeatletter
\renewcommand*{\eqref}[1]{\hyperref[{#1}]{\textup{\tagform@{\ref*{#1}}}}}		
\makeatother

\newcommand{\rd}{\rm{d}}

\newcommand{\rG}{\rm{G}}

\newcommand{\rU}{\rm{U}}

\newcommand{\cC}{\mathcal{C}}

\newcommand{\cE}{\mathcal{E}}

\newcommand{\cG}{\mathcal{G}}

\newcommand{\cL}{\mathcal{L}}

\newcommand{\cS}{\mathcal{S}}

\newcommand{\fg}{{\mathfrak g}}

\newcommand{\Spin}{\rm{Spin}}

\newcommand{\Ad}{\mathrm{Ad}}
\newcommand{\Aut}{\mathrm{Aut}}

\newcommand{\End}{{\mathrm{End}}}

\newcommand{\Lie}{\mathrm{Lie}}

\newcommand{\ad}{\mathrm{ad}}

\newcommand{\coker}{\mathop{\mathrm{coker}}}

\newcommand{\del}{\partial}

\newcommand{\id}{\mathrm{id}}
\newcommand{\im}{\mathop{\mathrm{im}}}
\renewcommand{\Im}{\mathop{\mathrm{Im}}}

\newcommand{\loc}{\rm{loc}}

\renewcommand{\Re}{\mathop{\mathrm{Re}}}

\newcommand{\tr}{\mathop{\mathrm{tr}}\nolimits}

\newcommand{\ev}{\mathrm{ev}}

\newcommand{\qandq}{\quad\text{and}\quad}

\def\Id{\mathbbm{1}}

\def\N{\mathbb{N}}

\def\ad{\mathrm{ad}}

\def\Aut{\mathrm{Aut}}
\def\rd{\mathrm{d}\hspace{-.1em}}

\def\coker{\mathrm{coker}}
\def\dist{\mathrm{dist}}

\def\End{\mathrm{End}}

\def\Im{\mathrm{Im}}

\def\Lie{\mathrm{Lie}}

\def\Re{\mathrm{Re}}

\def\Spin{\mathrm{Spin}}

\def\SU{\mathrm{SU}}
\def\tr{\mathrm{tr}}

\def\vol{\mathrm{vol}}
\def\Vol{\mathrm{Vol}}

\def\D{\slashed{D}}

\def\g{\mathfrak{g}}

\def\mon{\mathrm{mon}}

\def\<{\mathopen{}\left<}
\def\>{\right>\mathclose{}}
\def\({\mathopen{}\left(}
\def\){\right)\mathclose{}}

\renewcommand\epsilon{\varepsilon}

\newtheorem{theorem}{Theorem}[section]
\newtheorem{Mtheorem}{Main Theorem}
\newtheorem*{acknowledgment}{Acknowledgment}

\newtheorem{corollary}[theorem]{Corollary}

\newtheorem{definition}[theorem]{Definition}

\newtheorem{lemma}[theorem]{Lemma}

\newtheorem{proposition}[theorem]{Proposition}
\newtheorem{remark}[theorem]{Remark}

\newtheorem{hypoth}[theorem]{Hypothesis}
\crefname{theorem}{Theorem}{Theorems}						
\creflabelformat{theorem}{#2{#1}#3}							
\crefname{Mtheorem}{Main Theorem}{Main Theorems}			
\creflabelformat{Mtheorem}{#2{#1}#3}						
\crefname{lemma}{Lemma}{Lemmata}							
\creflabelformat{lemma}{#2{#1}#3}							
\crefname{corollary}{Corollary}{Corollaries}				
\creflabelformat{corollary}{#2{#1}#3}						
\crefname{proposition}{Proposition}{Propositions}			
\creflabelformat{proposition}{#2{#1}#3}						
\crefname{ineq}{inequality}{inequalities}					
\creflabelformat{ineq}{#2{\upshape(#1)}#3}					
\crefname{cond}{condition}{conditions}						
\creflabelformat{cond}{#2{\upshape(#1)}#3}					
\crefname{hypoth}{Hypothesis}{Hypotheses}					
\creflabelformat{hypoth}{#2{#1}#3}							
\crefname{def}{Definition}{Definitions}						
\creflabelformat{def}{#2{#1}#3}								
\crefname{appsec}{Appendix}{Appendices}

\crefname{sec}{Section}{Sections}

\title{The asymptotic geometry of $\rG_2$-monopoles}

\author{Daniel Fadel}
\address[Daniel Fadel]{Universidade Estadual de Campinas, Campinas, Brazil / Université de Bretagne Occidentale, Brest, France}
\urladdr{\href{https://sites.google.com/view/daniel-fadel-math-homepage/home}{sites.google.com/view/daniel-fadel-math-homepage/home}}
\email{\href{mailto:fadel.daniel@gmail.com}{fadel.daniel@gmail.com}}

\author{\'Akos Nagy}
\address[\'Akos Nagy]{University of California, Santa Barbara, CA, USA}
\urladdr{\href{https://akosnagy.com}{akosnagy.com}}
\email{\href{mailto:contact@akosnagy.com}{contact@akosnagy.com}}

\author{Gon\c{c}alo Oliveira}
\address[Gon\c{c}alo Oliveira]{IST Austria}
\urladdr{\href{https://sites.google.com/view/goncalo-oliveira-math-webpage/home}{sites.google.com/view/goncalo-oliveira-math-webpage/home}}
\email{\href{mailto:galato97@gmail.com}{galato97@gmail.com}}

\date{\today}

\calclayout
\hypersetup{
	pdffitwindow	 = true,
	unicode			 = true,
	pdffitwindow	 = true,
	pdftoolbar		 = false,
	pdfmenubar		 = false,
	pdfstartview	 = {FitH},
	pdftitle 		 = {},
	pdfkeywords		 = {},
	pdfauthor		 = {Daniel Fadel, \'Akos Nagy, and Gon\c{c}alo Oliveira},
	colorlinks		 = true,
	linkcolor		 = black,
	citecolor		 = black,
	filecolor		 = blue,
	urlcolor		 = blue
}

\begin{document}

\begin{abstract}
	This article investigates the asymptotics of $\rG_2$-monopoles.

	First, we prove that when the underlying $\rG_2$-manifold is nonparabolic (i.e. admits a positive Green's function), finite intermediate energy monopoles with bounded curvature have finite mass. The second main result restricts to the case when the underlying $\rG_2$-manifold is asymptotically conical. In this situation, we deduce sharp decay estimates and that the connection converges, along the end, to a pseudo-Hermitian--Yang--Mills connection over the asymptotic cone.
	
	Finally, our last result exhibits a Fredholm setup describing the moduli space of finite intermediate energy monopoles on an asymptotically conical $\rG_2$-manifold.
\end{abstract}

\maketitle

\tableofcontents

\section{Introduction}\label{sec:Intro}

\subsection{Context}

An important problem in $\rG_2$ geometry is to develop methods to distinguish $\rG_2$-manifolds. This problem can be put in several ways, and recent advances produced invariants able to detect connected components of the moduli space of $\rG_2$-holonomy metrics \cites{Crowley2014,Crowley2015,Crowley2015_2}.

Other approaches intended at producing invariants of $\rG_2$-manifolds aim to produce enumerative theories counting special submanifolds and gauge fields. For example, in \cite{J1} Joyce alluded to the possibility of constructing such an enumerative invariant of $\rG_2$-manifolds by ``counting'' rigid, compact, and coassociative submanifolds (see also \cite{J2}). On the other hand, Donaldson and Segal, in \cite{Donaldson2009}, proposed an enumerative invariant of certain (noncompact) $\rG_2$-manifolds by considering $\rG_2$-monopoles instead. They further suggest that this might be easier to define and possibly related to a more direct coassociative ``count''.

The underlying idea behind this proposal is inspired by Taubes' $Gr = SW$ Theorem in \cite{Taubes1999} for $4$-dimensional symplectic manifolds. The similarities stem from the fact that the Seiberg--Witten (SW) invariant is obtained from gauge theory while the Gromov-Witten (Gr) invariant is obtained from holomorphic curves, which in a symplectic manifold are calibrated, just like coassociatives in a $\rG_2$-holonomy manifold are.

The study of $\rG_2$-monopoles was initiated in \cites{Oliveira2014_thesis,Cherkis2015}. In \cite{Oliveira2014}, the third author gave evidence supporting the Donaldson--Segal program by finding families of $\rG_2$-monopoles parametrized by a positive real number $m$, called the \emph{mass}, and showed that in the large mass limit these monopoles concentrate along a compact coassociative submanifold.

This paper is the first installment of a series of papers aimed to study the Donaldson and Segal \cite{Donaldson2009} program, that is the relation between $\rG_2$-monopoles and coassociative submanifolds.

The goal of this article is to show that several of the asymptotic features satisfied by these examples are in fact general phenomena which follow from natural assumptions such as finiteness of the relevant energy. This is a very much needed development in order to justify the choice of function spaces to be used in a satisfactory moduli theory. In the next article in this series the first and third author will be dealing with investigating the bubbling of $\rG_2$-monopoles along coassociatives \cite{Fadel}.

More about the other gauge theoretical approaches for producing invariants of $\rG_2$-manifolds can, for example, be found in \cites{Donaldson2009,SaEarp2009,Walpuski2013,Hay3,Doan,Doan_Walpuski}.

\smallskip

\subsection{Summary}\label{par:Summary}

Let $(X^7, \varphi)$ be a noncompact, complete, and irreducible $\rG_2$-manifold. We respectively denote by $g$ and $\ast$ the metric and Hodge star operator induced by the $\rG_2$-structure $\varphi \in \Omega^3(X)$. We also let $\psi = \ast \varphi \in \Omega^4 (X)$. Given a compact Lie group $\rG$ with Lie algebra $\mathfrak{g}$ and a principal $\rG$-bundle $P$ over $X$, we consider pairs $(\nabla, \Phi)$, where $\nabla$ is a smooth connection on $P$ with curvature $F_\nabla \in \Omega^2 (X, \mathfrak{g}_P)$ and $\Phi$ a smooth section of $\mathfrak{g}_P = P \times_\Ad \mathfrak{g}$, called the \emph{Higgs field}. Such a pair $(\nabla, \Phi)$ is said to be a \emph{$\rG_2$-monopole} if 
\begin{equation}
	\ast (F_\nabla \wedge \psi) - \nabla \Phi = 0. \label{eq:Monopole}
\end{equation}
Furthermore, $\rG_2$-monopoles can be seen as (at least formally\footnote{In fact, formally, $\rG_2$-monopoles are also critical points for the \emph{Yang--Mills--Higgs (YMH) energy}:
\begin{equation}
	\mathcal{E} (\nabla, \Phi) = \int\limits_X \left( |F_\nabla|^2 + |\nabla \Phi|^2 \right) \ \vol_X.
\end{equation}
Here, we say formally because this energy need not be finite. Indeed, in contrast with the intermediate energy, the YMH energy is infinite for all known irreducible examples.}) critical points of the \emph{intermediate energy}:
\begin{equation}\label{eq:Intermediate_Energy}
	\mathcal{E}^\psi (\nabla, \Phi) = \int\limits_X \left( |F_\nabla \wedge \psi|^2 + |\nabla \Phi|^2 \right) \ \vol_X.
\end{equation} 
Note that $F_\nabla \wedge \psi$ only contains certain components of $F_\nabla$ and so the intermediate energy only controls part of the curvature of $\nabla$. Indeed, the 2-forms on $(X,\varphi)$ split into irreducible $\rG_2$-representations as $\Lambda^2 = \Lambda^2_7 \oplus \Lambda^2_{14}$, with the subscripts accounting for the dimension of the representation. Using this decomposition we can uniquely write $F_\nabla = F_\nabla^7 + F_\nabla^{14}$ and we find that $F_\nabla \wedge \psi = F_\nabla^7 \wedge \psi$. Thus, the intermediate energy only accounts for the ``smaller'' $F_\nabla^7$ component of the curvature. Furthermore, under certain technical and refined assumptions on the asymptotic behavior (see Section 1.4 in \cite{Oliveira2014_thesis}) it is in fact possible to prove that $\rG_2$-monopoles minimize $\cE^\psi$. In this article we drop such technical hypothesis and replace them by simpler more natural ones such as finiteness of the intermediate energy.

A word must be said about the reason for restricting to noncompact $\rG_2$-manifolds. Indeed, a short computation resulting from applying $\nabla^*$ to \cref{eq:Monopole} and using the Bianchi identity, shows that $\nabla^* \nabla \Phi = 0$ which in turns implies that $|\Phi|^2$ is subharmonic. Thus, if $X$ was to be compact then $|\Phi|$ would be constant and thus $\nabla \Phi = 0 = F_\nabla \wedge \psi$. In particular, $\nabla$ is a so-called $\rG_2$-instanton, which is a very interesting equation in itself. However, in this article we focus on ``pure'' $\rG_2$-monopoles and so we regard the case when $\nabla \Phi \neq 0$ as being more interesting. 

\smallskip

\subsection*{Notational warnings}

\begin{enumerate}

	\item Throughout this paper, we use
	\begin{equation}
		n = 7,
	\end{equation}
	the dimension of $X$, to emphasize how the dimension of the underlying manifold comes into play in the analysis. This is because many of our results hold for more general setups, in particular, for Yang--Mills--Higgs fields in different dimensions.

	\item For comparable quantities $\alpha$ and $\beta$, $\alpha \lesssim \beta$ means that there exists a real number $c > 0$, that is independent of the variable relevant in the given context, such that $\alpha \leqslant c \beta$.

\end{enumerate}

\smallskip

\subsection*{Main results}

Recall that a $\rG_2$-holonomy Riemannian manifold $(X, \varphi)$ is Ricci-flat. Therefore, by the Cheeger--Gromoll splitting theorem \cite{cheeger1971splitting}, any complete, noncompact and irreducible $(X,\varphi)$ has only one end, meaning that $X - B_r (x)$ has only one connected component for large $r\gg 1$. Our first result gives conditions under which monopoles $(\nabla, \Phi)$ have $|\Phi|$ converging uniformly to a constant along this end. When this is the case, $(\nabla, \Phi)$ is said to have finite mass and the value of the constant to which $|\Phi|$ converges is called the mass.

\begin{Mtheorem}[Finite intermediate energy and bounded curvature implies finite mass]\label{thm:Main_Theorem_1}
	Let $(X, \varphi)$ be a complete, noncompact and irreducible $\rG_2$-manifold of bounded geometry which furthermore is nonparabolic\footnote{See \Cref{sec:finite_mass} for definitions and a discussion on the necessity of this last hypothesis.} (i.e. admits a positive Green's function). Suppose that $(\nabla,\Phi)$ is a solution\footnote{Here and in what follows, we only consider \emph{smooth} solutions.} to the $\rm G_2$-monopole \cref{eq:Monopole} with finite intermediate energy \eqref{eq:Intermediate_Energy} and bounded\footnote{Under these conditions, $|F_\nabla^{14}|\in L^{\infty}(X)$ actually implies that the whole curvature $|F_{\nabla}|\in L^{\infty}(X)$, cf. \Cref{cor:Lpbounds}.} $|F_{\nabla}^{14}|$, i.e. $|F_\nabla^{14}|\in L^{\infty}(X)$. Then $(\nabla,\Phi)$ has finite mass, i.e. there is a constant $m \in [0, \infty)$, called the mass of $(\nabla,\Phi)$, such that for any choice of reference point $o \in X$ one has
	\begin{equation}\label{eq:Finite_Mass}
		\lim_{\dist(x,o) \to \infty} |\Phi (x)| = m.
	\end{equation}
\end{Mtheorem}

Notice that, as previously mentioned, for a $\rm G_2$-monopole $(\nabla,\Phi)$ the function $|\Phi|^2$ is subharmonic. Hence, if $m = 0$ in the previous theorem then we have $\Phi = 0$ everywhere and so $\nabla$ is a $\rG_2$-instanton. Hence, we are primarily interested in the case where $m$ is positive (see \Cref{rem:Finite_Mass}).

Our second main result gives the asymptotic structure of $\rG_2$-monopoles on the so called asymptotically conical (AC) $\rG_2$-manifolds. This is a very interesting class of complete, noncompact and nonparabolic $\rG_2$-manifolds for which explicit examples are known \cites{Bryant1989,Foscolo2018}, and on which $\rG_2$-monopoles have already been constructed \cites{Oliveira2014,Oliveira2016}. A $\rG_2$-manifold $(X, \varphi)$ is AC if its end is asymptotically isometric to a metric cone $(C = (1, \infty)_r \times \Sigma, g_C = \rd r^2 + r^2 g_\Sigma)$, see \Cref{def:AC} for the precise definition. In this case, the cross section of the asymptotic cone $(\Sigma, g_\Sigma)$ comes equipped with a nearly K\"ahler structure $(\omega,J)$ as defined in \Cref{def:NK}. In this situation, a connection $\nabla$ on a principal $\rG$-bundle over $(\Sigma,\omega,J)$ is said to be \emph{pseudo-Hermitian--Yang--Mills connection}, if
\begin{align}
	F_\nabla^{0,2} 		&= 0, \\
	\Lambda F_\nabla	&= 0,
\end{align}
where $F_\nabla^{0,2}$ denotes the $(0,2)$-component of the curvature with respect to the almost complex structure $J$, determined by the nearly K\"ahler structure and $\Lambda F_\nabla$ the contraction of the curvature with the fundamental 2-form $\omega$. In the next theorem, we restrict to the case of $\rG = \SU (2)$.

\begin{Mtheorem}[Asymptotics of $\rG_2$-monopoles on AC manifolds]\label{thm:Main_Theorem_2}
	Let $(X, \varphi)$ be an irreducible AC $\rG_2$-manifold, with radius function $r$, and let $(\nabla, \Phi)$ be a solution to the $\rG_2$-monopole \cref{eq:Monopole} with structure group $\rG = \SU (2)$. Suppose that $(\nabla,\Phi)$ has finite intermediate energy \eqref{eq:Intermediate_Energy} and is such that $|F_\nabla^{14}|$ decays to zero uniformly along the end, i.e. $|F_\nabla^{14}(x)|\to 0$ as $r(x)\to\infty$. Then, along the end of $(X, \varphi)$,
	\begin{equation}
		|\nabla \Phi| \lesssim r^{-(n-1)},
	\end{equation}
	and $|[\Phi, \nabla\Phi]|+ |[\Phi,F_\nabla]|$ decays exponentially.

	Furthermore, if $r^2 |F_\nabla^{14}|$ is bounded, i.e. $|F_\nabla^{14}|$ decays (at least) quadratically, then there is a principal $\rG$-bundle $P_\infty$ over $\Sigma$, together with a pair $(\nabla_\infty, \Phi_\infty)$ such that:
	\begin{enumerate}
		\item[(a)] $\Phi_\infty$ is a $\nabla_\infty$-parallel section of the Adjoint bundle $\mathfrak{g}_{P_\infty}$ over $\Sigma$, and
		\item[(b)] $\nabla_\infty$ is a pseudo-Hermitian--Yang--Mills connection with respect to the nearly K\"ahler structure on $\Sigma$;
	\end{enumerate}
	and 
	\begin{equation}
		(\nabla, \Phi) |_{\lbrace R \rbrace \times \Sigma} \to (\nabla_\infty, \Phi_\infty),
	\end{equation}
	uniformly as $R \to \infty$.
\end{Mtheorem}

\begin{remark}
Some remarks are now in place.
	\begin{itemize}
		\item $\rG_2$-monopoles solve the second order \cref{eq:2nd_Order_Eq_1,eq:2nd_Order_Eq_2} (see \Cref{lem:Second_Order_Eqs}). These are the Euler--Lagrange equations for both the intermediate energy $\mathcal{E}^{\psi}$ and the YMH energy. We also prove analogues of the above main results for general solutions of these equations, see \Cref{thm:Finite_Mass,thm:First_part_of_Main_Theorem_2,thm:main3}.
		
		\item The decay estimate for $|\nabla\Phi|$ given above is sharp as proven in \Cref{rem:Sharp} and as exemplified by the examples in \cite{Oliveira2014} which satisfy all the conditions in \cref{thm:Main_Fredholm_Theorem_2}.
	\end{itemize}
\end{remark}
This article also contains several other interesting results on the asymptotic behavior of $\rG_2$-monopoles. For example, in the conditions of \Cref{thm:Main_Theorem_1}, \Cref{cor:Lpbounds} gives uniform decay of all derivatives of both $F_\nabla$ and $\nabla\Phi$ and \Cref{cor:Asymptotics_Phi_Square} gives a further general refinement on the asymptotics of $m^2-|\Phi|^2$, and in the conditions of \Cref{thm:Main_Theorem_2}, \Cref{cor:refined_AC_asymp_Phi} gives that $m^2-|\Phi|^2\sim r^{2-n}$ as $r\to\infty$ and \Cref{cor:higher_order_der_bounds} gives $\nabla^{j+1}\Phi\in L_1^p$ for all $p\in [2,2n]$ and $j\in\mathbb{N}$. Also, as a consequence of these results we use a $\rG_2$-version of the Bogomolny trick to obtain a topological formula for the intermediate energy of a monopole on an asymptotically conical $\rG_2$-manifold. Given the importance of this result we state it here.

Along the end of an asymptotically conical $\rG_2$-manifold $(X^7,\varphi)$, the cohomology class $[\psi|_{\Sigma_R}]$ obtained by restricting the $4$-form $\psi$ to the links $\Sigma_R \cong \lbrace R \rbrace \times \Sigma$ of the asymptotic cone determine, for $R \gg 1$, a class $\Psi_\infty \in H^4(\Sigma, \mathbb{R})$ called the asymptotic cohomology class, see \Cref{def:AC_Cohomology}.\\
Given a solution $(\nabla,\Phi)$ to the $\rG_2$-monopole \cref{eq:Monopole} with structure group $\rG = \SU (2)$ on an asymptotically conical manifold, having nonzero intermediate energy and bounded $r^2 |F_\nabla^{14}|$, it follows from \Cref{thm:Main_Theorem_2} that $\nabla_\infty$ is reducible, since $\nabla_\infty \Phi_\infty = 0$ and $\Phi_\infty \neq 0$. Then, $\nabla_\infty$ reduces to a connection on a $\rU (1) \subseteq \SU (2)$-bundle. Such a $\rU (1)$-bundle determines a complex line bundle $L$ through the standard representation and its first Chern class $\beta = c_1 (L) \in H^2(\Sigma, \mathbb{Z})$ is called the monopole class of $(\nabla,\Phi)$, see \Cref{def:Monopole_Class}. The energy formula which in the $\rG_2$-setting replaces the Bogomolny trick is the following.

\begin{corollary}[$\rG_2$-analogue of the Bogomolny trick]\label{thm:Bogomolny_Trick}
	Let $(X^7,\varphi)$ be an irreducible asymptotically conical $\rG_2$-manifold, with asymptotic cohomology class $\Psi_\infty \in H^4(\Sigma, \mathbb{R})$. Suppose that $(\nabla,\Phi)$ is a solution to the $\rG_2$-monopole \cref{eq:Monopole} with structure group $\rG = \SU (2)$, finite intermediate energy and bounded $r^2 |F_\nabla^{14}|$, such that it has mass $m$ and monopole class $\beta \in H^2(\Sigma, \mathbb{Z})$. Then
	\begin{equation}
		\cE^{\psi}(\nabla, \Phi) = 4 \pi m \langle \beta \cup \Psi_\infty , [\Sigma] \rangle.
	\end{equation}
\end{corollary}

For a more general formulation of this result which applies to solutions of the second order \cref{eq:2nd_Order_Eq_1,eq:2nd_Order_Eq_2}, see \Cref{thm:Energy_Formula}.

This energy formula can be applied in specific cases to find vanishing theorems for $\rG_2$-monopoles. An example of such an application is given in \Cref{cor:Vanishing} stating that when $\Psi_\infty = 0$---which happens for instance when $H^2 (\Sigma, \mathbb{Z})$ is trivial\footnote{For example the Bryant--Salamon metric on $\mathbb{R}^4 \times S^3$.}---then there are no $\rG_2$-monopoles $(\nabla,\Phi)$ with $|F_\nabla^{14}|$ quadratically decaying and finite nonzero intermediate energy.

\Cref{sec:Linearized_Solutions} is entirely dedicated to proving decay properties of solutions to the linearized equation. It provides the foundations for the moduli theory developed in the subsequent sections. Namely, it establishes that all the $L^2$ solutions of the linearized equation decay at rate which is compatible with the appropriate Sobolev spaces to be used in \Cref{sec:Sobolev_Fredholm,sec:Sobolev_2,sec:Moduli}.

All $\rG_2$-monopoles with finite intermediate energy, bounded $r^2 |F_\nabla^{14}|$, fixed monopole class and mass $m > 0$ determines (modulo gauge) the same same asymptotic pair, $(\nabla_\infty, \Phi_\infty)$, at infinity; see \Cref{rem:Unique_nabla}. Moreover, for any such monopole $(\nabla,\Phi)$ we have that (in the right gauge) $|\Phi - \Phi_\infty|$ and $|r(\nabla-\nabla_\infty)|$ both decay along the conical end. Furthermore, the asymptotic configuration $(\nabla_\infty, \Phi_\infty)$ has a group of automorphisms isomorphic to $\rU (1)$ which we call $\Gamma_\infty$. Based on this, in \Cref{sec:Sobolev_Fredholm,sec:Sobolev_2,sec:Moduli} we develop a moduli theory describing such monopoles. The main result of these sections is stated as \Cref{thm:Moduli} which we restate, informally, here as follows:

\begin{Mtheorem}
	Let $(\nabla,\Phi)$ be a $\rG_2$-monopole with finite intermediate energy and $|F_\nabla^{14}|$ decaying quadratically as before. Then, there are Banach manifolds $\tilde{\mathcal{B}}^p_{1,\alpha}$, $\mathcal{F}^p_{1, \alpha}$ defined in \Cref{sec:Moduli} and a $\Gamma_\infty$-invariant (nonlinear) Fredholm map 
	\begin{equation}
		\mon : \tilde{\mathcal{B}}^p_{1,\alpha} \to \mathcal{F}^p_{1, \alpha},
	\end{equation}
	with the following significance. The moduli space of $\rG_2$-monopoles with finite intermediate energy, $|F^{14}|$ quadratically decaying, and the same monopole class and mass as $(\nabla,\Phi)$ is in bijection with
	\begin{equation}
		\mon^{-1} (0) / \Gamma_\infty \subseteq \mathcal{B}^p_{1,\alpha}.
	\end{equation}
\end{Mtheorem}

\smallskip

\subsection*{Comparison with previous work}

In \cite{Oliveira2014_thesis} the third author worked under much stronger hypothesis in order to deduce similar results to those of \Cref{thm:Main_Theorem_2}. In that reference it is already assumed that: (1) $(\nabla,\Phi)$ has finite mass, i.e. \cref{eq:Finite_Mass} holds; and (2) the connection $\nabla$ is asymptotic to a connection $\nabla_\infty$, pulled back from the link $\Sigma$ of the asymptotic cone, with $|\nabla - \nabla_\infty| \lesssim r^{-1-\epsilon}$ for some $\epsilon>0$. Under these hypothesis, the existence of $\Phi_\infty$ as in (a) of \Cref{thm:Main_Theorem_2} was then deduced. However, the proof of part (b) and of \Cref{thm:Bogomolny_Trick} in \cite{Oliveira2014_thesis} uses the additional hypothesis that (3) $|[ \Phi_\infty , \nabla - \nabla_\infty ]| \lesssim r^{- 6 - \epsilon}$ for some $\epsilon > 0$.\\
The moduli theory developed here is the same as that appearing in \cite{Oliveira2014_thesis} which to date had not yet been published in a journal. The work done in the preceding chapters lays different foundations for the development of this moduli theory than that appearing in \cite{Oliveira2014_thesis}.

\smallskip

\subsection*{Organization}

In \Cref{sec:Pre} we fix some nomenclature and notations, and derive preliminary identities satisfied by $\rG_2$-monopoles. Most notably a Bochner--Weitzenb\"ock formula for $\Delta |\nabla\Phi|^2$. Next, in \Cref{sec:Moser_e-reg}, we derive very useful consequences of the previous identities via Moser iteration and $\varepsilon$-regularity results, under the hypothesis of finite intermediate energy and bounded curvature. These yield that $|\nabla\Phi|^2$ decays, is in $L^p$ for all $p\in [1,\infty]$, and in case $|F_\nabla|$ decays, we get that $|\nabla^j F_\nabla|$ and $|\nabla^{j+1}\Phi|$ decay for all $j\in\mathbb{N}$. 

\Cref{sec:finite_mass} is mainly concerned with a proof of our first main theorem, but in fact proves a considerably stronger result, stated as \Cref{thm:Finite_Mass}, and further partial refinements. The main tools here are the integrability and decay properties of the previous section, and classical results on harmonic function theory of complete manifolds with nonnegative Ricci curvature, including Green's function asymptotics and Yau's gradient inequality, all combined through a strategy inspired by the original work of Taubes in the classical $3$-dimensional monopole equation in \cite{Jaffe1980}*{Chapter IV}. 

In \Cref{sec:BW_formulas} we prove refined Bochner and Weitzenb\"ock type formulas for finite mass monopoles away from the zero set of the Higgs field when the gauge group is $\rG = \SU (2)$. Using decay hypothesis, we get in particular strong Bochner inequalities sufficiently far along the end of our irreducible $\rG_2$-manifold, cf. \Cref{cor:Elliptic_Inequality_For_Exponential_Decay,lem:Improved_Bochner_Delta_Nabla_Phi}. We then restrict to the AC $\rG_2$-manifold case in \Cref{sec:AC}. The first striking consequence of the Bochner inequalities, together with the maximum principle, is the exponential decay of the $\Phi$-transversal components of $F_\nabla$ and $\nabla\Phi$ in this context, proved in \Cref{prop:Exponential_Decay_Transverse}. We then move to use a combination of the Agmon identity, Hardy's inequality and Moser iteration in \Cref{subsec:Hardy1,sec:final_est} to get a sharp polynomial decay rate of $|\nabla\Phi|$, completing the proof of the first part of our second main result, restated as \Cref{thm:First_part_of_Main_Theorem_2}. Then, in \Cref{sec:Boundary_data} we use the previous results, together with Uhlenbeck compactness and related techniques to prove the convergence result of the second part of our second main result. 

As an application of our second main result, in \Cref{sec:Bogomolny_Trick} we develop the $\rm G_2$-analogue of the Bogomolny trick which results in the energy formula stated as \Cref{thm:Energy_Formula}.

We devote \Cref{sec:Linearized_Solutions} to the study of the linearized $\rG_2$-monopole equation and using the same techniques of the previous sections we prove analogous decay results for its solutions.

Finally, in \Cref{sec:Sobolev_Fredholm,sec:Sobolev_2,sec:Moduli} we develop the moduli theory for $\rG_2$-monopoles with finite intermediate energy, quadratically decaying curvature, fixed monopole class and mass. The first of these sections defines the relevant Sobolev norms and proves that the linearized monopole equation is Fredholm. The second settles some useful technical results such as multiplication maps which are needed in order to handle the nonlinearities of the monopole equation. Finally, in the third and last of these sections we topologize the relevant moduli spaces using the Sobolev norms previously defined and prove that the monopole equation yields such a nonlinear Fredholm map. The main result is stated as \Cref{thm:Moduli}.

\smallskip

\begin{acknowledgment}

	The authors are extremely grateful to Mark Stern from whom they learned a lot of the techniques employed in this paper. We also thank Detang Zhou for informing us of the reference \cite{Li1995}. The second named author is thankful to Lorenzo Foscolo and Thomas Walpuski for helpful discussions, and to Universidade Federal Fluminense and IMPA for their hospitality during the early stage of this project.

	The first named author was financed in part by the Coordenação de Aperfeiçoamento de Pessoal de Nível Superior - Brasil (CAPES) - Finance Code 001, via the postdoctoral grant [88887.495008/2020-00] of the INCTMat-CAPES program (Universidade Federal Fluminense), and the postdoctoral grant [88887.643728/2021-00] of the CAPES/COFECUB bilateral collaboration project Ma 898/18 [88887.143014/2017-00] (Universidade Estadual de Campinas/Université de Bretagne Occidentale). He was also supported in part by the School of Mathematical Sciences of Peking University. The third named author was supported by the NOMIS Foundation, and while at Universidade Federal Fluminense by Funda\c{c}\~ao Serrapilheira 1812-27395, by CNPq grants 428959/2018-0 and 307475/2018-2, and FAPERJ through the program Jovem Cientista do Nosso Estado E-26/202.793/2019.

	Finally, the authors are thankful to the anonymous referee for their helpful comments.

\end{acknowledgment}

\bigskip

\section{Preliminaries}\label{sec:Pre}

\subsection{Notation and conventions}

In this article $n = 7$. We prefer to keep the $n$ explicit as this allows us to more easily read the use of several analytic results such as scaling, Moser iteration arguments, Hardy's inequality etc. Keeping $n$ instead of $7$ is also convenient for more easily compared with other monopole theories. 

Throughout the text, unless otherwise stated, we assume that $(X^7,\varphi)$ is a complete, noncompact and irreducible $\rm G_2$-manifold. Moreover, given a principal $G$-bundle $P$ over $X$, we always consider \emph{smooth} configurations $(\nabla,\Phi)$ on $P\to X$. We assume $G$ to be a compact Lie group, and we fix some Ad-invariant metric on the Lie algebra $\mathfrak{g}$ of $G$, which in turn induces a metric on the adjoint bundle $\mathfrak{g}_P$. In particular, when $G=\rm SU(2)$ we fix the metric on $\mathfrak{g}_P$ to be the one induced by the inner product $(a,b)\mapsto -2\tr(ab)$ on $\mathfrak{g}=\mathfrak{su}(2)$.

We let $\Delta := \rd^\ast \rd$ be the \emph{Hodge--Laplacian} operator on functions of $X$, and $\Delta_\nabla := \rd_\nabla \rd_\nabla^\ast + \rd_\nabla^\ast\rd_\nabla$ be the \emph{covariant Hodge--Laplacian}, induced by $\nabla$, acting on $\Omega^k(X,\mathfrak{g}_P)$. We note that $\Delta_\nabla = \rd_\nabla^\ast\rd_\nabla$ and coincides with the \emph{rough Laplacian} $\nabla^\ast\nabla$ on $\Omega^0 (X, \mathfrak{g}_P)$.

For any $\alpha \in \Omega^k (X,\mathfrak{g}_P)$ and $\beta \in \Omega^l (X,\mathfrak{g}_P)$, we define (locally)
\begin{equation}
	[ \alpha \wedge \beta ] := \sum\limits_{\substack{I, J \\ |I| = k, |J| = l}} [\alpha_I, \beta_J] \rd x^I \wedge \rd x^J \in \Omega^{k + l} (X,\mathfrak{g}_P).
\end{equation}
When $k$ or $l$ is zero (one of them is a "scalar"), then drop the wedge from the notation, that is, we write
\begin{equation}
	[ \alpha \wedge \beta ] = [\alpha, \beta].
\end{equation}

We denote by $c>0$ a generic constant and we write $\alpha\lesssim\beta$ to mean that $\alpha\leqslant c\beta$.

\smallskip

\subsection{Bounded geometry and Moser iteration}

We say that $(X,g)$ has \emph{bounded geometry} if its global injectivity radius, $\mathrm{inj}(X,g) = \inf_{x\in X} \mathrm{inj}_x (X, g)$, is positive (in particular this implies completeness), and the Riemann curvature tensor, together with all of its derivatives, is bounded, that is for each $j\in\mathbb{N}$, there is $c_j>0$ such that $|\nabla^j\mathrm{Riem}|\leqslant c_j$. 

We now cite a standard Moser iteration type result in the exact manner we need it in this article.
\begin{proposition}[Moser iteration, cf. \cite{W17}*{Lemma 10}]\label{prop:Moser}
	Let $B_r (x) \subseteq (X^n,g)$ be a convex geodesic ball and $u : B_r (x) \to \mathbb{R}$ be a smooth nonnegative function satisfying $\Delta u \leqslant c_0 u$, for some constant $c_0 \geqslant 0$. Then, there is a constant $c>0$ depending only on the geometry of $B_r (x)$ such that 
	\begin{equation}
	    \sup_{ y \in B_{\frac{r}{2}}(x)} u(y)  \leqslant \ c\left( c_0^{n/2} + r^{-n} \right) \int\limits_{B_r(x)} u \vol_X.
	\end{equation}
\end{proposition}
If $(X,g)$ has bounded geometry then the constant $c$ above can be taken to be universal in a way that it does not depend on $x$. In fact, there is $r_0\in (0,\mathrm{inj}(X,g))$ such that for every $r\in (0, r_0]$, $x\in X$, and any smooth nonnegative function $u:X\to\mathbb{R}$ satisfying $\Delta u \leqslant c_0 u$ on all of $X$, then
\begin{equation}
    \sup_{ y \in B_{\frac{r}{2}}(x)} u(y)  \lesssim \ \left( c_0^{n/2} + r^{-n} \right) \int\limits_{B_r(x)} u \vol_X.
\end{equation}

\smallskip

\subsection{Asymptotically conical $\rG_2$-manifolds}

Now we give some definitions and notations concerning AC $\rG_2$-manifolds.
\begin{definition}\label{def:NK}
	Given a 6-manifold $\Sigma$, a pair of forms $(\omega,\Omega_1)\in\Omega^2\oplus\Omega^3(\Sigma,\mathbb{R})$ determine a $\SU (3)$-\emph{structure} on $\Sigma$ if:
	\begin{itemize}
		\item The $\text{GL}(6,\mathbb{R})$ orbit of $\Omega_1$ is open, with stabilizer a covering of $\text{SL}(3,\mathbb{C})$;
		\item The following compatibility relations hold
		\begin{equation}
			\omega\wedge\Omega_1 = \omega\wedge\Omega_2 = 0,\quad\frac{\omega^3}{3!} = \frac{1}{4}\Omega_1\wedge\Omega_2,
		\end{equation} where $\Omega_2 = J\Omega_1$ and $J$ denotes the almost complex structure determined by $\Omega_1$.
		\item $g_\Sigma = \omega(\cdot{},J\cdot{})$ determines a Riemannian metric on $\Sigma$.
	\end{itemize} We let $\Omega$ be the complex volume form on $(\Sigma, g_\Sigma)$ such that $\text{Re}(\Omega) = \Omega_1$ and $\text{Im}(\Omega) = \Omega_2$. Furthermore, if the forms $(\omega,\Omega)$ satisfy
	\begin{equation}
		\rd\Omega_2 = -2\omega^2\quad\text{and}\quad\rd\omega = 3\Omega_1,
	\end{equation}
	then $(\Sigma,g_\Sigma)$ is said to be \emph{nearly K\"ahler}. 
\end{definition}

\begin{lemma}\label{lem:cone_G2 = linkNK}
	Suppose that $\Sigma$ is endowed with an $\SU (3)$-structure determined by $(\omega,\Omega_1)$. Then the Riemannian cone $(C(\Sigma) = (1,\infty)_r\times\Sigma, g_C = \rd r^2+ r^2g_\Sigma)$ with the $\rG_2$-structure 
	\begin{equation}
		\varphi_C = r^2 \rd r\wedge\omega + r^3\Omega_1,\quad\psi_C = r^4\frac{\omega^2}{2} - r^3 \rd r\wedge\Omega_2,
	\end{equation}
	is a $\rG_2$-manifold if and only if $(\Sigma^6,g_\Sigma)$ is nearly K\"ahler.
\end{lemma}

\begin{definition}\label{def:AC}
	We say that a noncompact, complete, $\rm G_2$-manifold $(X^7,\varphi)$ is \emph{asymptotically conical (AC) with rate $\nu<0$} when there exists a compact subset $K\subseteq X$, a closed nearly K\"ahler 6-manifold $(\Sigma, g_\Sigma)$ and a diffeomorphism $\Upsilon:C(\Sigma)\to X - K$ such that the cone metric $g_C$ on $C(\Sigma)$ and its Levi-Civita connection $\nabla_C$ satisfy:
	\begin{equation}
		\left|\nabla_C^j(\Upsilon^\ast g - g_C)\right|_{g_C} = O(r^{\nu - j})\quad\text{as $r\to\infty$,\quad for all $j\in\mathbb{N}$.}
	\end{equation} The connected components of $X - K$ are called the \emph{ends} of $X$ and $\Sigma$ is called the \emph{link} of the asymptotic cone. By a slight abuse of notation we let $r$ be any positive smooth extension of $r\circ\Upsilon^{-1}|_{X - K}$ to $X$ and call $r$ a \emph{radius function}. For each $R>0$, we let $B_R = \{ x \in X : r (x) \leqslant R \}$, which, for large enough $R$, is a smooth manifold-with-boundary, with a fixed diffeomorphism type. We also let $\Sigma_R = \partial B_R$, which is a closed Riemannian 6-manifold.
\end{definition}
\begin{remark}\label{rem:AC_vol_growth}
    Notice that any AC $\rm G_2$-manifold has bounded geometry. Moreover, they have maximal (Euclidean) volume growth, i.e. $\text{Vol}(B_r(x))\gtrsim r^n$ (see \cite{van2009regularity}*{Corollary 2.18}). Here the word ``maximal'' is used because Ricci-flatness (implied by $\rG_2$-holonomy) together with Bishop's absolute volume comparison theorem gives $\text{Vol}(B_r(x))\lesssim r^n$. In particular, AC $\rG_2$-manifolds are nonparabolic; indeed, they satisfy \cref{ineq:vol_growth_nonparabolic} (see \Cref{sec:finite_mass} for more details on nonparabolicity).
\end{remark}

Given an asymptotically conical $\rG_2$-manifold $(X^7,\varphi)$ as in \Cref{def:AC}, it has the property that along the conical end $|\Upsilon^*\psi-\psi_C|_{g_C} = O(r^{\nu})$ with derivatives. As a consequence, along the conical end there is a $4$-form $\eta$ with $|\eta| = O(r^{\nu})$ such that $\psi = (\Upsilon^{-1})^*\psi_C+\eta$. Furthermore, as $\psi_C = -\frac{1}{4}\rd (r^{4}\Omega_2)$ we find that the cohomology class in $H^4(\Sigma, \mathbb{R})$ determined by $\eta|_{\Sigma_R}$ and $\psi|_{\Sigma_R}$ agree, i.e.
\begin{equation}
    [\eta|_{\Sigma_R}] = [\psi|_{\Sigma_R}].
\end{equation}
By construction and the homotopy invariance, $[\psi|_{\Sigma_R}]$ is constant for sufficiently large $R$ and for convenience we now name the class in $H^4(\Sigma, \mathbb{R})$ which it represents.

\begin{definition}\label{def:AC_Cohomology}
	In case of an AC $\rG_2$-manifold, as in \Cref{def:AC}, a class $\Psi_\infty \in H^4(\Sigma, \mathbb{R})$ is said to be an \emph{asymptotic cohomology class} if 
	\begin{equation}
	    \Psi_\infty \coloneqq [\psi|_{\Sigma_R}],
	\end{equation}
	for all sufficiently large $R$.
\end{definition}

\smallskip

\subsection{A Bochner--Weitzenb\"ock formula}

Here we derive some basic but crucial equations satisfied by $\rG_2$-monopoles.

\begin{lemma}\label{lem:Second_Order_Eqs}
	Let $(\nabla, \Phi)$ be any solution of the $\rG_2$-monopole \cref{eq:Monopole} on $P \to X$. Then the pair $(\nabla, \Phi)$ satisfies
	\begin{subequations}
	\begin{align}
		\Delta_\nabla \Phi 		&= 0, \label{eq:2nd_Order_Eq_1} \\
		\rd_\nabla^* F_\nabla	&= [\nabla \Phi, \Phi]. \label{eq:2nd_Order_Eq_2}
	\end{align}
	\end{subequations}
	In particular, $\Delta_\nabla F_\nabla = [[F_\nabla , \Phi], \Phi] - [\nabla \Phi \wedge \nabla \Phi]$.
\end{lemma}
\begin{proof}
	The first equation, $\Delta_\nabla \Phi = 0$ is immediate from applying $\rd_\nabla^*$ to the $\rG_2$-monopole \cref{eq:Monopole} and using the Bianchi identity $\rd_\nabla F_\nabla = 0$ together with $\rd \psi = 0$. As for the second equation, we first use the fact that $3 F_\nabla^7 = \ast ( \ast(F_\nabla \wedge \psi) \wedge \psi)$ to compute
	\begin{equation}
		3 \rd_\nabla^*F_\nabla^7 = \ast \rd_\nabla \ast^2 (\nabla \Phi \wedge \psi) = \ast \left( [ F_\nabla , \Phi] \wedge \psi \right) = [\nabla\Phi,\Phi]. \label{eq:YMH_rewrite}
	\end{equation}
	Notice that $3F_\nabla^7 = F_\nabla + \ast (F_\nabla \wedge \varphi)$ and $3F_\nabla^{14} = 2F_\nabla - \ast (F_\nabla \wedge \varphi)$. Thus, using the fact that $\varphi$ is closed we find
	\begin{equation}
		\rd_\nabla^*F_\nabla = 3 \rd_\nabla^* F_\nabla^7 = \frac{3}{2} \rd_\nabla^*F_\nabla^{14}.
	\end{equation}
	The result follows from inserting this into the equation above.
\end{proof}

\begin{lemma}\label{lem:Bochner_Delta_nabla_Phi_0}
	For any solution $(\nabla, \Phi)$ of the second order \cref{eq:2nd_Order_Eq_1,eq:2nd_Order_Eq_2}, we have
	\begin{equation}\label{eq:Bochner_identity_nablaPhi}
	    \nabla^*\nabla (\nabla \Phi) = [[\nabla \Phi, \Phi], \Phi] - 2\ast [\ast F_\nabla \wedge \nabla \Phi].
	\end{equation}
	In particular,
	\begin{align}
		\frac{1}{2}\Delta \ |\nabla \Phi|^2 + |\nabla^2 \Phi|^2	&= \langle \nabla \Phi, \nabla^*\nabla (\nabla \Phi) \rangle \\
		&= - 2 \langle \nabla \Phi, \ast [\ast F_\nabla \wedge \nabla \Phi] \rangle - |[\Phi, \nabla \Phi]|^2, \label{eq:Weitzenbock_Intermediate}
	\end{align}
	which implies
	\begin{equation}\label[ineq]{ineq:Bochner_ineq_nablaPhi}
		\frac{1}{2}\Delta |\nabla \Phi|^2 + |\nabla^2 \Phi|^2 + |[\Phi, \nabla \Phi]|^2 \lesssim |F_\nabla| |\nabla \Phi|^2. 
	\end{equation}
\end{lemma}

\begin{proof}
	Using the Ricci-flatness and the Bochner--Weitzenb\"ock formula, we have
	\begin{equation}\label{eq:BW_1}
    	\nabla^* \nabla (\nabla \Phi) = \Delta_\nabla \nabla \Phi - \ast [\ast F_\nabla \wedge \nabla \Phi].
	\end{equation}
	Now, using the second order \cref{eq:2nd_Order_Eq_1,eq:2nd_Order_Eq_2} and the Bianchi identity we compute 
	\begin{align}
		\Delta_\nabla \nabla \Phi &= \rd_\nabla^* [F_\nabla , \Phi] \\
		&= [\rd_\nabla^*F_\nabla , \Phi] - \ast [ \ast F_\nabla \wedge \nabla \Phi]\\
		&= [[\nabla \Phi, \Phi], \Phi] - \ast [\ast F_\nabla \wedge \nabla \Phi].\label{eq:BW_2}
	\end{align}
	Putting \cref{eq:BW_1,eq:BW_2} together implies \cref{eq:Bochner_identity_nablaPhi}.
\end{proof}

\smallskip

\subsection{Finite mass configurations}
To finish this preliminary section, we introduce the precise definition of finite mass configurations and make a simple but useful remark. Recall that since $(X^7,\varphi)$ is a complete, noncompact and irreducible $\rG_2$-manifold, it follows from the Cheeger--Gromoll splitting theorem that $(X^7,\varphi)$ has only one end.
\begin{definition}
	A configuration $(\nabla, \Phi)$ is said to have \emph{finite mass} if $\left|\Phi\right|$ converges uniformly to a constant $m\in\mathbb{R}_+ $ along the end, i.e. for any choice of reference point $o\in X$ one has
	\begin{equation}
		\lim_{\dist(x,o)\to\infty}|\Phi(x)| = m.
	\end{equation}
	Then the constant $m$ is called the \emph{mass} of $(\nabla, \Phi)$.
\end{definition}

\begin{remark}\label{rem:Finite_Mass}
If $(\nabla, \Phi)$ is a solution to the second order \cref{eq:2nd_Order_Eq_1,eq:2nd_Order_Eq_2} then, in particular, $\Delta_\nabla\Phi = \nabla^*\nabla\Phi = 0$ and this implies that
\begin{align}
	\frac{1}{2}\Delta \left|\Phi\right|^2 &= \frac{1}{2} \rd^\ast \rd |\Phi|^2\\
	&= \rd^\ast \langle\nabla\Phi,\Phi\rangle\\
	&= -\ast \rd \langle\ast\nabla\Phi,\Phi\rangle\\
	&= \langle \nabla^*\nabla \Phi, \Phi\rangle - \left|\nabla\Phi\right|^2\\
	&= - \left|\nabla\Phi\right|^2 \leqslant 0. \label{eq:subharmonic}
\end{align}
Thus, the function $\left|\Phi\right|^2$ is subharmonic. When $(\nabla, \Phi)$ also has finite mass $m \in \mathbb{R}_+$, then by the maximum principle (cf. \cite{Jaffe1980}*{Chapter~VI, Proposition~3.3}) one has either $|\Phi|\equiv m$ or $|\Phi|<m$ everywhere on $X$. Moreover, by the uniform convergence $|\Phi|\to m$ along the end, one has that $|\Phi|\geqslant\frac{m}{2}$ outside a sufficiently large geodesic ball. 
\end{remark}

\section{Consequences of Moser iteration and $\epsilon$-regularity}\label{sec:Moser_e-reg}

In this section we deduce step by step the consequences that can be taken from the use of Moser iteration and $\epsilon$-regularity along the end of $X$. The final result of the section which concentrates our conclusions and follows from the preceding work is \Cref{cor:Lpbounds}.

We start with a simple consequence of \Cref{lem:Bochner_Delta_nabla_Phi_0} using Moser iteration.

\begin{lemma}\label{lem:bounded_curvature_moser}
	Let $(\nabla, \Phi)$ be a solution to the second order \cref{eq:2nd_Order_Eq_1,eq:2nd_Order_Eq_2}. Then for any $x\in X$ and $0<r<\frac{1}{2}\mathrm{inj}_x(X,g)$,
	\begin{equation}\label[ineq]{ineq:moser}
		\sup_{B_{\frac{r}{2}}(x)} |\nabla\Phi|^2 \lesssim \left(\|F_\nabla\|_{L^\infty (B_r(x))}^{n/2} + r^{-n}\right) \int\limits_{B_r(x)} |\nabla\Phi|^2 \vol_X.
	\end{equation}
\end{lemma}
\begin{proof}
	This follows from a direct application of the \cref{ineq:Bochner_ineq_nablaPhi} in \Cref{lem:Bochner_Delta_nabla_Phi_0} with the Moser iteration result stated in \Cref{prop:Moser}.
\end{proof}

\begin{corollary}\label{cor:YMH_Lpbounds}
Let $(X, \varphi)$ be a complete, noncompact and irreducible $\rG_2$-manifold of bounded geometry. Let $(\nabla, \Phi)$ be a solution to the second order \cref{eq:2nd_Order_Eq_1,eq:2nd_Order_Eq_2}. If $|F_\nabla|\in L^\infty (X)$ and $|\nabla\Phi|^2\in L^1 (X)$, then $|\nabla\Phi|^2\in L^\infty (X)\cap L^p (X)$ for all $p\in [1,\infty)$ and decays uniformly to zero along the end. 
\end{corollary}
\begin{proof}
	Since $(X,\varphi)$ has bounded geometry, there is $r_0\in (0,\mathrm{inj}(X,g))$ such that the \cref{ineq:moser} of \Cref{lem:bounded_curvature_moser} holds for all $x\in X$ and $r\in (0,r_0]$. Given that $|F_\nabla|\in L^\infty (X)$ and $|\nabla\Phi|^2\in L^1 (X)$ we find that 
  	\begin{equation}
  		\| \nabla \Phi \|_{L^\infty (X)}^2 \lesssim \left(\|F_\nabla\|_{L^\infty (X)}^{n/2} + r_0^{-n}\right) \int\limits_X |\nabla\Phi|^2 \vol_X <\infty,
  	\end{equation}
 	hence $|\nabla\Phi|^2\in L^\infty (X)\cap L^1 (X)\subseteq L^p (X)$ for all $p\geqslant 1$. Moreover, since $|\nabla\Phi|^2\in L^1 (X)$, if $x_i\to\infty$ then $\int_{B_{r_0}(x_i)} |\nabla\Phi|^2 \vol_X \to 0$ and thus by \cref{ineq:moser} one gets $|\nabla\Phi|^2(x_i)\to 0$. This shows that $|\nabla\Phi|$ decays, completing the proof.
\end{proof}

\begin{remark}
	In fact, it is possible to use \cref{ineq:moser} to obtain a (possibly rude) quantification of the $|\nabla \Phi|$ decay, under the hypotheses of \Cref{cor:YMH_Lpbounds}. For this, fix $y \in X$ and take a sequence of points $\lbrace x_i \rbrace_{i \in \mathbb{N}}$ placed along a geodesic ray emanating from $y$ with $\dist(x_i,x_{i+1}) = r$, for some fixed $r\in (0,r_0]$. Then, $\dist(x_i , y) = ir \to  \infty$ and summing inequality \cref{ineq:moser} centered at all points $x_i$ we find
	\begin{equation}
	    \sum\limits_{i = 1}^\infty \sup_{B_{\frac{r}{2}}(x_i)} |\nabla\Phi|^2 \lesssim \left(\|F_\nabla\|_{L^\infty (X)}^{n/2} + r^{-n}\right) \sum\limits_{i = 1}^\infty \int\limits_{B_r(x_i)}|\nabla\Phi|^2\vol_X \lesssim \left(\|F_\nabla\|_{L^\infty (X)}^{n/2} + r^{-n}\right) \int\limits_{X}|\nabla\Phi|^2 \vol_X.
	\end{equation}
	Hence, if $|\nabla \Phi|^2 \in L^1 (X)$ then the series in the left hand side must converge and so 
	\begin{equation}
	    \lim_{i \to \infty} \  \left( \dist(x_i , y) \sup_{B_{\frac{r}{2}}(x_i)} |\nabla\Phi|^2 \right) = 0 .
	\end{equation}
\end{remark}

\begin{definition}
	Let $U\subseteq X$ be an open subset. When finite, we define the {\em energy} and the {\em intermediate energy} of a field configuration $(\nabla, \Phi)$ by the integrals over $U$ of
	\begin{subequations}
	\begin{align}
		e		&= \frac{1}{2}|F_\nabla|^2 + \frac{1}{2} |\nabla \Phi|^2, \label{eq:Energy} \\
		e_\psi	&= \frac{1}{2}|F_\nabla \wedge \psi|^2 + \frac{1}{2} |\nabla \Phi|^2, \label{eq:Psi_Energy} 
	\end{align}
	\end{subequations}
	to which we refer as the {\em energy density} and {\em intermediate energy density} respectively.
\end{definition}

Notice that in case the pair $(\nabla, \Phi)$ is a $\rG_2$-monopole we have $e_\psi = |\nabla \Phi|^2$ and so the intermediate energy is simply the squared $L^2 (X)$-norm of $\nabla \Phi$. In general, it follows from linear algebra that
\begin{equation}\label{eq:linear_algebra_identity}
	|F_\nabla\wedge\psi|^2 = 3|F_\nabla^7|^2.
\end{equation}

We now cite the following $\epsilon$-regularity result for the energy density $e$.
\begin{proposition}[$\epsilon$-regularity; cf. \cite{afuni2019regularity}*{Theorem~B} and \cite{Uhlenbeck1982a}*{Theorem~1.3}]\label{prop:total_epsilon_regularity}
	Let $(X^n,g)$ be a complete oriented Riemannian $n$-manifold of bounded geometry, and let $P$ be a $\rG$-bundle over $X$ where $\rG$ is a compact Lie group. Then there are constants $\epsilon_0 = \epsilon_0(X,g,\mathfrak{g})>0$ and $r_0 = r_0(X,g)\in (0,\mathrm{inj}(X,g))$ with the following significance. Let $(\nabla, \Phi)$ be a solution to the second order \cref{eq:2nd_Order_Eq_1,eq:2nd_Order_Eq_2} on $P \to X$. If $x\in X$ and $0<r \leqslant r_0$ are such that 
	\begin{equation}
		r^{-(n-4)} \int\limits_{B_r(x)} e \: \vol_X < \epsilon_0,
	\end{equation}
 	then
	\begin{equation}\label[ineq]{ineq:estimates_coulomb}
		\sup_{B_{\frac{r}{2}}(x)}\left(\left|\nabla^j F_\nabla\right|^2 + \left|\nabla^{j+1} \Phi\right|^2\right) \lesssim_j r^{-n-2j} \int\limits_{B_r(x)} e \: \vol_X,\quad\forall j\in\mathbb{N}.
	\end{equation}
\end{proposition}

\begin{proof}[Sketch of proof]
  The $C^0$-bound from the $j = 0$ case of \cref{ineq:estimates_coulomb} is a particular case of \cite{afuni2019regularity}*{Theorem~B}. From this bound, for any fixed $p>n/2$, by possibly taking smaller $r_0$ and $\epsilon_0$ one can make $\|F_\nabla\|_{L^p (B_{\frac{r}{2}}(x))}$ to be smaller than Uhlenbeck's constant given by \cite{Uhlenbeck1982a}*{Theorem~1.3}. Thus we can find a Coulomb gauge over $B_r(x)$ in which the second order \cref{eq:2nd_Order_Eq_1,eq:2nd_Order_Eq_2} become an elliptic system and standard elliptic estimates apply, implying the \cref{ineq:estimates_coulomb} for all $j\in\mathbb{N}$.
\end{proof}

As a consequence of \Cref{prop:total_epsilon_regularity,cor:YMH_Lpbounds}, we get:

\begin{corollary}\label{cor:allderivativesdecay}
	Let $(X, \varphi)$ be a complete, noncompact and irreducible $\rG_2$-manifold of bounded geometry. Let $(\nabla, \Phi)$ be a solution to the second order \cref{eq:2nd_Order_Eq_1,eq:2nd_Order_Eq_2}. Suppose that $|F_\nabla|$ decays uniformly to zero along the end and $|\nabla\Phi|^2\in L^1 (X)$. Then one actually has that $|\nabla^j F_\nabla|$ and $|\nabla^{j+1}\Phi|$ decay uniformly to zero along the end for all $j\in\mathbb{N}$.
\end{corollary}
\begin{proof}
	By \Cref{cor:YMH_Lpbounds} and the decay hypothesis on the curvature we know that $e$ decays uniformly to zero at infinity. Therefore, if $(x_i)$ is a sequence escaping to infinity then $\int_{B_{r_0}(x_i)} e \: \vol_X \to 0$, so that by \cref{ineq:estimates_coulomb} one has $|\nabla^j F_\nabla|(x_i),|\nabla^{j+1} \Phi|(x_i)\to 0$.
\end{proof}

Now we turn to the particular case of $\rG_2$-monopoles. We start with an $\epsilon$-regularity result for $e_\psi$.

\begin{proposition}[$\epsilon$-regularity for $e_\psi$]\label{proposition:epsreg}
	Let $(X^7,\varphi)$ be a complete $\rG_2$-manifold of bounded geometry and $P$ a principal $\rG$-bundle over $X$, where $\rG$ is a compact Lie group. Then there are constants $\epsilon = \epsilon (X, \varphi, \mathfrak{g})>0$ and $r_0 = r_0 (X, \varphi) \in (0, \mathrm{inj} (X, g_\varphi))$ with the following significance. Let $(\nabla, \Phi)$ satisfy the $\rG_2$-monopole \cref{eq:Monopole}. If $x\in X$ and $0 < r \leqslant r_0$ are such that
	\begin{equation}
		r^{-(n-4)}\int\limits_{B_r (x)} e_\psi \vol_X < \epsilon,	
	\end{equation} then
		\begin{equation}\label[ineq]{ineq:epsreg}
			\sup_{B_{\frac{r}{2}}(x)} e_\psi \lesssim \left(\|F_\nabla^{14}\|_{L^\infty (B_r(x))}^{n/2}+r^{-4}\right) \ \left(r^{-(n-4)}\int\limits_{B_r(x)} e_\psi \vol_X + r^4\|F_{\nabla}^{14}\|_{L^\infty (B_r(x))}^2\right).
		\end{equation}	
\end{proposition}
\begin{proof}
	First, we note that since $(\nabla,\Phi)$ is a $\rG_2$-monopole, we have 
	\begin{equation}
		e_{\psi} = |\nabla\Phi|^2 = |F_\nabla\wedge\psi|^2 = 3|F_\nabla^7|^2,
	\end{equation}
	where the last equality is valid in general and follows from linear algebra. In particular, it follows from \Cref{lem:Bochner_Delta_nabla_Phi_0} that
	\begin{equation}
		\Delta e_\psi \lesssim \|F_\nabla^{14}\|_{L^\infty (B_r(x))} e_\psi + e_\psi^{3/2} \quad \text{on }B_r(x).
	\end{equation}
	Next, using a well-known almost monotonicity property for the normalized energy in dimensions greater than four, cf. \cite{afuni2019regularity}*{Theorem~2.1}, we have
	\begin{equation}\label[ineq]{ineq:monotonicity}
		s^{-(n-4)} \int\limits_{B_s(x)} e_\psi \vol_X \lesssim r^{-(n-4)} \int\limits_{B_r(x)} e_\psi \vol_X + r^4 \|F_{\nabla}^{14}\|_{L^\infty (B_r(x))}^2,\quad\text{for all } s \in (0,r].
	\end{equation}
	With these observations in mind, the result follows by a standard nonlinear mean value inequality for the Laplacian, which in turn is a consequence of Moser iteration via the so-called ``Heinz trick''; e.g. apply \cite{fadel2020behavior}*{Theorem~A.3} with the parameters $d = 4$, $\tau(r) = r^4\|F_{\nabla}^{14}\|_{L^\infty (B_r(x))}^2$, $a\lesssim 1$, $a_0 = 0$ and $a_1 = \|F_\nabla^{14}\|_{L^\infty (B_r(x))}$ (see also \cite{FO19}*{Theorem~5.1}).
\end{proof}

Using the same reasoning that allowed us to deduce \Cref{cor:allderivativesdecay} from \Cref{prop:total_epsilon_regularity}, we obtain the next result from \Cref{proposition:epsreg,cor:allderivativesdecay}.

\begin{corollary}\label{cor:Lpbounds}
	Let $(X,\varphi)$ be a complete, noncompact and irreducible $\rG_2$-manifold of bounded geometry. Suppose $\left( A, \Phi \right)$ is a solution to the $\rG_2$-monopole \cref{eq:Monopole} with finite intermediate energy \eqref{eq:Intermediate_Energy} (i.e. $|\nabla\Phi|^2\in L^1 (X)$) and such that $|F_\nabla^{14}|\in L^\infty (X)$. Then $|F_\nabla|\in L^{\infty}(X)$ and the function 
	$$e_\psi = |\nabla\Phi|^2\in L^\infty (X) \cap L^p (X) $$
	for all $p \in [1, \infty)$ and decays uniformly to zero at infinity. If furthermore $|F_\nabla^{14}|$ decays uniformly to zero at infinity, then $|\nabla^j F_\nabla|$ and $|\nabla^{j+1}\Phi|$ decay uniformly to zero at infinity for all $j\in\mathbb{N}$.
\end{corollary}

\begin{proof}
	Given that $e_\psi \in L^1 (X)$ and $|F_\nabla^{14}|\in L^\infty (X)$, we can use \Cref{proposition:epsreg} to conclude that $e_\psi\in L^\infty (X)$ or, equivalently, that $|F_\nabla^7|\in L^\infty (X)$ and therefore $|F_\nabla|\in L^\infty (X)$. Thus, by \Cref{cor:YMH_Lpbounds}, we get the first part of the desired result. For the second part, note that we already have that $|F_\nabla^7|$ decays, so if $|F_\nabla^{14}|$ decays then $|F_\nabla|$ decays. Hence, \Cref{cor:allderivativesdecay} applies. 
\end{proof}

\bigskip

\section{Finite mass from finite intermediate energy}\label{sec:finite_mass}

This section contains the proof of our first main result \Cref{thm:Main_Theorem_1}. In fact, we prove a more refined version of that result stated as \Cref{thm:Finite_Mass}; see also \Cref{cor:Asymptotics_Phi_Square}.\\

Let $(X^7,\varphi)$ be a complete, noncompact and irreducible $\rG_2$-manifold and let $P\to X$ be a principal $G$-bundle, where $G$ is a compact Lie group. Recall from \Cref{rem:Finite_Mass} that if $(\nabla,\Phi)$ is a finite mass solution to the second order \cref{eq:2nd_Order_Eq_1,eq:2nd_Order_Eq_2} then $|\Phi|^2$ is a bounded subharmonic function on $X$. Moreover, if $|\Phi|$ is constant then $|\nabla\Phi|^2=-\frac{1}{2}\Delta|\Phi|^2 = 0$. Now, since we are interested in \emph{irreducible} solutions $(\nabla,\Phi)$, meaning those for which $\nabla\Phi\neq 0$, it follows that the existence of such (if any) forces $(X^7,g_{\varphi})$ to support a nonconstant, upper bounded subharmonic function. It turns out that this last condition is equivalent to the Riemannian manifold $(X^7,g_{\varphi})$ to be \emph{nonparabolic}, i.e. to support a positive Green's function (see e.g. \cite{grigor1999analytic}*{Theorem 5.1 (3)}).

Recall that a \emph{Green's function} (for the scalar Laplace operator $\Delta=\rd^\ast \rd$) on $X$ is a smooth function $G(x,y)$ defined on $X\times X\setminus\{(x,x):x\in X\}$ which is symmetric in the two variables $x$ and $y$ and satisfies the following properties: for all $f\in C_c^{\infty}(X)$ and $x\in X$ we have
\begin{equation}
	\Delta_x\int_X G(x,y)f(y)\text{vol}_X(y) = f(x) \quad \text{and} \quad \int_X G(x,y)\Delta f(y)\text{vol}_X(y) = f(x).
\end{equation}
The existence of a Green's function on a complete Riemmanian manifold was first proved by Malgrange \cite{malgrange1956existence} via a nonconstructive argument, and later Li--Tam \cite{li1987symmetric} gave a constructive proof (by a compact exhaustion method) which is particularly important to understand the behavior of such a function. Some manifolds do not admit a positive Green's function (e.g. $\mathbb{R}^2$), while others may admit them (e.g. $\mathbb{R}^m$ for $m\geqslant 3$). Thus, such an existence property distinguishes the function theory of complete noncompact manifolds. 

When the manifold $X^n$ is nonparabolic (i.e. admits a positive Green's function), Li--Tam's construction produces precisely the unique minimal positive Green's function $G(x,y)>0$ and one can read off the following properties of $G(x,y)$: for all $x\neq y$ in $X$,
\begin{itemize}
	\item[(i)] $G(x,y)\sim \dist(x,y)^{2-n}$ as $\dist(x,y)\to 0$ (i.e. the singularity along the diagonal is of the same order as that in $\mathbb{R}^n$);
	\item[(ii)] $G(x,\cdot{})$ is harmonic away from $x$ (in fact, $G(x,y)$ is superharmonic on $X$ if we allow $+\infty$ as a value of the function on the diagonal);
	\item[(iii)] $\sup_{X\setminus B_r(x)} G(x,\cdot{})=\sup_{y\in\partial B_r(x)} G(x,y)<\infty$ for all $r>0$.
\end{itemize} In particular, from (i) and (iii) we have that $G(x,\cdot{})\in L_{\text{loc}}^q(X)$ for any $q<\frac{n}{n-2}$.  

Now, it was proved by Varopoulos \cite{varopoulos1981poisson} that a complete Riemannian manifold $X$ with nonnegative Ricci curvature is nonparabolic if and only if for some (therefore all) $x\in X$ one has
\begin{equation}\label{ineq:vol_growth_nonparabolic}
\int_{1}^{\infty}\frac{t}{\text{Vol}(B_t(x))}dt<\infty.
\end{equation} In fact, if this is the case, Li--Yau \cite{li1986parabolic} proved that, for all $x\neq y$ in $X$, the minimal positive Green's function $G(x,y)$ on $X$ must satisfy
\begin{equation}\label{eq:Green_function_bounds}
	C^{-1}\int_{\dist(x,y)}^{\infty}\frac{t}{\text{Vol}(B_t(x))}dt \leqslant G(x,y) \leqslant C\int_{\dist(x,y)}^{\infty}\frac{t}{\text{Vol}(B_t(x))}dt,
\end{equation}
for some constant $C=C(n)>0$ depending only on the dimension of $X$. In particular, in this case we have
\begin{equation}
	G(x,y)\to 0\quad\text{as}\quad \dist(x,y)\to\infty.
\end{equation}
We also note that the difference between another positive Green's function and $G(x,y)$ must be a positive harmonic function. Now, it is a consequence of Yau's work in \cite{yau1975harmonic} that a complete manifold with nonnegative Ricci curvature does not admit any nonconstant nonnegative harmonic functions. Thus, $G(x,y)$ must be the unique positive Green's function up to an additive constant. 

After recalling the above facts, we are now in position to settle the main consequence of all these preliminary observations. Namely, we find a large class of manifolds for which solutions of \cref{eq:2nd_Order_Eq_1,eq:2nd_Order_Eq_2} with $\nabla\Phi\in L^2$ satisfying mild bounded curvature assumptions have finite mass. We state here the detailed version of our main \Cref{thm:Main_Theorem_1}. Its proof is inspired by Taubes' original work on the standard 3-dimensional Bogomolny equation in \cite{Jaffe1980}*{Theorem~IV.10.3}.

\begin{theorem}\label{thm:Finite_Mass}
	Let $(X^7,\varphi)$ be a complete, noncompact and irreducible $\rG_2$-manifold of bounded geometry, which furthermore is nonparabolic. Let $G(x,y)$ be the minimal positive Green's function of the scalar Laplacian on $(X,g_\varphi)$. Let $(\nabla, \Phi)$ be a solution to the second order \cref{eq:2nd_Order_Eq_1,eq:2nd_Order_Eq_2} with $|\nabla\Phi|^2\in L^1 (X)$. Suppose either that $|F_\nabla| \in L^\infty (X)$ or that $(\nabla, \Phi)$ is a $\rG_2$-monopole such that $|F_\nabla^{14}|\in L^\infty (X)$. Finally, let 
	\begin{equation}\label{eq:Green}
		w (x) = 2\int\limits_X G (x, \cdot)|\nabla\Phi|^2\vol_X. \label{eq:w_def}
	\end{equation} Then the function $w : X \to \mathbb{R}_+ $ defined by \cref{eq:Green} is the unique (smooth) solution to $\Delta w = 2|\nabla\Phi|^2$ which decays uniformly to zero at infinity, and there is a constant $m \geqslant 0$ such that 
	\begin{equation}
		w = m^2 - |\Phi|^2.
	\end{equation}
	In particular, $(\nabla, \Phi)$ has finite mass $m$.
\end{theorem}

\begin{proof}
	Since $(X^7,g_{\varphi})$ is a complete nonparabolic manifold and $|\nabla\Phi|^2$ is a smooth, nonnegative and integrable function, it follows from the result in \cite{ni2002poisson}*{(proof of) Lemma 2.3}, together with elliptic regularity, that $w$ as defined by \cref{eq:Green} is a smooth solution to the Poisson equation $\Delta w = 2|\nabla\Phi|^2$. 
	
	 We now show that $w$ decays uniformly to zero at infinity, and therefore is the unique such solution\footnote{If $\tilde{w}$ is another smooth solution to $\Delta\tilde{w} = 2|\nabla\Phi|^2$ decaying at infinity, then both $\pm(w-\tilde{w})$ are harmonic functions which decay at infinity, so that by the maximum principle we must have $w=\tilde{w}$.}. Fix any reference point $o\in X$. Let $R,\rho>0$ and suppose that $x\in X\setminus B_{R+\rho}(o)$. Note that, by triangle's inequality, for any $y\in B_{\rho}(x)$ we have $\dist(o,y)>R$. Thus, setting $q:=\frac{n-1}{n-2}$, separating the region of integration into two and using H\"older's inequality we have
	\begin{align}
		0\leqslant \frac{1}{2}w(x) &= \left(\int_{B_{\rho}(x)}+\int_{X\setminus B_{\rho}(x)}\right) G(x,y)|\nabla\Phi(y)|^2\text{vol}_X(y)\\
		&\leqslant \|G(x,\cdot{})\|_{L^{q}(B_{\rho}(x))}\|\nabla\Phi\|_{L^{2(n-1)}(X\setminus B_R(o))}^2 + \sup_{y\in X\setminus B_{\rho}(x)} G(x,y)\|\nabla\Phi\|_{L^2(X)}^2.\label{ineq:w_decays} 
	\end{align} Here we recall that, by either of the cases proven in \Cref{cor:YMH_Lpbounds,cor:Lpbounds}, from the assumptions on $(\nabla,\Phi)$ we have that $\nabla\Phi\in L^p (X)$ for all $p \in [2,\infty]$. Moreover, from the discussion preceding the theorem, note that $G(x,\cdot{})\in L_{\text{loc}}^q(X)$ and, since $(X^7,g_{\varphi})$ is Ricci-flat, $G(x,y)\to 0$ as $\dist(x,y)\to\infty$. Hence, given $\varepsilon>0$, we can choose $\rho\gg 1$ so that the last term in the right-hand side of \cref{ineq:w_decays} is less than $\varepsilon/4$, and then we can choose $R\gg 1$ so that the first term in the right-hand side of \cref{ineq:w_decays} is less than $\varepsilon/4$. This gives $r:=R+\rho>0$ such that $\dist(o,x)>r$ implies $w(x)<\varepsilon$, showing that $w$ decays uniformly to zero as we wanted.
	
	Finally, since $\Delta|\Phi|^2 = -2|\nabla\Phi|^2$ by \cref{eq:subharmonic}, we conclude that $m^2:= |\Phi|^2+w$ is a smooth, nonnegative harmonic function on $X$. Since $(X^7,g_{\varphi})$ is Ricci-flat it follows that $m^2$ must be a constant (cf. \cite{yau1975harmonic}). This completes the proof. 
\end{proof}

\begin{remark}
	Note from the above proof that the assumptions of bounded geometry on $(X^7,\varphi)$ and bounded curvature on $\nabla$ are made in \Cref{thm:Finite_Mass} just to ensure, via \Cref{cor:YMH_Lpbounds,cor:Lpbounds}, that $\nabla\Phi\in L^{2(n-1)}(X)$.
\end{remark}

An interesting easy consequence of \Cref{thm:Finite_Mass} is the following.
\begin{corollary}
	Under the hypothesis of \Cref{thm:Finite_Mass}, the following inequality holds on $X$:
	\begin{equation}
		|\nabla\Phi|^2\lesssim \|F_\nabla\|_{L^\infty (X)}(m^2-|\Phi|^2).
	\end{equation}    
\end{corollary}
\begin{proof}
	By \Cref{lem:Bochner_Delta_nabla_Phi_0}, we know that there is $c>0$ such that
	\begin{equation}
		\Delta|\nabla\Phi|^2 \leqslant 2 c \|F_\nabla\|_{L^\infty (X)}|\nabla\Phi|^2.
	\end{equation} 
	On the other hand, we know that $w = m^2-|\Phi|^2$ is a nonnegative function such that $\Delta w = 2|\nabla\Phi|^2$. Thus,
	\begin{equation}
		\Delta \left( c\|F_\nabla\|_\infty w - |\nabla\Phi|^2 \right) \geqslant 0.
	\end{equation} Moreover, by \Cref{cor:YMH_Lpbounds,cor:Lpbounds,thm:Finite_Mass} we know that both $|\nabla\Phi|$ and $w$ decay. Therefore, the desired conclusion follows by the maximum principle.
\end{proof}

Now we recall that on a complete, noncompact Riemannian manifold with nonnegative Ricci curvature, an application of the Bishop--Gromov volume comparison theorem and a clever argument by Yau \cite{Yau1976} shows that for any $x \in X$ and $r>2$, the volume of the radius $r$ ball centered at $x \in X$ satisfies
\begin{equation}
	\Vol(B_1 (x))r \leqslant \Vol(B_r (x)) \lesssim r^n.
\end{equation} We now combine \Cref{thm:Finite_Mass} with a polynomial volume growth assumption on $(X^7,g_{\varphi})$ to obtain a slight refinement of the asymptotic behavior of $|\Phi|$. For the rest of this section, suppose that $(X^7,\varphi)$ is a complete, noncompact and irreducible $\rG_2$-manifold satisfying the following: there is $0\leqslant k<n-2$ such that for any $x\in X$ and $r\gg 1$ we have
\begin{equation}
	\text{Vol}(B_r(x))\sim r^{n-k}.
\end{equation}
Since $n - k > 2$, it follows that the \cref{ineq:vol_growth_nonparabolic} holds and therefore $(X^7,g_{\varphi})$ is nonparabolic. Furthermore, using \cref{eq:Green_function_bounds} we get that the minimal positive Green's function satisfies the asymptotic estimate
\begin{equation}
	G(x,y) \sim \dist(x,y)^{-(n-k-2)}\quad\text{when}\quad \dist(x,y)\gg 1.
\end{equation}
Another important observation of our volume growth assumption is the following. Let $x \in X$, $\rho$ be the radial coordinate on $T_x X$ and $\lambda: T_x X \to \mathbb{R}$ the function so that
\begin{equation}
	\exp_x^*\left( \vol_X \right) = \lambda \ \rd \rho \wedge \vol_{\mathbb{S}^{n-1}},
\end{equation}
where $n = 7$. Then the Laplacian comparison theorem, see \cite{Walpuski}*{Proposition~20.7}, states that
\begin{equation}\label{eq:Local_BG}
	\del_\rho \left( \rho^{-(n-1)} \lambda \right) \leqslant 0,
\end{equation}
away from the cut locus. As $\rho^{-(n-1)}\lambda$ converges to a constant as $\rho \to 0$ we find that for all $\rho>0$ we have $\lambda \lesssim \rho^{n-1}$. Furthermore, from the Bishop--Gromov volume comparison and the Ricci-flatness together with our volume growth assumption we have that $\lambda \sim \rho^{(n-1) -k}$ for $\rho \gg 1$.

Thus, under the above hypotheses on $(X^7,\varphi)$, we have the following consequences.

\begin{corollary}\label{cor:Asymptotics_Phi_Square}
	Let $(\nabla, \Phi)$ be a solution to the second order \cref{eq:2nd_Order_Eq_1,eq:2nd_Order_Eq_2} with $|\nabla \Phi|^2 \in L^1 (X)$ and suppose either that $|F_\nabla|\in L^{\infty}(X)$ or that $(\nabla,\Phi)$ is a $\rG_2$-monopole such that $|F_{\nabla}^{14}|\in L^{\infty}(X)$. Then, for any $o\in X$ there is $c>0$ and $R_0 \gg 1$ such that if $\dist(x,o)\geqslant R_0$ 
	\begin{equation}
		|\Phi(x)|^2 \leqslant m^2 - \frac{c}{\dist(o,x)^{n-k-2}}.
	\end{equation}
\end{corollary}

\begin{proof}
	It follows from \Cref{thm:Finite_Mass} that $w = m^2-|\Phi|^2$ decays to zero, and as $\Delta w = 2|\nabla \Phi|^2$, it is superharmonic. Now, let $G (o,\cdot)$ be the decaying Green's function with a pole at $o$. As $w$ is bounded, for any $R>0$ there is $\epsilon>0$ such that $w|_{\partial B_R(o)} > \epsilon G|_{\partial B_R(o)}$, and as both of these decay to zero at infinite we find that
	\begin{equation}
		\sup_{x \in X - B_R(o)} w > \sup_{x \in X - B_R(o)} \epsilon G.
	\end{equation}
	Thus, recalling that for $\dist(x,o) \gg 1$, the Green's function is of order $\dist(x,o)^{-(n-k-2)}$, we find from rearranging the above inequality that 
	\begin{equation}
		|\Phi(x)|^2 \leqslant m^2 - \frac{c}{\dist(x,o)^{n-k-2}},
	\end{equation}
	for some constant $c>0$.
\end{proof}

We finish this section by proving a simple result which constrains the asymptotic behavior of $\nabla\Phi$ (cf. \Cref{rem:Sharp}).

\begin{lemma}\label{lem:constrain_decay_nablaPhi}
	Let $(\nabla, \Phi)$ be a solution to the second order \cref{eq:2nd_Order_Eq_1,eq:2nd_Order_Eq_2} with $0 \neq |\nabla \Phi|^2 \in L^1 (X)$, and $r:X \to \mathbb{R}$ be a smooth positive radial function, meaning that for $r \gg 1$ we have $C^{-1}\dist(o,\cdot) \leqslant r(\cdot) \leqslant C\dist(o,\cdot)$ for some reference point $o\in X$. Then, there is a sequence of points $\lbrace x_i \rbrace_{i \in \mathbb{N}}$ with $r(x_i) \to \infty$ such that if it exists, then
	\begin{equation}
		\lim\limits_{i \to \infty} r(x_i)^{n-k-1} \left( |\partial_r |\Phi |^2|(x_i) \right) > 0;
	\end{equation}
	in other words $|\partial_r|\Phi|^2| = O (r^{-(n-k-1)})$.
\end{lemma}

\begin{proof}
	We prove the contrapositive. Suppose that $\lim_{r \to \infty} r^{n-k-1}|\partial_r |\Phi|^2| = 0$, then 
	\begin{align}
		\| \nabla \Phi \|_{L^2 (X)}^2	&= \lim\limits_{R \to \infty} \int\limits_{B_R} |\nabla \Phi|^2 \vol_X \\
										&= \lim\limits_{R \to \infty} \int\limits_{\Sigma_R} \langle \Phi , \ast \nabla \Phi \rangle \\
										&= \lim\limits_{R \to \infty} \frac{1}{2} \int\limits_{\Sigma_R} \partial_r |\Phi|^2 \vol_{\Sigma_R} \\
										&\lesssim \lim\limits_{R \to \infty} \left( R^{n-k-1} \sup_{\Sigma_R}| \partial_r |\Phi|^2| \right).
	\end{align}
	Thus, we find that $\nabla\Phi = 0$.
\end{proof}

\bigskip

\section{Bochner--Weitzenb\"ock formulas along the end}\label{sec:BW_formulas}

In the Bochner--Weitzenb\"ock formulas presented below we assume that the gauge group $\rG = \SU (2)$ and that we are away from the zeros of $\Phi$. By \Cref{rem:Finite_Mass}, if $(\nabla,\Phi)$ is a finite mass solution to the second order \cref{eq:2nd_Order_Eq_1,eq:2nd_Order_Eq_2}, then this condition is met sufficiently far out along the end of our complete, noncompact, irreducible $\rG_2$-manifold $(X^7,\varphi)$.

Whenever $\Phi(x) \neq 0$ for some $x \in X$, we can decompose $\mathfrak{g}_P = \mathfrak{g}_P^{||} \oplus \mathfrak{g}_P^\perp$ near $x$, by setting
\begin{equation}
	\mathfrak{g}_P^{||} = \ker \left( \ad_{\Phi(x)} :\mathfrak{g}_P \to \mathfrak{g}_P  \right),
\end{equation}
and $\mathfrak{g}_P^\perp$ its orthogonal complement. Clearly, when $\rG = \SU (2)$ then in fact $\mathfrak{g}_P^{||} = \langle \Phi \rangle$ is the real vector subbundle of $\mathfrak{g}_P$ whose fiber at a point $y$ (near $x$) is the $1$-dimensional vector subspace of $(\mathfrak{g}_P)_y$ generated by $\Phi(y)\neq 0$. In what follows we split any section $\chi$ of $\mathfrak{g}_P$ defined around $x$ as $\chi = \chi^{||} + \chi^\perp$, and we note that with the particular choice of metric on $\mathfrak{g}_P$ induced by $(a,b)\mapsto -2\tr(ab)$ it holds $|[\Phi,\chi]|\geqslant |\Phi||\chi^{\perp}|$. (Of course, for other choices of normalization of the negative of the Cartan--Killing form of $\mathfrak{g}=\mathfrak{su}(2)$ the inequality holds up to a constant.)

We start by applying this decomposition to refine the standard Bochner--Weitzenb\"ock inequality in \Cref{lem:Bochner_Delta_nabla_Phi_0}.

\begin{lemma}\label{lem:Bochner_Delta_nabla_Phi_1}
	Let $(\nabla, \Phi)$ be a solution to the second order \cref{eq:2nd_Order_Eq_1,eq:2nd_Order_Eq_2}. Then
	\begin{equation}\label[ineq]{ineq:Bochner_Delta_nabla_Phi_YMH}
		\frac{1}{2} \Delta |\nabla \Phi|^2 + |\nabla^2 \Phi|^2 + |[\Phi, \nabla \Phi]|^2\lesssim |(F_\nabla)^\perp| |(\nabla \Phi)^{||}| |(\nabla \Phi)^\perp| + |(F_\nabla)^{||}|\ |(\nabla \Phi)^\perp|^2. 
	\end{equation}
	If, moreover, $(\nabla, \Phi)$ is a solution to the $\rG_2$-monopole \cref{eq:Monopole} then
	\begin{equation}\label[ineq]{ineq:Bochner_Delta_nabla_Phi_Monopole}
		\begin{aligned}
			\frac{1}{2} \Delta |\nabla \Phi|^2 + |\nabla^2 \Phi|^2 + |[\Phi, \nabla \Phi]|^2 &\lesssim |(F_\nabla^{14})^\perp| |(\nabla \Phi)^{||}| |(\nabla \Phi)^\perp| \\
			& \quad + \left( |(\nabla \Phi)^{||}| + |(F_\nabla^{14})^{||}| \right) \ |(\nabla \Phi)^\perp|^2.
		\end{aligned}
	\end{equation}
\end{lemma}

\begin{proof}
	Recall \cref{eq:Weitzenbock_Intermediate} from \Cref{lem:Bochner_Delta_nabla_Phi_0}. Splitting $\nabla \Phi = (\nabla \Phi)^{||} + (\nabla \Phi)^\perp$ and similarly for $F_\nabla$ we find
	\begin{align}
		\langle \nabla \Phi, \ast [\ast F_\nabla \wedge \nabla \Phi] \rangle 	&=\langle \nabla \Phi, \ast [\ast (F_\nabla)^\perp \wedge (\nabla \Phi)^\perp] \rangle \\
		& \quad + \langle \nabla \Phi, \ast [\ast (F_\nabla)^\perp \wedge (\nabla \Phi)^{||}] \rangle + \langle \nabla \Phi,\ast [\ast (F_\nabla)^{||} \wedge (\nabla \Phi)^\perp] \rangle \\
		 = & 2 \langle (\nabla \Phi)^{||} , \ast [\ast (F_\nabla)^\perp \wedge (\nabla \Phi)^\perp] \rangle + \langle (\nabla \Phi)^\perp , \ast [\ast (F_\nabla)^{||} \wedge (\nabla \Phi)^\perp] \rangle.
	\end{align}
	Thus we get
	\begin{equation}
		|\langle \nabla \Phi , \ast [\ast F_\nabla \wedge \nabla \Phi] \rangle| \lesssim |(F_\nabla)^\perp| |(\nabla \Phi)^{||}| |(\nabla \Phi)^\perp| + |(F_\nabla)^{||}| |(\nabla \Phi)^\perp|^2,
	\end{equation}
	thus inserting into \cref{eq:Weitzenbock_Intermediate} yields \cref{ineq:Bochner_Delta_nabla_Phi_YMH}. 
	
	As for the case when $(\nabla, \Phi)$ satisfies the $\rG_2$-monopole \cref{eq:Monopole}, note that we have $3 \ast F_\nabla^7 = \nabla\Phi\wedge\psi$; in particular, $\left| (F_\nabla^7)^{||}\right|\lesssim \left|(\nabla\Phi)^{||}\right|$ and $\left| (F_\nabla^7)^\perp\right|\lesssim \left| (\nabla\Phi)^\perp\right|$. Using these, the orthogonal decomposition $F_\nabla = F_\nabla^7 + F_\nabla^{14}$ and inserting into \cref{ineq:Bochner_Delta_nabla_Phi_YMH}, we get \cref{ineq:Bochner_Delta_nabla_Phi_Monopole}.
\end{proof}

Next we compute Bochner type formulas for both $|[\Phi, \nabla \Phi]|^2$ and $|[F_\nabla,\Phi]|^2$.

\begin{lemma}\label{lem:Bochner1}
	Let $(\nabla, \Phi)$ be a solution to the second order \cref{eq:2nd_Order_Eq_1,eq:2nd_Order_Eq_2}. Then,
	\begin{equation}
		\frac{1}{2} \Delta |[\Phi, \nabla \Phi]|^2 + |\nabla [\Phi, \nabla \Phi]|^2 + |\Phi|^2|[\Phi, \nabla \Phi]|^2 \lesssim |(F_\nabla)^{||}||[\Phi, \nabla \Phi]|^2 + |(\nabla \Phi)^{||}| \ |[F_\nabla, \Phi]||[\Phi,\nabla \Phi]|.
	\end{equation}
\end{lemma}

\begin{proof}
	We start by computing each term in $\Delta_\nabla [\Phi, \nabla \Phi] = \rd_\nabla \rd_\nabla^* [\Phi, \nabla \Phi] + \rd_\nabla^* \rd_\nabla [\Phi, \nabla \Phi]$. First we get
	\begin{equation}
		\rd_\nabla^* [\Phi, \nabla \Phi] = - \rd_\nabla \ast \left( [\nabla \Phi \wedge \ast \nabla \Phi] + [\Phi , \rd_\nabla \ast \nabla \Phi] \right) = 0,
	\end{equation}
	as $[\nabla \Phi \wedge \ast \nabla \Phi] = 0$ and $\rd_\nabla \ast \rd_\nabla \Phi = 0$ are both zero. We also have
	\begin{align}
		\rd_\nabla^* \rd_\nabla [\Phi, \nabla \Phi]	&= \rd_\nabla^* \left( [\nabla \Phi \wedge \nabla \Phi] + [\Phi , \rd_\nabla^2 \Phi] \right) \\
		&= \rd_\nabla^* [\nabla \Phi \wedge \nabla \Phi] - \ast \rd_\nabla [ \Phi, [\Phi, \ast F_\nabla]] \\
		&= \rd_\nabla^* [\nabla \Phi \wedge \nabla \Phi] - \ast [\nabla \Phi \wedge [\Phi, \ast F_\nabla]] - \ast [\Phi, [\nabla \Phi \wedge \ast F_\nabla]] - [\Phi, [\Phi, \rd_\nabla^* F_\nabla]] \\
		&= \rd_\nabla^* [\nabla \Phi \wedge \nabla \Phi] - 2 \ast [\Phi, [\nabla \Phi \wedge \ast F_\nabla]] + [\Phi, [\Phi, [\Phi, \nabla \Phi]]].
	\end{align}
	Putting these two together we find that
	\begin{equation}
		\Delta_\nabla [\Phi, \nabla \Phi] = \rd_\nabla^* [\nabla \Phi \wedge \nabla \Phi] + \ast [\nabla\Phi \wedge [\ast F_\nabla,\Phi] ] + [\Phi , \Delta_\nabla \nabla \Phi].
	\end{equation}
	Now, a short computation shows that $\rd_\nabla^* [\nabla \Phi \wedge \nabla \Phi] = 2 [\nabla_i \nabla_j \Phi , \nabla_i \Phi] e^j $ and so
	\begin{align}
		\langle [\Phi , \nabla \Phi] , \rd_\nabla^* [\nabla \Phi \wedge \nabla \Phi] \rangle	&= 2 \langle [\Phi ,\nabla_j \Phi] , [\nabla_i \nabla_j \Phi , \nabla_i \Phi] \rangle \\
		&= 2 \langle [\nabla_j \Phi , [\Phi , \nabla_i \Phi] ] , \nabla_i \nabla_j \Phi \rangle \\
		&= - 2 \langle [\nabla_i \Phi , [\nabla_j \Phi , \Phi] ] , \nabla_i \nabla_j \Phi \rangle - 2 \langle [ \Phi , [\nabla_i \Phi , \nabla_j \Phi] ] , \nabla_i \nabla_j \Phi \rangle \\
		&= - 2 \langle [\nabla_i \Phi , [\nabla_j \Phi , \Phi] ] , [F_{ij},\Phi] \rangle - 2 \langle [\Phi , [\nabla_i \Phi , \nabla_j \Phi] ] , [F_{ij} , \Phi] \rangle \\
		& \quad + 2 \langle [\nabla_i \Phi , [\nabla_j \Phi , \Phi] ] , \nabla_j \nabla_i \Phi \rangle \\
		&= - 2 \langle [\nabla_i \Phi , [\nabla_j \Phi , \Phi] ] , [F_{ij}, \Phi] \rangle - 2 \langle [\Phi , [\nabla_i \Phi , \nabla_j \Phi]] , [F_{ij} , \Phi] \rangle \\
		& \quad - 2 \langle [\nabla_i \Phi , [ \Phi , \nabla_j \Phi ] ] , \nabla_j \nabla_i \Phi \rangle 
	\end{align}
	as $ [\nabla_j \Phi , \nabla_i \Phi]$ is anti-symmetric in $i,j$. Thus,
	\begin{align}
		2\langle [\Phi , \nabla \Phi] , \rd_\nabla^* [\nabla \Phi \wedge \nabla \Phi] \rangle &= 4 \langle [\nabla_j \Phi , [\Phi , \nabla_i \Phi] ] , \nabla_i \nabla_j \Phi \rangle \\
		&= - 2 \langle [\nabla_i \Phi , [\nabla_j \Phi , \Phi] ] , [F_{ij}, \Phi] \rangle - 2 \langle [\Phi , [\nabla_i \Phi , \nabla_j \Phi]] , [F_{ij} , \Phi] \rangle \\
		&= - 2 \langle [\nabla_i \Phi , [\nabla_j \Phi , \Phi] ] + [\Phi , [\nabla_i \Phi , \nabla_j \Phi]], [F_{ij}, \Phi] \rangle \\
		&= 2 \langle [\nabla_j \Phi , [\Phi , \nabla_i \Phi] ] , [F_{ij}, \Phi] \rangle .
	\end{align}
	We also have
	\begin{align}
		\nabla^*\nabla [\Phi , \nabla \Phi] &= \Delta_\nabla [\Phi , \nabla \Phi] - \ast [\ast F_\nabla \wedge [\Phi , \nabla \Phi]] \\
		&= [\Phi , \Delta_\nabla \nabla \Phi] + \rd_\nabla^* [\nabla \Phi \wedge \nabla \Phi] + \ast [\nabla\Phi \wedge [\ast F_\nabla,\Phi] ] - \ast [\ast F_\nabla \wedge [\Phi , \nabla \Phi]],
	\end{align}
	and using the second order equations again we find $\Delta_\nabla \nabla \Phi = [[\nabla \Phi, \Phi], \Phi]- \ast [\ast F_\nabla \wedge \nabla \Phi]$ and so
	\begin{align}
		\nabla^*\nabla [\Phi , \nabla \Phi] &= [\Phi , [[\nabla \Phi, \Phi], \Phi] ] + \rd_\nabla^* [\nabla \Phi \wedge \nabla \Phi] \\
		& \quad + \ast [\nabla\Phi \wedge [\ast F_\nabla,\Phi] ] - \ast [\ast F_\nabla \wedge [\Phi , \nabla \Phi]] - [\Phi , \ast [\ast F_\nabla \wedge \nabla \Phi] ] \\
		&= [\Phi , [[\nabla \Phi, \Phi], \Phi] ] + \rd_\nabla^* [\nabla \Phi \wedge \nabla \Phi] - 2 [\Phi , \ast [\ast F_\nabla \wedge \nabla \Phi] ],
	\end{align}
	where in the last inequality we used the (graded) Jacobi identity. Thus,
	\begin{align}
		\Delta \frac{|[\Phi, \nabla \Phi]|^2}{2} &= \langle [\Phi, \nabla \Phi] , \nabla^*\nabla [\Phi, \nabla \Phi] \rangle - |\nabla [\Phi, \nabla \Phi]|^2 \\
		&= \langle [\Phi, \nabla \Phi] , [\Phi , [[\nabla \Phi, \Phi], \Phi] ] \rangle - 2 \langle [\Phi, \nabla \Phi] ,  [\Phi , \ast [\ast F_\nabla \wedge \nabla \Phi] ] \rangle \\
		& \quad + \langle [\Phi, \nabla \Phi] , \rd_\nabla^* [\nabla \Phi \wedge \nabla \Phi] \rangle - |\nabla [\Phi, \nabla \Phi]|^2 \\
		&= - | [[\nabla \Phi, \Phi], \Phi]|^2 + 2 \langle [\nabla_j \Phi , [\Phi , \nabla_i \Phi] ] , [F_{ij}, \Phi] \rangle - 2 \langle [\Phi, \nabla \Phi] ,  [\Phi , \ast [\ast F_\nabla \wedge \nabla \Phi] ] \rangle \\
		& \quad - |\nabla [\Phi, \nabla \Phi]|^2 .
	\end{align}
	Since $| [[\nabla \Phi, \Phi], \Phi]|^2\geqslant |\Phi|^2|[\Phi,\nabla\Phi]|^2$, we conclude that
	\begin{align}
		\Delta \frac{|[\Phi, \nabla \Phi]|^2}{2} + |\nabla [\Phi, \nabla \Phi]|^2 + |\Phi|^2 | [\Phi, \nabla \Phi]|^2 &\lesssim |(\nabla \Phi)^{||}| |[F_\nabla, \Phi]| \ |[\Phi,\nabla \Phi]| \\
		& \quad + |\Phi| \ \left( |(F_\nabla)^{||}||(\nabla \Phi)^\perp| + |(F_\nabla)^\perp||(\nabla \Phi)^{||}| \right) \ |[\Phi,\nabla \Phi]| ,
	\end{align}
	with the stated inequality following from noticing that $|\Phi||(\nabla \Phi)^\perp| \leqslant |[\nabla \Phi , \Phi]|$ and similarly for $(F_\nabla)^\perp$.
\end{proof}

\begin{lemma}\label{lem:Bochner2}
	Let $(\nabla, \Phi)$ be a solution to the second order \cref{eq:2nd_Order_Eq_1,eq:2nd_Order_Eq_2}. Then,
	\begin{align}
		\frac{1}{2} \Delta |[F_\nabla, \Phi]|^2 + |\nabla [F_\nabla, \Phi] |^2 +|\Phi|^2|[F_\nabla, \Phi]|^2 &\lesssim (|\mathrm{Riem}| + |F_\nabla| + |\Phi|^{-2} |\nabla \Phi| ) \ |[F_\nabla, \Phi]|^2 \\
		&\text{ } + \left(|(\nabla \Phi)^{||}| + |\Phi|^{-2} |F_\nabla| |(\nabla \Phi)^{||}| + |\Phi|^{-1}|(\nabla F_\nabla)^{||}| \right) \ |[\nabla \Phi, \Phi]| \ |[F_\nabla, \Phi]| \\
		&\text{ } + |\Phi|^{-1} | \nabla_i[ F_\nabla , \Phi] ||( \nabla_i \Phi)^{||}| | [F_\nabla , \Phi] |.
	\end{align}
\end{lemma}

\begin{proof}
	We start by computing
	\begin{equation}\label{eq:First_Eq_Bochner_F_perp}
		\frac{1}{2} \Delta |[F_\nabla,\Phi]|^2 = \langle [F_\nabla,\Phi] , \nabla^*\nabla [F_\nabla,\Phi] \rangle - |\nabla [F_\nabla , \Phi]|^2 .
	\end{equation}
	We work out the first term above using the Leibniz rule, the Weitzenb\"ock formula for $F_\nabla$ given by $\nabla^*\nabla F_\nabla = \Delta_\nabla F_\nabla  + \mathrm{Riem}(F_\nabla) + (F_\nabla \cdot F_\nabla)$, and the second order equations to get
	\begin{align}
		\nabla^*\nabla [F_\nabla,\Phi]&= - \nabla_i \nabla_i [F , \Phi] \\
		&= [- \nabla_i \nabla_i F , \Phi] - 2 [\nabla_i F , \nabla_i \Phi] - [F , \nabla_i \nabla_i \Phi] \\
		&= [\Delta_\nabla F_\nabla , \Phi] + \mathrm{Riem}([F_\nabla, \Phi]) + (F_\nabla \cdot [F_\nabla, \Phi]) - 2 [\nabla_i F_\nabla , \nabla_i \Phi] + [F_\nabla, \Delta_\nabla \Phi] \\	
		&= [ [[F_\nabla , \Phi], \Phi] , \Phi] - [ [\nabla \Phi \wedge \nabla \Phi] , \Phi] + \mathrm{Riem}([F_\nabla, \Phi]) \\
		& \quad + (F_\nabla \cdot [F_\nabla, \Phi]) - 2 [\nabla_i F_\nabla , \nabla_i \Phi],
	\end{align}
	so that the first term on the right-hand side of \eqref{eq:First_Eq_Bochner_F_perp} is
	\begin{align}
		\langle [F_\nabla,\Phi] , \nabla^*\nabla [F_\nabla,\Phi] \rangle &= - |\Phi|^2 |[F_\nabla, \Phi]|^2 - \langle [F_\nabla,\Phi] ,[ [\nabla \Phi \wedge \nabla \Phi] , \Phi] \rangle + \langle [F_\nabla,\Phi] , \mathrm{Riem}([F_\nabla, \Phi])\rangle \\
		& \quad + \langle [F_\nabla,\Phi] , (F_\nabla \cdot [F_\nabla, \Phi]) \rangle - 2 \langle [F_\nabla, \Phi] , [\nabla_i F_\nabla , \nabla_i \Phi] \rangle \\
		&\lesssim (-|\Phi|^2 + |\mathrm{Riem}| + |F_\nabla|) \ |[F_\nabla, \Phi]|^2 \\
		& \quad + \left(|(\nabla \Phi)^{||}| \ |\Phi| \ |(\nabla \Phi)^\perp| + |[\nabla_i F_\nabla , \nabla_i \Phi]|\right) \ |[F_\nabla, \Phi]| .
	\end{align}
	Now, notice that $|\Phi| |(\nabla \Phi)^\perp| \leqslant \left| [\nabla \Phi , \Phi]\right|$ while
	\begin{align}
		|[\nabla_i F_\nabla , \nabla_i \Phi]| &\lesssim |(\nabla_i F_\nabla)^{||}| |( \nabla_i \Phi)^\perp| + |(\nabla F_\nabla)^\perp| |( \nabla_i \Phi)^{||}| \\
		&\lesssim |\Phi|^{-1}|(\nabla F_\nabla)^{||}| |[\Phi, \nabla \Phi]| + |\Phi|^{-2}|[\Phi, [\Phi, \nabla_i F_\nabla ]]| |( \nabla_i \Phi)^{||}| \\
		&\lesssim |\Phi|^{-1}|(\nabla F_\nabla)^{||}| |[\Phi, \nabla \Phi]| + |\Phi|^{-2}|[\Phi, \nabla_i[\Phi, F_\nabla ]] - [\Phi, [\nabla_i \Phi, F_\nabla ]]| |( \nabla_i \Phi)^{||}| \\
		&\lesssim |\Phi|^{-1}|(\nabla F_\nabla)^{||}| |[\Phi, \nabla \Phi]| + |\Phi|^{-1}| \nabla_i[\Phi, F_\nabla ]||( \nabla_i \Phi)^{||}| \\
		& \quad + |\Phi|^{-2} |[\nabla_i \Phi, [F_\nabla, \Phi ]] +[F_\nabla, [\Phi , \nabla_i \Phi]] | |( \nabla_i \Phi)^{||}| \\
		&\lesssim |\Phi|^{-1}|(\nabla F_\nabla)^{||}| |[\Phi, \nabla \Phi]| + |\Phi|^{-1}| \nabla_i[\Phi, F_\nabla ]||( \nabla_i \Phi)^{||}| \\
		& \quad + |\Phi|^{-2} \left( |\nabla \Phi| |[F_\nabla, \Phi ]| + |F_\nabla| |[\Phi , \nabla \Phi]| \right) |( \nabla_i \Phi)^{||}| , 
	\end{align}
	which upon inserting above gives
	\begin{align}
		\langle [F_\nabla,\Phi] , \nabla^*\nabla [F_\nabla,\Phi] \rangle + |\Phi|^2|[F_\nabla, \Phi]|^2 &\lesssim (|\mathrm{Riem}| + |F_\nabla| + |\Phi|^{-2} |\nabla \Phi| ) \ |[F_\nabla, \Phi]|^2 \\
		&\text{ }+ \left(|(\nabla \Phi)^{||}| + |\Phi|^{-2} |F_\nabla| |(\nabla \Phi)^{||}| + |\Phi|^{-1}|(\nabla F_\nabla)^{||}| \right) \ |[\nabla \Phi, \Phi]| \ |[F_\nabla, \Phi]| \\
		&\text{ }+ |\Phi|^{-1} | \nabla_i[ F_\nabla , \Phi] ||( \nabla_i \Phi)^{||}| | [F_\nabla , \Phi] |.
	\end{align}
	Inserting into \cref{eq:First_Eq_Bochner_F_perp} we obtain the inequality in the statement.
\end{proof}

We now give the main consequence of the Bochner inequalities proved in this section.

\begin{corollary}\label{cor:Elliptic_Inequality_For_Exponential_Decay}
	Let $(\nabla, \Phi)$ be a solution to the second order \cref{eq:2nd_Order_Eq_1,eq:2nd_Order_Eq_2} with finite mass $m\neq 0$, set $f = \left( [\nabla \Phi , \Phi] , [F_\nabla, \Phi] \right) \in \Omega^1(\fg_P) \oplus \Omega^2(\fg_P)$ and suppose that $|\mathrm{Riem}|$, $|\nabla \Phi|$, $|F_\nabla|$ and $|\nabla F_\nabla|$ decay uniformly to zero at infinity. Then, for every $\delta\in (0,1]$, there is a sufficiently large compact set $K=K(\delta)\subset X$ outside of which there holds
	\begin{equation}\label{ineq: ref_Bochner_transv}
		\frac{1}{2} \Delta |f|^2 + (1 - \delta)|\nabla f|^2 \leqslant -\frac{1}{2}|\Phi|^2|f|^2.
	\end{equation} Moreover, if $0<\delta_1<\delta_2\leqslant 1$ then $K(\delta_1)\supset K(\delta_2)$.
\end{corollary}

\begin{proof}
    Since $(\nabla, \Phi)$ has finite mass $m>0$, it follows from \Cref{rem:Finite_Mass} that there is a sufficiently large compact set $K\subset X$ such that
	\begin{equation}\label{ineq: lb_Phi}
	|\Phi|\geqslant\frac{m}{2}>0,\quad\text{ on $X\setminus K$}.
	\end{equation}
    Using this and summing the inequalities in \Cref{lem:Bochner1,lem:Bochner2} and using Young's inequality to deal with the mixed terms of the form $|[\nabla \Phi, \Phi]| \ |[F_\nabla, \Phi]| $, we get that there is $c>0$ such that outside $K$ there holds
    \begin{align}
		\frac{1}{2}\Delta|f|^2 &\leqslant -|\nabla f|^2 + c\frac{2}{m}|\nabla[F_\nabla,\Phi]||(\nabla\Phi)^{||}||[F_\nabla,\Phi]|-|\Phi|^2|f|^2\nonumber\\
		&\text{ }+ c\left( |\mathrm{Riem}|+|F_\nabla|+\frac{4}{m^2}|\nabla\Phi| + |(\nabla\Phi)^{||}| + \frac{4}{m^2}|F_\nabla||(\nabla\Phi)^{||}| +\frac{2}{m}|(\nabla F_\nabla)^{||}| \right)|f|^2.\label{ineq: general_ref_Bochner_transv}
	\end{align}
	 Now, by the hypotheses, the functions $|\mathrm{Riem}|$, $|\nabla\Phi|$, $|F_\nabla|$ and $|\nabla F_\nabla|$ decay uniformly to zero at infinity, so by taking $K$ sufficiently large, we can make the whole term inside the parenthesis in the second line of the right-hand side of inequality \eqref{ineq: general_ref_Bochner_transv} to be less than $\frac{m^2}{16c}$. On the other hand, for any $\delta\in (0,1]$, Young's inequality gives
	\begin{align}
	c\frac{2}{m}|\nabla[F_\nabla,\Phi]||(\nabla\Phi)^{||}||[F_\nabla,\Phi]| &\leqslant \delta|\nabla[F_\nabla,\Phi]|^2 + \frac{c^2}{\delta m^2}|(\nabla\Phi)^{||}|^2|[F_\nabla,\Phi]|^2\\
	&\leqslant \delta|\nabla f|^2 + \frac{c^2}{\delta m^2}|(\nabla\Phi)^{||}|^2|f|^2,
	\end{align} and again since $|\nabla\Phi|$ decays, by taking a sufficiently large compact subset $K(\delta)\supset K$, we can arrange $\frac{c^2}{\delta m^2}|(\nabla\Phi)^{||}|^2\leqslant\frac{m^2}{16}\leqslant\frac{|\Phi|^2}{4}$ outside $K(\delta)$; note that one chooses $K(\delta)$ in such a way that $K(\delta_1)\supset K(\delta_2)$ whenever $0<\delta_1<\delta_2\leqslant 1$. In conclusion, combining these facts with inequality \eqref{ineq: general_ref_Bochner_transv}, we get the desired inequality \eqref{ineq: ref_Bochner_transv} outside $K(\delta)$.

\end{proof}

Here we state and prove a further improvement of the inequalities in \Cref{lem:Bochner_Delta_nabla_Phi_1}. These are valid only sufficiently far out along the end and follow explicitly making use of the finiteness of the mass.

\begin{lemma}[Improved Bochner inequalities]\label{lem:Improved_Bochner_Delta_Nabla_Phi}
	In the conditions of \Cref{thm:Finite_Mass}, the following inequalities hold outside a sufficiently large compact subset:
	\begin{subequations}
	\begin{align}
		\Delta |\nabla \Phi|^2 &\leqslant c|(F_\nabla)^\perp| |(\nabla \Phi)^{||}| |(\nabla \Phi)^\perp| + \left(c|(F_\nabla)^{||}| - \frac{m^2}{4}\right)|(\nabla \Phi)^\perp|^2, \label[ineq]{ineq:YMH_improv_1} \\
		\Delta |\nabla \Phi|^2 &\lesssim |(F_\nabla)^\perp| |(\nabla \Phi)^{||}| |(\nabla \Phi)^\perp|,  \label[ineq]{ineq:YMH_improv_2} \\
		\Delta |\nabla \Phi|^2 &\lesssim |(F_\nabla)^\perp|^2 |(\nabla \Phi)^{||}|^2. \label[ineq]{ineq:YMH_improv_3}
	\end{align}
	\end{subequations}
	In case $(\nabla, \Phi)$ is a $\rG_2$-monopole, then we have:
	\begin{subequations}
	\begin{align}
		\Delta |\nabla \Phi|^2 &\leqslant c|(F_\nabla^{14})^\perp| |(\nabla \Phi)^{||}| |(\nabla \Phi)^\perp| + \left(c|(\nabla \Phi)^{||}| + c|(F_\nabla^{14})^{||}| - \frac{m^2}{4}\right)|(\nabla \Phi)^\perp|^2,  \label[ineq]{ineq:improv_1} \\
		\Delta |\nabla \Phi|^2 &\lesssim |(F_\nabla^{14})^\perp| |(\nabla \Phi)^{||}| |(\nabla \Phi)^\perp|, \label[ineq]{ineq:improv_2} \\
		\Delta |\nabla \Phi|^2 &\lesssim |(F_\nabla^{14})^\perp|^2 |(\nabla \Phi)^{||}|^2.  \label[ineq]{ineq:improv_3}
	\end{align}
	\end{subequations}
\end{lemma}

\begin{proof}
	We do the proof of the monopole case, the other is entirely analogous. Recall from \cref{ineq:Bochner_Delta_nabla_Phi_Monopole} that
	\begin{equation}
		\Delta |\nabla \Phi|^2 + |[\Phi, \nabla \Phi]|^2\lesssim |(F_\nabla^{14})^\perp| |(\nabla \Phi)^{||}| |(\nabla \Phi)^\perp| + \left(|(\nabla \Phi)^{||}| + |(F_\nabla^{14})^{||}| \right)|(\nabla \Phi)^\perp|^2.
	\end{equation}
	Since, by \Cref{thm:Finite_Mass}, $(\nabla, \Phi)$ has finite mass $m>0$, it follows that outside a sufficiently large compact set $K$ one has $|\Phi| \geqslant m/2$ (see \Cref{rem:Finite_Mass}) and thus
	\begin{equation}
		|[\Phi, \nabla \Phi]|^2 \geqslant |\Phi|^2|(\nabla \Phi)^\perp|^2 \geqslant \frac{m^2}{4}|(\nabla \Phi)^\perp|^2.
	\end{equation}
	Therefore we get \cref{ineq:improv_1}. Now recall that $|F_\nabla^{14}|$ decays by hypothesis and that $|\nabla\Phi|$ also decays as a consequence of \Cref{cor:Lpbounds}. Thus, if $K$ is large enough, then the last term in \cref{ineq:improv_1} becomes negative, so that we get \cref{ineq:improv_2}. 
	
	Finally, to show \cref{ineq:improv_3}, note that we can use Young's inequality in the form 
	\begin{equation}
		2 |(F_\nabla^{14})^\perp| |(\nabla \Phi)^{||}| |(\nabla \Phi)^\perp| \leqslant \epsilon^{-1} |(F_\nabla^{14})^\perp|^2 |(\nabla \Phi)^{||}|^2 + \epsilon |(\nabla \Phi)^\perp|^2,
	\end{equation}
	with $\epsilon>0$ to be fixed later. Then, by \cref{ineq:improv_1}, we find that
	\begin{equation}
		\Delta |\nabla \Phi|^2 \leqslant c\epsilon^{-1} |(F_\nabla^{14})^\perp|^2 |(\nabla \Phi)^{||}|^2 + \left( c\epsilon + c|(\nabla \Phi)^{||}| + c|(F_\nabla^{14})^{||}| - \frac{m^2}{4} \right) |(\nabla \Phi)^\perp|^2.
	\end{equation}
	Now choose $\epsilon \ll m^2$, then given that both $|(\nabla \Phi)^{||}|$, $|(F_\nabla^{14})^{||}|$ decay, we conclude that the second term becomes negative so
	\begin{equation}
		\Delta |\nabla \Phi|^2 \lesssim \epsilon^{-1} |(F_\nabla^{14})^\perp|^2 |(\nabla \Phi)^{||}|^2 ,
	\end{equation}
	as we wanted.
\end{proof}

\bigskip

\section{Refined asymptotics in the AC case}\label{sec:AC}

In this section, let $(X, \varphi)$ be an (irreducible) AC $\rG_2$-manifold as in \Cref{def:AC} and $\rG = \SU (2)$. Here we prove the first part of our second \Cref{thm:Main_Theorem_2}. For the reader's convenience we restate this here as follows.

\begin{theorem}\label{thm:First_part_of_Main_Theorem_2}
	Let $(X, \varphi)$ be an AC $\rG_2$-manifold and $(\nabla, \Phi)$ satisfy the second order \cref{eq:2nd_Order_Eq_1,eq:2nd_Order_Eq_2} with $|\nabla\Phi|^2\in L^1 (X)$. Suppose either that $|F_\nabla|$ decays uniformly along the end or $(\nabla, \Phi)$ is a $\rG_2$-monopole such that $|F_\nabla^{14}|$ decays uniformly along the end. Then:
	\begin{itemize}
	  \item[(i)] the transverse components of $\nabla\Phi$ and $F_\nabla$ decay exponentially along the end;
	  \item[(ii)] $|\nabla \Phi| = O (r^{-(n-1)})$ as $r\to\infty$.
  \end{itemize}
\end{theorem}

\begin{remark}[The decay of $|\nabla\Phi|$ in (ii) is sharp]\label{rem:Sharp}
	Let $(X, \varphi)$ be an irreducible AC $\rG_2$-manifold and $(\nabla, \Phi)$ satisfy the second order \cref{eq:2nd_Order_Eq_1,eq:2nd_Order_Eq_2} with $|\nabla\Phi|^2\in L^1 (X)$. Suppose either that $|F_\nabla|$ is bounded or $(\nabla, \Phi)$ is a $\rG_2$-monopole and $|F_\nabla^{14}|$ is bounded. Let $m > 0$ be the mass of $(\nabla,\Phi)$, cf. \Cref{thm:Finite_Mass}. We show that if $\nabla\Phi\neq 0$, then $(\nabla,\Phi)$ cannot decay faster than as in (ii) of \Cref{thm:First_part_of_Main_Theorem_2}, as a consequence of the same argument in \Cref{lem:constrain_decay_nablaPhi}. Indeed, start noting that since $\Delta_\nabla \Phi = 0$, it follows from Stokes' Theorem that
	\begin{equation}
		\int\limits_{B_R} |\nabla\Phi|^2 \vol_X = \int\limits_{\Sigma_R} \langle \Phi, \ast \nabla\Phi \rangle.
	\end{equation}
	Thus, if $|\nabla \Phi| = o (r^{-(n-1)})$ as $r\to\infty$ then
	\begin{equation}
		\lim\limits_{r\to\infty} \langle\Phi,\ast\nabla\Phi\rangle r^{n-1} \leqslant \lim\limits_{r \to \infty} |\Phi||\nabla\Phi| r^{n-1} \leqslant m \lim_{r \to \infty} |\nabla\Phi| r^{n-1} = 0; 
	\end{equation}
	hence, using $|\nabla\Phi|^2 \in L^1 (X)$,
	\begin{equation}
		\|\nabla \Phi\|_{L^2 (X)}^2 = \lim\limits_{R \to \infty} \int\limits_{\Sigma_R} \langle \Phi, \ast \nabla \Phi \rangle = 0,
	\end{equation}
	i.e. $\nabla \Phi = 0$.
\end{remark}

\smallskip

\subsection{Exponential decay for the transverse components}

In this subsection we prove that the components on $\nabla \Phi$ and $F_\nabla$ transverse to the Higgs field decay exponentially with $r$. For the gauge group $\rG = \SU (2)$ it is enough to prove that both $[\Phi, \nabla \Phi]$ and $[\Phi, F_\nabla]$ decay exponentially.

When $(X,\varphi)$ is AC note that $|\mathrm{Riem}|$ decays uniformly along the end. Let $(\nabla, \Phi)$ satisfy the second order \cref{eq:2nd_Order_Eq_1,eq:2nd_Order_Eq_2} with $|\nabla\Phi|^2\in L^1 (X)$. If $|F_\nabla|$ decays uniformly along the end or $(\nabla, \Phi)$ is a $\rG_2$-monopole such that $|F_\nabla^{14}|$ decays uniformly along the end, then we have that $|\nabla\Phi|, |F_\nabla|$ and $|\nabla F_\nabla|$ also decay uniformly along the end by \Cref{cor:allderivativesdecay,cor:Lpbounds}. Furthermore, one has that $|\Phi| \to m$ uniformly along the end by \Cref{thm:Finite_Mass}. Thus, we are in the conditions of \Cref{cor:Elliptic_Inequality_For_Exponential_Decay}, which in turn implies (taking $\delta=1$) that sufficiently far along the end of $X$ we have that $f = \left( [\nabla \Phi , \Phi] , [F_\nabla, \Phi] \right) \in \Omega^1(\fg_P) \oplus \Omega^2(\fg_P)$ satisfies
\begin{equation}\label[ineq]{ineq:transverse_bochner}
	\Delta |f|^2 \leqslant -\frac{m^2}{4}|f|^2.
\end{equation} This has the remarkable consequence that the transverse components controlled by $|f|$ decay exponentially along the end. The following gives part (i) of \Cref{thm:First_part_of_Main_Theorem_2}.

\begin{proposition}\label{prop:Exponential_Decay_Transverse}
	Let $(X,\varphi)$ be AC and assume that the pair $(\nabla, \Phi)$ satisfies the second order \cref{eq:2nd_Order_Eq_1,eq:2nd_Order_Eq_2} and $|\nabla\Phi|^2\in L^1 (X)$. Suppose furthermore that either that $|F_\nabla|$ decays uniformly along the end or $(\nabla, \Phi)$ is a $\rG_2$-monopole such that $|F_\nabla^{14}|$ decays uniformly along the end. Denote by $m>0$ the mass of $(\nabla, \Phi)$. Then, there are constants $c>0$ (depending only on the geometry), $R>0$ (depending only on the geometry and $m$) and $M>0$ (depending on $(\nabla, \Phi)$) such that for $r \geqslant R$ we have
	\begin{equation}
		|[\nabla \Phi, \Phi]|^2 + |[F_\nabla , \Phi]|^2 \leqslant M e^{-cmr}.
	\end{equation}
 	In particular, for $r \geqslant R$,
  	\begin{align}
  		|(\nabla \Phi)^\perp|^2 &\lesssim m^{-2}|[\nabla \Phi, \Phi]|^2 \leqslant m^{-2} M e^{-cmr}\lesssim e^{-cmr},\\
  		|(F_\nabla^{14})^\perp|^2 & \leqslant |F_\nabla^\perp|^2 \lesssim m^{-2}|[F_\nabla,\Phi]|^2 \leqslant m^{-2} M e^{-cmr}\lesssim e^{-cmr}.  
  	\end{align}
\end{proposition}

\begin{proof}
	Along the end of $X$, let $r$ be the pullback of the radius function from the cone. Then, using the almost isometry to the cone we can write
	\begin{equation}
		- \Delta = \frac{\del^2}{\del r^2} + \frac{n-1}{r} \frac{\del}{\del r} + \frac{1}{r^2} \Delta_\Sigma + \ldots,
	\end{equation}
	where the dots denote lower order terms. Furthermore, if $w(r)$ is a function of $r$ we find that
	\begin{align}
		\Delta w	&= -w''(r)|\rd r|^2 + w'(r) \Delta r \\
					&= -w''(r) - \frac{n-1}{r} w'(r) + \ldots.
	\end{align}
	Thus, let $M>0$ and $c>0$ both to be fixed later and set 
	\begin{equation}
		w \coloneqq M e^{-cmr}.
	\end{equation} 
	Then,
	\begin{equation}
		\Delta w = \left( -c^2m^2 + \frac{n-1}{r} cm + \ldots \right) \ w .
	\end{equation}
	It follows that taking $R\gg 1$ sufficiently large, depending on $m$ and the geometry, we can choose $c>0$ depending only on the geometry\footnote{Recall that since $(X, \varphi)$ is AC, we have that $- \Delta r \geqslant (n - 1) r^{- 1} (1 - O(r^{-\nu^\prime}))$, for some $\nu^\prime > 0$.} such that for $r\geqslant R$
	\begin{equation}
		\Delta w \geqslant - \frac{m^2}{4} w .
	\end{equation}
	 Using \cref{ineq:transverse_bochner}, we find that for $r \geqslant R$
	\begin{equation}
		\Delta \left( |f|^2 - w \right) \leqslant - \frac{m^2}{4}\ \left( |f|^2 - w \right) .
	\end{equation}
	Now, both $f$ and $w$ decay to zero along the end, and for $M>0$ large enough, depending on $(\nabla,\Phi)$, we have $|f|^2 \leqslant w$ at $r = R$. Therefore, using the inequality above we can apply the maximum principle to $|f|^2-w$ in the region $r\geqslant R$ to find that
	\begin{equation}
		|f|^2 \leqslant w,
	\end{equation}
 	within that region.
\end{proof}

\smallskip

\subsection{Bounds from Hardy's inequality}\label{subsec:Hardy1}

This section uses Hardy's inequality, the Agmon technique and Moser iteration to prove the following.
\begin{proposition}\label{prop:Decay_1}
	Under the hypothesis of \Cref{thm:First_part_of_Main_Theorem_2} and for all $\epsilon > 0$, we have
	\begin{equation}
		|\nabla \Phi|^2 \leqslant \frac{C}\epsilon r^{-2(n-2) + \epsilon}.
	\end{equation}
\end{proposition}

We divide the proof of this result into a series of lemmas which we prove below. The concluding proof is given at the end. 

For the rest of the paper, let $L \gg 2l > l \gg 1$ and $r_{l, L}$ be the following function:
\begin{equation}
	r_{l, L} = \begin{cases}
					0 & \text{on $B_l$,} \\
          			2(r-l)	& \text{on $B_{2l} - B_l$,} \\
 					r									& \text{on $B_L - B_{2l}$,} \\
 					L									& \text{on $X - B_L$.}
 	\end{cases} \label{eq:rlL_def}
\end{equation}
Note that $r_{l, L} \in L_1^\infty (X)$ and $\rd r_{l, L} = \partial_r r_{l, L} \rd r$, with
\begin{equation}
	\partial_r r_{l, L} = \begin{cases}
          				2	& \text{on $B_{2l} - B_l$,} \\
 						1								& \text{on $B_L - B_{l}$,} \\
 						0								& \text{otherwise.}
 	\end{cases} \label{eq:drlL_def}
\end{equation}

\begin{lemma}\label{lem:Agmon_first_step}
	Under the hypotheses of \Cref{thm:First_part_of_Main_Theorem_2}, the tensors $\nabla(r \nabla \Phi)$, $r\nabla^2\Phi$ and $r^\alpha(\nabla \Phi)^\perp$ are all square integrable, for all $\alpha>0$.
\end{lemma}

\begin{proof}
	We start claiming that for any real function $f \in L_1^\infty (X)$, with support in $X - B_l$, we have the \emph{Agmon identity}
	\begin{equation}
		\|\nabla(f \nabla \Phi) \|_{L^2 (X)}^2 = \|\rd f \otimes \nabla \Phi \|_{L^2 (X - B_l)}^2 + \int\limits_{X - B_l} \langle f^2 \nabla \Phi , \nabla^*\nabla (\nabla \Phi) \rangle .
	\end{equation}
	Indeed, it follows from \Cref{lem:Improved_Bochner_Delta_Nabla_Phi,prop:Exponential_Decay_Transverse}, under the hypotheses above, that
	\begin{equation}
		\langle\nabla \Phi , \nabla^*\nabla (\nabla \Phi) \rangle \lesssim e^{-cmr} |\nabla \Phi|^2,
	\end{equation}
	which is in $L^1(X)$ since $|\nabla\Phi|^2\in L^1(X)$. Thus, the claim follows by the obvious approximation argument and integration by parts. 
	
	Next, we prove the statement. Let $f = r_{l, L}$ as above. Then, since $r^2 e^{-cmr} \leqslant 1$ for $r\geqslant l \gg 1$, we have, for all $L$ that
	\begin{equation}
		\|\nabla(r_{l, L} \nabla \Phi) \|_{L^2 (X - B_l)}^2 \lesssim \| \nabla \Phi \|_{L^2 (X)}^2 + \int\limits_{X - B_l} r_{l, L}^2 e^{-cmr} |\nabla \Phi|^2 \vol_X \lesssim \| \nabla \Phi \|_{L^2 (X)}^2.
	\end{equation}
	Hence $\nabla(r \nabla \Phi) \in L^2 (X)$. Since $\nabla \Phi$ is also in $L^2 (X)$, we have that
	\begin{equation}
		r \nabla^2 \Phi = \nabla(r \nabla \Phi) - \rd r \otimes \nabla \Phi \in L^2 (X).
	\end{equation}
	Finally, note that \Cref{prop:Exponential_Decay_Transverse} immediately implies that $r^\alpha(\nabla \Phi)^\perp\in L^2 (X)$ for all $\alpha>0$, which concludes the proof.
\end{proof}

Next we improve the above result to allow higher powers of the radius function. We make use of the following Hardy type inequality.

\begin{lemma}[Hardy's Inequality, cf. \cite{degeratu2013witten}*{Proposition 3.7}]\label{lem:Hardy}
	Let $(X, \varphi)$ be an AC $\rG_2$-manifold with rate $\nu < 0$. Then there is a constant $C_H > 0$ such that for any function $\xi \in H^1$ with support in $X - B_l$, one has
	\begin{equation}
		\| \nabla \xi\|_{L^2 (X - B_l)}^2 \geqslant \left( \frac{n-2}{2} \right)^2 \| r^{- 1} \xi \|_{L^2 (X - B_l)}^2 - C_H \|r^{-1 + \nu} \xi \|_{L^2 (X - B_l)}^2. \label[ineq]{ineq:Hardy}
	\end{equation}
\end{lemma}

\begin{lemma}\label{lem:First_time_Hardy}
	Under the hypotheses of \Cref{thm:First_part_of_Main_Theorem_2}, there is a constant $C>0$, such that for all $\epsilon > 0$
	\begin{equation}
		\|r^{\frac{n - 4}{2} - \epsilon}\nabla \Phi\|_{L^2 (X - B_l)}^2 \leqslant \frac{C}\epsilon \| \nabla \Phi \|_{L^2 (X - B_l)}^2. \label[ineq]{ineq:nablaphi_bound}
	\end{equation}
	Moreover both $\nabla(r^\alpha \nabla \Phi)$ and $r^\alpha\nabla^2\Phi$ are square integrable as long as $\alpha < \tfrac{n - 2}{2} = \tfrac{5}{2}$. 
\end{lemma}

\begin{proof}
	It suffices to consider the case where $\alpha > 1$ by \Cref{lem:Agmon_first_step}. If $f \in L_1^\infty (X)$ is a function with support in $X - B_l$, then (cf. proof of \Cref{lem:Agmon_first_step})
	\begin{equation}\label[ineq]{ineq:Agmon}
		\|\nabla(f \nabla \Phi) \|_{L^2 (X - B_l)}^2 \leqslant \|\rd f \otimes \nabla \Phi \|_{L^2 (X - B_l)}^2 + c \int\limits_{X - B_l} f^2 e^{-cmr} |\nabla \Phi|^2 \vol_X.
	\end{equation}
	On the other hand, by \Cref{lem:Agmon_first_step} we may apply Hardy's \cref{ineq:Hardy} to the function $f|\nabla\Phi|$ and using Kato's inequality this gives us
	\begin{equation}
		\frac{(n - 2)^2}{4} \ \|r^{- 1} f \nabla \Phi \|_{L^2 (X - B_l)}^2 - C_H \|r^{-1 + \nu} f \nabla \Phi \|_{L^2 (X - B_l)}^2 \leqslant \|\nabla(f \nabla \Phi) \|_{L^2 (X - B_l)}^2.
	\end{equation}
	The combination of these two inequalities gives
	\begin{align}
 		\frac{(n - 2)^2}{4} \ \|r^{- 1} f \nabla \Phi \|_{L^2 (X - B_l)}^2 &\leqslant \|\rd f \otimes \nabla \Phi \|_{L^2 (X - B_l)}^2 + C_H \|r^{- 1 + \nu} f \nabla \Phi \|_{L^2 (X - B_l)}^2 \\
 			& \quad + c \int\limits_{X - B_l} f^2 e^{-cmr} |\nabla \Phi|^2 \vol_X. \label[ineq]{ineq:combination}
	\end{align}
	Let now $f = r_{l, L}^\alpha$ from \cref{eq:rlL_def}. Using $|\rd r| = 1 + O \left( r^\nu \right)$ and \cref{eq:drlL_def}, we get that
	\begin{equation}
		|\rd r_{l, L}^\alpha| \leqslant 2\alpha r^{\alpha-1} \chi_{[l, 2l]} (r) + \alpha r_{l, L}^\alpha r^{- 1} \chi_{[2l, L]} (r) + O \left(r_{l, L}^\alpha r^{- 1 + \nu} \right)\chi_{[l, L]} (r),
	\end{equation}
	and thus, we can rearrange \cref{ineq:combination} as
	\begin{align}
		\left( \frac{(n-2)^2}{4} - \alpha^2 \right) \ \|r^{- 1} r_{l, L}^\alpha \nabla \Phi \|_{L^2 (X - B_l)}^2 &\lesssim \|r^{- 1 + \nu} r_{l, L}^\alpha \nabla \Phi \|_{L^2 (X - B_l)}^2 + \int\limits_{X - B_l} r_{l, L}^{2\alpha} e^{-cmr} |\nabla \Phi|^2 \vol_X \\
		& \quad + l^{2 (\alpha - 1)} \| \nabla \Phi \|_{L^2 (B_{2l} - B_l)}^2.
	\end{align}
	The right hand side is finite and bounded independent of $L$, as long as $\alpha \leqslant 1 - \nu > 1$. Thus, if also $\alpha < \tfrac{n - 2}{2} = \tfrac{5}{2}$, then
	\begin{equation}\label[ineq]{ineq:nablaphi_bound_0}
		\| r^{\alpha - 1} \nabla \Phi \|_{L^2 (X - B_l)}^2 \leqslant \frac{C}{\tfrac{5}{2} - \alpha} \| \nabla \Phi \|_{L^2 (X - B_l)}^2.
	\end{equation}
	This works for any $\alpha < \min \left( 1 - \nu, \tfrac{5}{2} \right)$. If $\nu \leqslant - \tfrac{3}{2}$, then for all $\alpha < \tfrac{5}{2}$, we immediately have \cref{ineq:nablaphi_bound_0}.\footnote{Note that, all known examples have $\nu \leqslant - 3$.}

	On the other hand, if $\nu \in (-\tfrac{3}{2}, 0]$, we can still prove \cref{ineq:nablaphi_bound_0} for all $\alpha < \tfrac{n - 2}{2} = \tfrac{5}{2}$ by iterating the following argument, finitely many times: Start with $\alpha_0 \coloneqq 1$, from which obtain that $\| r^{\alpha_0 - 1} \nabla \Phi \|_{L^2 (X)}^2 < \infty$. Then, for all $k \geqslant 1$, let
	\begin{equation}
		\alpha_k \coloneqq \min \left( 1 - k \nu, \tfrac{5}{2} \right) \leqslant \alpha{k - 1} - \nu,
	\end{equation}
	and the sequence $\alpha_k$ reaches $\tfrac{5}{2}$ in finitely many steps, and then becomes constant. Thus we get, by the same argument
	\begin{align}
		\left( \frac{(n-2)^2}{4} - \alpha_k^2 \right) \ \|r^{- 1 + \alpha_k}  \nabla \Phi \|_{L^2 (X - B_l)}^2 &\lesssim \ \|r^{- 1 + \nu + \alpha_k}  \nabla \Phi \|_{L^2 (X - B_l)}^2  + \| \nabla \Phi \|_{L^2 (X - B_l)}^2 \\
		&\lesssim \ \|r^{- 1 + \alpha_{k - 1}}  \nabla \Phi \|_{L^2 (X - B_l)}^2  + \| \nabla \Phi \|_{L^2 (X - B_l)}^2.
	\end{align}
	Thus an induction proves \cref{ineq:nablaphi_bound_0} with $\alpha = \alpha_k$, for all $k \in \N$. Writing $\epsilon = 5 - 2 \alpha$ gives \cref{ineq:nablaphi_bound}.\\
	Finally, similar to \Cref{lem:Agmon_first_step}, using \cref{ineq:Agmon} again, with $f = r_{l, L}^\alpha$, we get
	\begin{align}
		\|\nabla (r_{l, L}^\alpha \nabla \Phi) \|_{L^2 (X - B_l)}^2 &\leqslant \| \rd r_{l, L}^\alpha \otimes \nabla \Phi \|_{L^2 (X - B_l)}^2 + c \int\limits_{X - B_l} r_{l, L}^{2 \alpha} e^{-cmr} |\nabla \Phi|^2 \vol_X \\
		&\lesssim \| r^{\alpha - 1} \nabla \Phi \|_{L^2 (X - B_l)}^2 + \|r^{- 1 + \nu} r_{l, L}^\alpha \nabla \Phi \|_{L^2 (X - B_l)}^2 + \|\nabla \Phi \|_{L^2 (X - B_l)}^2.
	\end{align}
	Now the right hand side is bounded for all $L$ when $\alpha < \tfrac{n - 2}{2}$. Thus $\nabla(r^\alpha \nabla \Phi)$ is square integrable, and hence so is $r^\alpha \nabla^2 \Phi = \nabla(r^\alpha \nabla \Phi) - \alpha r^{\alpha - 1} \rd r \otimes \nabla \Phi$.
\end{proof}

Now we are ready to prove \Cref{prop:Decay_1}.

\begin{proof}[Proof of \Cref{prop:Decay_1}]
	Pick $x \in X$ and let $R = \tfrac{1}{4} \min \left(r (x), \mathrm{inj}_x (X, g)\right)\approx r (x)$. It follows from \Cref{prop:Exponential_Decay_Transverse} and \cref{ineq:improv_3} that
  	\begin{equation}
  		\Delta |\nabla\Phi|^2 \lesssim e^{-cmr}|\nabla\Phi|^2.
  	\end{equation}
	For all $\epsilon > 0$, Moser iteration gives
	\begin{equation}\label{ineq:moser_iteration}
		|\nabla \Phi (x)|^2	\lesssim R^{- n} \int\limits_{B_R (x)} |\nabla \Phi|^2 \vol_X \lesssim r(x)^{- 2(n - 2) + \epsilon} \int\limits_{B_R (x)} |r^{(n-4)/2-\epsilon} \nabla \Phi|^2 \vol_X,
	\end{equation} where in the last inequality we used that $R\approx r(x)$. By \Cref{lem:First_time_Hardy}, the last integral in the right-hand side of \cref{ineq:moser_iteration} can be bounded by $\tfrac{C}{\varepsilon}$, yielding the stated result.
\end{proof}

\smallskip

\subsection{Bounds from an improved Hardy inequality}\label{subsec:Hardy2}

This section follows the same strategy of the previous except that we combine the previously obtained bound with an improved Hardy-inequality which holds for $H^1$-functions supported along an end $X - B_l$. We summarize the main result as follows.

\begin{proposition}\label{prop:Decay_2}
	Under the hypothesis of \Cref{thm:First_part_of_Main_Theorem_2}, there are constants $c_1, c_2 > 0$ such that for all $t \in (0, 1)$ and $\alpha < \tfrac{n - 1}{2}$, we have
	\begin{equation}\label[ineq]{ineq:Inequality_Decay_2}
		\| r^{\alpha - 1/2} \nabla \Phi \|_{L^2 (X)}^2 \leqslant c_1 + c_2 (n - 1 - 2 \alpha)^{- 1}.
	\end{equation}
	In particular,
	\begin{equation}
		|\nabla \Phi|^2 = O(r^{-2(n-1) + \epsilon})\quad\text{as $r\to\infty$} \label{eq:Psi_bound_nonoptimal}
	\end{equation}
	for all $\epsilon > 0$.
\end{proposition} 

We start with the statement of the improved Hardy inequality.

\begin{lemma}[Improved Hardy's Inequality as in \cite{degeratu2013witten}]\label{lem:Hardy_2}
	Let $(X, \varphi)$ be an AC $\rG_2$-manifold with rate $\nu < 0$. Then, there is a constant $C_H>0$ such that for all $\xi \in H^1$ with support in $X - B_l$
	\begin{equation}
		\| \nabla_{\partial_r} \xi \|_{L^2 (X - B_l)}^2 \geqslant \left( \frac{n-2}{2} \right)^2 \| r^{- 1} \xi \|_{L^2 (X - B_l)}^2 - C_H \|r^{-1 + \nu} \xi \|_{L^2 (X - B_l)}^2 . \label[ineq]{ineq:Hardy_2}
	\end{equation}
\end{lemma}

The proof of this result is exactly the same as that in \cite{degeratu2013witten} and so we jump into the proof of the main result of this section.

\begin{proof}[Proof of \Cref{prop:Decay_2}]
	Let $r_{l, L}$ as in \cref{eq:rlL_def} and $\alpha > 0$, to be determined later. Applying the improved Hardy \cref{ineq:Hardy_2} to the function $r_{l, L}^\alpha\sqrt{r} |\nabla \Phi|$ and using Kato's inequality gives us
	\begin{equation}
		\frac{(n-2)^2}{4} \ \| r_{l, L}^\alpha r^{-1/2} \nabla \Phi \|_{L^2 (X - B_l)}^2 \leqslant \| \nabla_{\partial_r} (r_{l, L}^\alpha \sqrt{r} \nabla \Phi ) \|_{L^2 (X - B_l)}^2 + C_H \|r_{l, L}^\alpha r^{-1/2 + \nu} \nabla \Phi \|_{L^2 (X - B_l)}^2, \label[ineq]{ineq:Hardy_rewrite_1}
	\end{equation}
	which we now combine with the previously used strategy, now with the goal of proving \Cref{prop:Decay_2}.

	Note that there is a number $C = C (l, \nabla, \Phi)$, such that
	\begin{equation}
		\| r^{\alpha - 1/2} \nabla \Phi \|_{L^2 (X)}^2 = C + \lim\limits_{L \rightarrow \infty} \| r_{l, L}^{\alpha - 1/2} \nabla \Phi \|_{L^2 (X)}^2,
	\end{equation}
	so it is enough to show \cref{ineq:Inequality_Decay_2} with the left-hand side replaced by $\| r_{l, L}^{\alpha - 1/2} \nabla \Phi \|_{L^2 (X)}^2$.

	First we prove that for all $\alpha < \tfrac{n - 1}{2}$, $\| r_{l, L}^{\alpha - 1/2} \nabla \Phi \|_{L^2 (X)}$ stays bounded as $L \rightarrow \infty$. We use proof by contradiction. Thus let us assume that $\alpha$ is such that $\alpha < \tfrac{n - 1}{2}$ and
	\begin{equation}
		\lim\limits_{L \rightarrow \infty} \| r_{l, L}^{\alpha - 1/2} \nabla \Phi \|_{L^2 (X)}^2 = \infty. \label{eq:PbC}
	\end{equation}
	Let $\nabla^\Sigma$ denote covariant differentiation, using the connection $\nabla$, in the directions along the kernel of $\rd r$, so that $\nabla = \nabla_{\partial_r} \otimes \rd r + \nabla^\Sigma$. Note that $\nabla^\Sigma r = 0$, and thus $\nabla^\Sigma r_{l, L} = 0$. Then, using the computation in the proof of \Cref{lem:Agmon_first_step}, we have
	\begin{align}
		\|	\nabla \left( r_{l, L}^\alpha \sqrt{r} \nabla \Phi \right)\|_{L^2 (X - B_l)} 
		&= \| \nabla_{\partial_r} \left( r_{l, L}^\alpha \sqrt{r} \nabla \Phi \right) \|_{L^2 (X - B_l)}^2 + \| r_{l, L}^\alpha \sqrt{r} \nabla^\Sigma \nabla \Phi \|_{L^2 (X - B_l)}^2 \\
		&\lesssim \| \rd \left( r_{l, L}^\alpha \sqrt{r} \right) \otimes \nabla \Phi \|_{L^2 (X - B_l)} + \int\limits_{X - B_l} r_{l, L}^{2 \alpha} r e^{-cmr} |\nabla \Phi|^2 \vol_X. \label[ineq]{ineq:Hardy_rewrite_2}
	\end{align}
	 Combining \cref{ineq:Hardy_rewrite_1,ineq:Hardy_rewrite_2}, we get
	\begin{align}
		\frac{(n-2)^2}{4} \ \| r_{l, L}^\alpha r^{-1/2} \nabla \Phi \|_{L^2 (X - B_l)}^2 + \| r_{l, L}^\alpha \sqrt{r} \nabla^\Sigma \nabla \Phi \|_{L^2 (X - B_l)}^2 &\lesssim \| \rd (r_{l, L}^\alpha \sqrt{r}) \otimes \nabla \Phi \|_{L^2 (X)}^2 \\
			& \quad +\|r_{l, L}^\alpha r^{-1/2 + \nu} \nabla \Phi \|_{L^2 (X - B_l)}^2 \\
			& \quad +O (1), \label[ineq]{ineq:Intermediate_Improved_Main_Computation}
	\end{align}
	where $O (1)$ corresponds to the rightmost term of the right hand side of \cref{ineq:Hardy_rewrite_2}, which is bounded independently of $L$ for any fixed $\alpha$.

	Now let
	\begin{align}
		t_L = t_L (l, \alpha) &:= \frac{\| r_{l, L}^\alpha r^{-1/2} \nabla \Phi \|_{L^2 (B_L - B_l)}^2}{\| r_{l, L}^\alpha r^{-1/2} \nabla \Phi \|_{L^2 (X - B_l)}^2} \in (0, 1), \label{eq:t_L_def} \\
		t = t (l, \alpha) &:= \limsup\limits_{L \rightarrow \infty} t_L \in (0, 1]. \label{eq:t_def}
	\end{align}
	Using \cref{eq:rlL_def,eq:t_def}, we get, as $L \rightarrow \infty$, that
	\begin{equation}
		\| \rd (r_{l, L}^\alpha \sqrt{r}) \otimes \nabla \Phi \|_{L^2 (X)}^2 = \left( (\alpha^2 + \alpha) t_L + \frac{1}{4} \right) \ \| r_{l, L}^\alpha r^{- 1/2} \nabla \Phi \|_{L^2 (X)}^2 + O (1), \label{eq:Hardy_1st_term}
	\end{equation}
	where this $O (1)$ term depends only on $l, \alpha$ and the $L^2 (X)$-norm of $\nabla \Phi$ in $B_{2l} - B_l$; in particular, for any fixed $\alpha$, it remains bounded as $L \to \infty$.

	Next, we obtain a lower bound for the term $\|r_{l, L}^\alpha \sqrt{r}\nabla^\Sigma \nabla \Phi \|_{L^2 (X)}^2$ appearing in the left hand side of \cref{ineq:Intermediate_Improved_Main_Computation}. First, we compute
	\begin{equation}
		0 = \int\limits_X \rd \left( r_{l, L}^{2 \alpha} |\nabla \Phi|^2 \iota_{\partial_r} (\vol_X) \right) = \int\limits_{X - B_l} \left( \partial_r \left( r_{l, L}^{2 \alpha} \right) - r_{l, L}^{2 \alpha} \Delta r) |\nabla \Phi|^2 + r_{l, L}^{2 \alpha} \partial_r (|\nabla \Phi|^2) \right) \vol_X. \label{eq:Agmon--Stern_trick_1}
	\end{equation}
	The first equality follows from an extension of Stokes' theorem due to \cite{Gaffney54}, which states that for an orientable complete Riemannian $n$-manifold $(M^n,g)$, one has $\int_M \rd\gamma = 0$ provided that $|\gamma|, |\rd \gamma| \in L^1 (X)$. Here $\gamma := r_{l, L}^{2 \alpha} |\nabla \Phi|^2 \iota_{\partial_r} (\vol_X)$ satisfies these conditions. In the second equality, we use the fact that $r_{l, L}|_{B_l} = 0$ and $\rd (\iota_{\partial_r} \vol_X) = - (\Delta r) \vol_X$. Using \cref{eq:drlL_def}, we can further expand and rearrange \cref{eq:Agmon--Stern_trick_1}, to get that there is $C_1 = O \left( \alpha l^{2 \alpha} \| \nabla \Phi \|_{L^2 (B_{2l} - B_l)}^2 \right)$, bounded independently of $L$, such that
	\begin{equation}
		\int\limits_{X - B_l} \left( \alpha t_L r_{l, L}^{2 \alpha} r^{- 1} - \frac{1}{2} r_{l, L}^{2 \alpha} \Delta r) |\nabla \Phi|^2 \right) \vol_X = - \int\limits_{X - B_l} r_{l, L}^{2 \alpha} \langle \nabla \Phi, \nabla_{\partial_r} \nabla \Phi \rangle \: \vol_X + C_1. \label{eq:Agmon--Stern_trick_2}
	\end{equation}
	Let us rewrite the right hand side of \cref{eq:Agmon--Stern_trick_2} first. Using a normal frame, whose first element is $\partial_r$, and the harmonicity of $\Phi$, we show that
	\begin{align}
		- \langle \nabla \Phi, \nabla_{\partial_r} \nabla \Phi \rangle &= - \langle \nabla_{\partial_r} \Phi, \nabla_{\partial_r} \nabla_{\partial_r} \Phi \rangle - \sum\limits_{i = 2}^n \langle \nabla_i \Phi, \nabla_{\partial_r} \nabla_i \Phi \rangle \\
			&= \sum\limits_{i = 2}^n \left( \langle \nabla_{\partial_r} \Phi, \nabla_i \nabla_i \Phi \rangle - \langle \nabla_i \Phi, \nabla_i \nabla_{\partial_r} \Phi + [F_{ri}, \Phi] \rangle \right) \\
			&\leqslant \left| \nabla_{\partial_r} \Phi \right| \sum\limits_{i = 2}^n \left| \nabla_i \nabla_i \Phi \right| + \sum\limits_{i = 2}^n  \left| \nabla_i \Phi \right| \left| \nabla_i \nabla_{\partial_r} \Phi \right| - \sum\limits_{i = 2}^n \langle \nabla_i \Phi, [F_{ri}, \Phi] \rangle \\
			&\leqslant \left| \nabla_{\partial_r} \Phi \right| \sqrt{n - 1} \left( \sum\limits_{i = 2}^n \left| \nabla_i \nabla_i \Phi \right|^2 \right)^\frac{1}{2} + \left| \nabla^\Sigma \Phi \right| \left( \sum\limits_{i = 2}^n \left| \nabla_i \nabla_{\partial_r} \Phi \right|^2 \right)^\frac{1}{2} \\
			& \quad - \sum\limits_{i = 2}^n \langle \nabla_i \Phi, [F_{ri}, \Phi] \rangle.
	\end{align}
	Using that $n - 1 = 6 > 1$ and the Cauchy--Schwarz inequality, we get
	\begin{align}
		- \langle \nabla \Phi, \nabla_{\partial_r} \nabla \Phi \rangle &\leqslant \sqrt{n - 1} \left( \left| \nabla_{\partial_r} \Phi \right| \left( \sum\limits_{i = 2}^n \left| \nabla_i \nabla_i \Phi \right|^2 \right)^\frac{1}{2} + \left| \nabla^\Sigma \Phi \right| \left( \sum\limits_{i = 2}^n \left| \nabla_i \nabla_{\partial_r} \Phi \right|^2 \right)^\frac{1}{2} \right) \\
		& \quad - \sum\limits_{i = 2}^n \langle \nabla_i \Phi, [F_{ri}, \Phi] \rangle \\
			&\leqslant \sqrt{n - 1} |\nabla \Phi| |\nabla^\Sigma \nabla \Phi| - \sum\limits_{i = 2}^n \langle \nabla_i \Phi, [F_{ri}, \Phi] \rangle.
	\end{align}
	Next, using that $F_{ri} = F_{ri}^{14} + F_{ri}^7$ and $F_{ri}^7 = \sum_{j = 2}^n \varphi_{rij} \nabla_j \Phi$ (which we get by the $\rG_2$-monopole \cref{eq:Monopole}), we get that
	\begin{align}
		- \sum\limits_{i = 2}^n \langle \nabla_i \Phi, [F_{ri}, \Phi] \rangle	&= - \sum\limits_{i = 2}^n \langle \nabla_i \Phi, [F_{ri}^{14}, \Phi] \rangle - \sum\limits_{i = 2}^n \sum\limits_{j = 2}^n \varphi_{rij} \langle \nabla_i \Phi, [\nabla_j \Phi, \Phi] \rangle \\
			&= \sum\limits_{i = 2}^n \langle [\nabla_i \Phi, \Phi], F_{ri}^{14} \rangle - \sum\limits_{j = 2}^n \varphi_{rij} \langle \nabla_i \Phi, [(\nabla_j \Phi)^\perp, \Phi] \rangle \\
			&= \sum\limits_{i = 2}^n \langle [(\nabla_i \Phi)^\perp, \Phi], F_{ri}^{14} \rangle + \sum\limits_{i = 2}^n \sum\limits_{j = 2}^n \varphi_{rij} \langle [\nabla_i \Phi, \Phi], (\nabla_j \Phi)^\perp \rangle \\
			&= \sum\limits_{i = 2}^n \langle [(\nabla_i \Phi)^\perp, \Phi], F_{ri}^{14} \rangle + \sum\limits_{i = 2}^n \sum\limits_{j = 2}^n \varphi_{rij} \langle [(\nabla_i \Phi)^\perp, \Phi], (\nabla_j \Phi)^\perp \rangle \\
			&\leqslant O \left( |(\nabla \Phi)^\perp| \left( |F_{ri}^{14}| + |(\nabla \Phi)^\perp| \right) \right).
	\end{align}
	Hence finally, we get
	\begin{equation}
		- \langle \nabla \Phi, \nabla_{\partial_r} \nabla \Phi \rangle \leqslant \sqrt{n - 1} |\nabla \Phi| |\nabla^\Sigma \nabla \Phi| + O \left( |(\nabla \Phi)^\perp| |F_{ri}^{14}| + |(\nabla \Phi)^\perp|^2 \right).
	\end{equation}
	Thus, using also \Cref{prop:Exponential_Decay_Transverse}, for the right hand side of \cref{eq:Agmon--Stern_trick_2}, we get that (for another $C_2>0$, bounded independently of $L$ or $\alpha$)
	\begin{equation}
		- \int\limits_{X - B_l} r_{l, L}^{2 \alpha} \langle \nabla \Phi, \nabla_{\partial_r} \nabla \Phi \rangle \: \vol_X \leqslant \sqrt{n - 1} \| r_{l, L}^\alpha r^{- 1/2} \nabla \Phi \|_{L^2 (X)} \| r_{l, L}^\alpha \sqrt{r} \nabla^\Sigma \nabla \Phi \|_{L^2 (X)} + C_2. \label[ineq]{ineq:Agmon--Stern_trick_4}
	\end{equation}
	Next we estimate the left hand side of \cref{eq:Agmon--Stern_trick_2}. Recall that for some $\nu^\prime > 0$, the radial function $r$ satisfies
	\begin{equation}
		\frac{n - 1}{r} - O \left( r^{- 1 - \nu^\prime} \right) \leqslant - \Delta r \leqslant \frac{n - 1}{r},
	\end{equation}
	with the lower bound following from the assumption that $(X,\varphi)$ is AC and the upper bound from the Laplacian comparison theorem together with the Ricci-flatness. Using this, we get (for some $C_3 > 0$) that
	\begin{equation}
		\int\limits_{X - B_l} \left( \alpha t_L r_{l, L}^{2 \alpha} r^{- 1} - \frac{1}{2} r_{l, L}^{2 \alpha} \Delta r) |\nabla \Phi|^2 \right) \vol_X \geqslant \left( \alpha t_L + \frac{n - 1}{2} - C_3 l^{- \nu^\prime} \right) \ \| r_{l, L}^\alpha r^{- 1/2} \nabla \Phi \|_{L^2 (X)}^2. \label[ineq]{ineq:Agmon--Stern_trick_5}
	\end{equation}
	Combining \cref{eq:Agmon--Stern_trick_2,ineq:Agmon--Stern_trick_4,ineq:Agmon--Stern_trick_5}, we get
	\begin{multline}
		\sqrt{n - 1} \| r_{l, L}^\alpha r^{- 1/2} \nabla \Phi \|_{L^2 (X)} \| r_{l, L}^\alpha \sqrt{r} \nabla^\Sigma \nabla \Phi \|_{L^2 (X)} + C_2 \geqslant \\ \left( \alpha t_L + \frac{n - 1}{2} - C_3 l^{- \nu^\prime} \right) \ \| r_{l, L}^\alpha r^{- 1/2} \nabla \Phi \|_{L^2 (X)}^2.
	\end{multline}
	Let us subtract $C_2$, divide by $\sqrt{n - 1} \| r_{l, L}^\alpha r^{- 1/2} \nabla \Phi \|_{L^2 (X)}$, and square to get
	\begin{multline}
		\| r_{l, L}^\alpha \sqrt{r} \nabla^\Sigma \nabla \Phi \|_{L^2 (X)}^2 \geqslant \\ \frac{1}{n - 1} \left( \alpha t_L + \frac{n - 1}{2} - C_3 l^{- \nu^\prime} \right)^2 \ \| r_{l, L}^\alpha r^{- 1/2} \nabla \Phi \|_{L^2 (X)}^2 \\
		- \frac{2 C_2 \left( \alpha t_L + \frac{n - 1}{2} - C_3 l^{- \nu^\prime} \right)}{\sqrt{n - 1}} + \frac{1}{(n - 1) \| r_{l, L}^\alpha r^{- 1/2} \nabla \Phi \|_{L^2 (X)}^2}.
	\end{multline}
	By \cref{eq:PbC}, we last term is bounded above for large $L$, and thus we get
	\begin{equation}\label[ineq]{ineq:Hardy_2st_term_LHS}
		\frac{1}{n - 1} \left( \alpha t_L + \frac{n - 1}{2} - C_3 l^{- \nu^\prime} \right)^2 \| r_{l, L}^\alpha r^{- 1/2} \nabla \Phi \|_{L^2 (X)}^2 \leqslant \| r_{l, L}^\alpha \sqrt{r} \nabla^\Sigma \nabla \Phi \|_{L^2 (X)}^2 + O (1).
	\end{equation}
	Hence inserting \cref{eq:Hardy_1st_term,ineq:Hardy_2st_term_LHS} into \cref{ineq:Intermediate_Improved_Main_Computation}, gives
	\begin{eqnarray}\label[ineq]{ineq:Hardy_final}
		\left( \frac{(n - 2)^2}{4} - \left( (\alpha^2 + \alpha) t + \frac{1}{4} \right) + \frac{1}{n - 1} \left( \alpha t_L + \frac{n - 1}{2} - C_3 l^{- \nu^\prime} \right)^2 \right) \| r_{l, L}^\alpha r^{- 1/2} \nabla \Phi \|_{L^2 (X)}^2\nonumber\\
		\leqslant \| r_{l, L}^\alpha r^{- 1/2+\nu} \nabla \Phi \|_{L^2 (X)}^2 + O (1).
	\end{eqnarray}
	Thus $r_{l, L}^\alpha r^{- 1/2} \nabla \Phi $ is bounded in $L^2 (X)$ independently of $L$, as long as $r^{\alpha-1/2+\nu}\nabla\Phi\in L^2(X)$ and the quantity in the parentheses above is positive, this last condition being equivalent in the $L \rightarrow \infty$ limit to
	\begin{equation}
		\left( \alpha t + \frac{c l^{- \nu^\prime}}{n - 1} \right)^2 < \frac{(n - 1)^2}{4} + \frac{n c^2 l^{- 2 \nu^\prime}}{(n - 1)^2} - c (n - 1) l^{- \nu^\prime}.
	\end{equation}
	Since $t \in (0, 1]$, the above inequality is definitely satisfied for $\alpha = \tfrac{n - 1}{2} - O \left( l^{- \nu^\prime/2} \right)$. For such $\alpha$, in case $\nu \leqslant -1$, by \Cref{lem:First_time_Hardy} we do have that $r^{\alpha-1/2+\nu}\nabla\Phi\in L^2(X)$ and therefore that $r^{\alpha - 1/2} \nabla \Phi\in L^2 (X)$. In case $- 1 < \nu < 0$ we can still achieve the desired integrability by a finite iteration of the above argument, in the same way as we did in the final part of the proof of \Cref{lem:First_time_Hardy}. This finishes the proof by contradiction argument and, in fact, yields \cref{ineq:Inequality_Decay_2}.

	Combining this with Moser iteration (as in the proof of \Cref{prop:Decay_1}), we get the bound
	\begin{equation}
		|\nabla \Phi|^2 \leqslant C r^{- 2 (n-1) + O \left( l^{- \nu^\prime/2} \right)}.
	\end{equation}
	Using this, we can bootstrap the previous computation, since for $l$ large (but finite), we can now replace \cref{ineq:Agmon--Stern_trick_5} with
	\begin{align}
		\int\limits_{X - B_l} \left( \alpha t r_{l, L}^{2 \alpha} r^{- 1} - \frac{1}{2} r_{l, L}^{2 \alpha} \Delta r) |\nabla \Phi|^2 \right) \vol_X &\geqslant \left( \alpha t + \frac{n - 1}{2} \right) \ \| r_{l, L}^\alpha r^{- 1/2} \nabla \Phi \|_{L^2 (X)}^2 \\
		& \quad - c \| r_{l, L}^\alpha r^{- (1+\nu')/2} \nabla \Phi \|_{L^2 (X)}, \label[ineq]{ineq:Agmon--Stern_trick_5b}
	\end{align}
	as long as $\alpha < \tfrac{n - 1}{2}$. After going through the remaining steps in an analogous way, we get the improved inequality
	\begin{equation}
		\left( \frac{(n - 2)^2}{4} - \left( (\alpha^2 + \alpha) t + \frac{1}{4} \right) + \frac{1}{n - 1} \left( \alpha t + \frac{n - 1}{2} \right)^2 \right) \| r_{l, L}^\alpha r^{- 1/2} \nabla \Phi \|_{L^2 (X)}^2 \leqslant O (1),
	\end{equation}
	which implies that, as long as $\alpha < \tfrac{n - 1}{2}$, we have
	\begin{equation}
		\| r^{\alpha - 1/2} \nabla \Phi \|_{L^2 (X)}^2 \leqslant c_1 + \lim\limits_{L \rightarrow \infty} \| r_{l, L}^\alpha r^{- 1/2} \nabla \Phi \|_{L^2 (X)}^2 \leqslant c_1 + c_2 (n - 1 - 2 \alpha t)^{- 1}.
	\end{equation}
	Using Moser iteration again, together with $t \in (0, 1)$, we get that on $X - B_l$ and for all $\epsilon > 0$
	\begin{equation}
		|\nabla \Phi| \leqslant \frac{C_l}\epsilon r^{- (n - 1) + \epsilon},
	\end{equation}
	which proves \cref{eq:Psi_bound_nonoptimal} and thus concludes the proof.
\end{proof}

\smallskip

\subsection{The final estimate}
\label{sec:final_est}

In this section we finish the proof of \Cref{thm:First_part_of_Main_Theorem_2}, giving a proof of its part (ii).

We start by recalling \cref{ineq:Inequality_Decay_2} stated in \Cref{prop:Decay_2}. This states that there are positive constants $c_1$ and $c_2$ such that
\begin{equation}
	\| r^{\alpha - 1/2} \nabla \Phi \|_{L^2 (X)}^2 \leqslant c_1 + c_2 (n - 1 - 2 \alpha)^{- 1}.
\end{equation}
Let $i \in \mathbb{N}$ be such that $i \gg c_1+c_2$ and choose $\alpha = \tfrac{n - 1}{2} - \tfrac{1}{i}$ and $t = 1 - \tfrac{1}{i}$. Then we have
\begin{equation}
	\int\limits_X |r^{\frac{n-2}{2}-\frac{1}{i}} \nabla \Phi|^2 \vol_X \lesssim i,
\end{equation}
and given that in $\{ x \in X : 1 \leqslant r (x) \leqslant R^i \}$ we have $r^{-1/i} \geqslant R^{-1}$ we find that
\begin{equation}
	\int\limits_{\{ x \in X : 1 \leqslant r (x) \leqslant R^i \}} |r^{\frac{n-2}{2}} \nabla \Phi |^2 \vol_X \lesssim i R.
\end{equation}
This may be also regarded as the average of the $L^2$-norm of $r^{\frac{n-2}{2}} \nabla \Phi $ on the $i$-cylinders $C_k = \{ x \in X : R^k \leqslant r (x) \leqslant R^{k + 1} \}$ for $k = 0, 1, \ldots, i - 1$, which is therefore uniformly bounded independently of $i$. As a consequence, there must be an infinite sequence of cylinders $\lbrace C_{i_j} \rbrace_j$ with $i_j \nearrow \infty$ for which 
\begin{equation}
	\int\limits_{C_{i_j}} | \nabla \Phi |^2 \vol_X \lesssim R r^{-(n-2)}.
\end{equation}
Then, using Moser iteration in a ball $B_\rho(x_{i_j}) \subseteq C_{i_j}$ gives that
\begin{equation}
	|\nabla \Phi (x_{i_j})|^2 \lesssim \rho^{-n} \int\limits_{B_\rho (x_{i_j})} |\nabla \Phi|^2 \vol_X \lesssim \rho^{-n} Rr^{-(n-2)}.
\end{equation}
In case, $R >2$ we can set $\rho = R^{i_j-1}$. Then, $R^2 r^{- 1} \geqslant \rho^{-1} \geqslant R r^{- 1}$ in $C_{i_j}$ and the Moser iteration above yields a bound at all $x_{i_j}$ in the sub-cylinders $C_{i_j}' = \{ x \in X : R^{i_j}+R^{i_j-1} \leqslant r (x) \leqslant R^{i_j+1}-R^{i_j-1} \} \subseteq C_{i_j}$ of the form
\begin{equation}\label[ineq]{ineq:Decay_x_ni}
	\sup_{x_{i_j} \in C_{i_j}' } |\nabla \Phi (x_{i_j})|^2 \lesssim r^{-n}(R^2-1)^n Rr^{-(n-2)} \lesssim R(R^2-1)^n r^{-2(n-1)}.
\end{equation}
We have thus obtained a sequence of cylinders $C_{i_j}'$ going off to infinity and where the \cref{ineq:Decay_x_ni} holds. As a consequence of this, if it was not true that $|\nabla \Phi |^2 = O(r^{-2(n-1)})$ for $r \gg 1$, there must a sequence $y_i$ with $r(y_i) \nearrow \infty$ at which $r^{2(n-1)}|\nabla \Phi |^2$ attains a local maxima and 
\begin{equation}
	\limsup_{i \to \infty} r(y_i)^{2(n-1)} |\nabla \Phi (y_i) |^2 = \infty.
\end{equation}
From the condition that $r^{2(n-1)}|\nabla \Phi |^2$ attains a local maxima at the $y_i$ we find, by differentiating in the $r$-direction, that at $y_i$
\begin{equation}
	2(n-1) |\nabla \Phi|^2 + r \partial_r \left( |\nabla \Phi|^2 \right) = 0.
\end{equation}
On the other hand, the second derivative test at $y_i$ yields
\begin{align}
	0	&\leqslant \Delta \left( r^{2(n-1)} |\nabla \Phi |^2 \right) \\
		&= - (2n-2) (3n-4) r^{2(n-2)} |\nabla \Phi |^2 - (2n-2) r^{2n-3} \partial_r \left( |\nabla \Phi|^2 \right) + r^{2(n-1)} \Delta |\nabla \Phi|^2 + \ldots ,
\end{align}
with the $\ldots$ denotes lower order terms in comparison to $r^{2(n-2)}|\nabla\Phi|^2$. Inserting the first order equation above we find that
\begin{align}
	0	&\leqslant - (2n-2) (3n-4) r^{2(n-1)} |\nabla \Phi |^2 + (2n-2)^2 r^{2(n-1)} |\nabla \Phi|^2 + r^{2(n-1)} \Delta |\nabla \Phi|^2 + \ldots \\
		&= - 2(n-1)(n-2) r^{2(n-2)} |\nabla \Phi |^2 + r^{2(n-1)} \Delta |\nabla \Phi|^2 + \ldots ,
\end{align}
which we can rewrite as
\begin{equation}
	2(n-1)(n-2) r^{-2} |\nabla \Phi |^2 \leqslant \Delta |\nabla \Phi|^2 + \ldots .
\end{equation}
Thus, using the improved Bochner inequality in \Cref{lem:Improved_Bochner_Delta_Nabla_Phi} we find that
\begin{equation}
	|\nabla \Phi|^2 \lesssim r^2 |(F_\nabla)^\perp|^2 |(\nabla \Phi)^{||}|^2 + o (|\nabla \Phi|^2).
\end{equation}
Given that $(F_\nabla)^\perp$ decays exponentially, as shown in \Cref{prop:Exponential_Decay_Transverse}, this is impossible unless $|\nabla \Phi(y_i)| = 0$ which would contradict $r^{2(n-1)}|\nabla \Phi|^2(y_i) \nearrow \infty$. This completes the proof of \Cref{thm:First_part_of_Main_Theorem_2}.

\smallskip

\subsection{Asymptotic decay rate of $m^2-|\Phi|^2$}\label{sec:decay_rate_Phi}

In this brief subsection we state and prove an easy consequence of part (ii) of \Cref{thm:First_part_of_Main_Theorem_2}, giving the precise polynomial asymptotic decay rate of $m^2-|\Phi|^2$ on this AC case.

\begin{corollary}\label{cor:refined_AC_asymp_Phi}
	Under the hypothesis of \Cref{thm:First_part_of_Main_Theorem_2}, let $m$ be the mass of $(\nabla,\Phi)$, given by \Cref{thm:Finite_Mass}. Then along the end of $X$ we have
	\begin{equation}
		m^2-|\Phi|^2 \sim r^{2-n}.
	\end{equation}
\end{corollary}

\begin{proof}
	By \Cref{cor:Asymptotics_Phi_Square} and \Cref{rem:AC_vol_growth}, we already know that $m^2-|\Phi|^2\gtrsim r^{2-n}$. To prove that $m^2-|\Phi|^2\lesssim r^{2-n}$, we first recall from the result of \Cref{thm:Finite_Mass} that
	\begin{equation}
		(m^2-|\Phi|^2)(x) = 2\int_X G(x,y)|\nabla\Phi|^2(y)\vol_X(y),
	\end{equation}
	and then conclude by using the sharp decay estimate $|\nabla\Phi|^2\lesssim r^{-2(n-1)}$ of part (ii) of \Cref{thm:First_part_of_Main_Theorem_2} combined with the Green's function behavior $G(x,y)\sim\dist(x,y)^{2-n}$, as in the argument in \cite{van2009regularity}*{(first part of the proof of) Theorem~2.30}.
\end{proof}

\smallskip

\subsection{Bounds on higher order covariant derivatives of $\Phi$}

In this subsection, we assume the hypothesis of \Cref{thm:First_part_of_Main_Theorem_2}. The arguments here follow from an inductive procedure based on \cite{stern2010geometry}*{Section~2}.

\begin{lemma}\label{lem:higher_derivatives_l2}
    $\nabla^{j+1}\Phi\in L^2 (X)$ for all $j\in\mathbb{N}$.
\end{lemma}
\begin{proof}
    By assumption $\nabla\Phi\in L^2 (X)$, and by \Cref{lem:Agmon_first_step} we know that $\nabla^2\Phi\in L^2 (X)$. We now induct. Suppose that $\nabla^i(\nabla\Phi)\in L^2 (X)$ for all $i \leqslant j$. Recall from either \Cref{cor:allderivativesdecay,cor:Lpbounds} that $\nabla^i F_\nabla,\nabla^{i+1}\Phi\in L^\infty(X)$ for all $i$, as well as $\|\Phi\|_{L^\infty} \leqslant m$ by \Cref{thm:Finite_Mass} and \Cref{rem:Finite_Mass}. Let $\{f_k:X\to [0,1]\}$ be a sequence of uniformly $C^j$-bounded functions, with $\lim_{k\to\infty} f_k = 1$ pointwise. Differentiating the Bochner identity of \cref{eq:Bochner_identity_nablaPhi} $j$ times and taking the $L^2 (X)$-inner product with $f_k^2\nabla^j(\nabla\Phi)$ gives:
    \begin{align}
        0 &= \langle \nabla^j\left(\nabla^\ast\nabla + 2\ast[\ast F_\nabla\wedge\cdot{}] + \text{ad}(\Phi)^2(\cdot{})\right)(\nabla\Phi),  f_k^2\nabla^j(\nabla\Phi)\rangle_{L^2 (X)}\\
        &\geqslant \|\nabla(f_k\nabla^j(\nabla\Phi))\|_{L^2 (X)}^2 - \|\rd f_k\otimes\nabla^j(\nabla\Phi)\|_{L^2 (X)}^2 - \sum\limits_{0 \leqslant l \leqslant j} c_l \|\nabla^l(\nabla\Phi)\|_{L^2 (X)}^2,
    \end{align} for constants $c_l>0$ that are determined by the sup norms of $\nabla^i F_\nabla,\nabla^{i+1}\Phi$, $i \leqslant j$, the mass $m \geqslant 0$ of $(\nabla,\Phi)$, and the other coefficients of $[\nabla^j,\nabla^\ast]$. Taking the limit as $k\to\infty$ gives $\nabla^{j+1}(\nabla\Phi)\in L^2 (X)$.  
\end{proof}

\begin{corollary}\label{cor:higher_order_der_bounds}
	$\nabla^{j+1}\Phi\in L_1^{p}(X)$ for all $p\in [2,2n]$ and $j\in\mathbb{N}$.
\end{corollary}

\begin{proof}
    Using Kato's inequality we get that if $\nabla^{j+2}\Phi\in L^p(X)$, then $\rd |\nabla^{j+1}\Phi|\in L^p(X)$. Hence, using \Cref{lem:higher_derivatives_l2}, we get $|\nabla^{j+1}\Phi|\in L_1^2(X)$ for all $j\in\mathbb{N}$. Now iterate this argument using the Sobolev Embeddings $L_1^p(X)\hookrightarrow L^{\frac{np}{n-p}}(X)$ valid for all $p<n$, together with H\"older's inequality. 
\end{proof}


\bigskip

\section{Boundary data}\label{sec:Boundary_data}

In this section we prove the second part of \Cref{thm:Main_Theorem_2}.

\begin{theorem}\label{thm:main3}
	Let $(X, \varphi)$ be an irreducible AC $\rG_2$-manifold and $P \rightarrow X$ be a principal $\SU (2)$-bundle. Assume that $(\nabla, \Phi)$ satisfies the second order \cref{eq:2nd_Order_Eq_1,eq:2nd_Order_Eq_2} with $\nabla \Phi \in L^2 (X)$, and either $|F_\nabla|$ decays quadratically along the end or $(\nabla, \Phi)$ is a $\rG_2$-monopole such that $|F_\nabla^{14}|$ decays quadratically along the end.

	Then there is a principal $\SU (2)$-bundle, $P_\infty \rightarrow \Sigma$ and a pair $(\nabla_\infty, \Phi_\infty)$, where $\nabla_\infty$ is a connection on $P_\infty$ and $\Phi_\infty$ is a $\nabla_\infty$-parallel section of the adjoint bundle $\mathfrak{g}_{P_\infty}$ over $\Sigma$, such that $(\nabla, \Phi) |_{\Sigma_R}$ converges uniformly to $(\nabla_\infty , \Phi_\infty)$ as $R \to \infty$. Moreover, if $(\nabla, \Phi)$ is a $\rG_2$-monopole, then $\nabla_\infty$ is a pseudo-Hermitian--Yang--Mills connection on $P_\infty$ with respect to the nearly K\"ahler structure on $\Sigma$.
\end{theorem}

We prove the several statements in this theorem as a result of three Lemmata stated and proved bellow.

\begin{lemma}[Existence of $\nabla_\infty$]\label{lem:Existence_of_A_infinity}
	Under the hypotheses of \Cref{thm:main3}, there is a connection $\nabla_\infty$ on a bundle over the link $\Sigma$ such that $\nabla|_{\Sigma_R}$ converges uniformly to $\nabla_\infty$ as $R \to \infty$.
\end{lemma}

\begin{proof}
	By hypothesis, we know that
	\begin{equation}
		|F_\nabla|_g^2 \lesssim r^{-4}.
	\end{equation}
	For the $\rG_2$-monopole case, in which we merely assume quadratic decay of $F_\nabla^{14}$, the quadratic decay of the full $F_\nabla$ follows from \Cref{thm:First_part_of_Main_Theorem_2} (ii):
	\begin{equation}
		3 |F_\nabla^7|_g^2 = |\nabla \Phi|_g^2 \lesssim r^{-2(n-1)}.
	\end{equation}
	Now consider the cylinders $C_R = \left\{ x \in X : r (x) \in [R, eR] \right\}$ with the conical metric $g_C$ which for large $R$ approximates well the $\rG_2$-metric $g$. Then, we rescale it by $r^{-2}$ to obtain the cylindrical metric 
	\begin{equation}
		h = r^{-2}g_C = \rd t^2 + g_\Sigma,
	\end{equation}
	where $t = \log (r)$. With respect to this translation invariant metric we can identify all the cylinders $C_R$ with $(C = [0,1]_t \times \Sigma,h)$. As from the above, we have
	\begin{equation}
		|\nabla\Phi|_h^2 \lesssim e^{-2nt} \quad \& \quad |F_\nabla|_h^2 \lesssim 1,
	\end{equation}
	we find that the restriction of $\nabla_i = \nabla|_{C_i}$ seen as a connection over $C$ has uniformly bounded curvature with respect to $h$. Thus, Uhlenbeck's compactness results \cite{Uhlenbeck1982a} apply and by possibly passing to a subsequence, $A_i$ convergences modulo gauge, as $i \to \infty$, to a well defined connection $\nabla_\infty$ on $C$.
	
	We must now argue that such a connection is unique and does not depend on the subsequence. For that consider $\nabla_i$ on $C_i$ written in radial gauge with respect to $r$, i.e. $\nabla_i = a_i (r)$ with $a_i (\cdot)$ a 1-parameter family of connections over $\Sigma$ parametrized by $r \in [R, e R]$, then $F_{\nabla_i} = \rd r \wedge \del_r a_i(r) + F_{a_i}(r)$ where $F_{a_i}(r)$ is the curvature of $a_i(r)$ over $\lbrace r \rbrace \times \Sigma$. Using this, we find $|\del_r \nabla_i|_g \leqslant |F_{\nabla_i}|_r \lesssim r^{-2}$ and so
	\begin{equation}
		\int\limits_R^{eR} |\del_r \nabla_i|_g \ \rd r \lesssim R^{-1},
	\end{equation}
	which decays as $R \to \infty$. This not only shows that the limit 
	\begin{equation}
		\nabla_\infty = \lim\limits_{r \to \infty} \nabla (r),
	\end{equation}
	exists, as proves it is independent of the coordinate $r$ and so a pullback from a connection over $\Sigma$. Thus, it agrees with the connection $\nabla_\infty$ obtained as the uniform limit of the $\nabla_i$ which is therefore pullback from $\Sigma$.
\end{proof}

\begin{lemma}[$\nabla_\infty$ is pseudo-Hermitian--Yang--Mills]
	In case $(\nabla, \Phi)$ is a $\rG_2$-monopole, the connection $\nabla_\infty$ from \Cref{lem:Existence_of_A_infinity} is pseudo-Hermitian--Yang--Mills, i.e. it satisfies 
	\begin{equation}
		\Lambda F_{\nabla_\infty} = 0 = F_{\nabla_\infty}^{0,2},
	\end{equation}
	with respect to the nearly K\"ahler structure on $\Sigma$ induced by the conical $\rG_2$-structure.
\end{lemma}

\begin{proof}
	With respect to the metric $g_C$ and the coordinate $r$ we can view 
	\begin{equation}
		(C, g_C) = \left( [1,e]_r \times \Sigma , \rd r^2 + r^2 g_\Sigma \right),
	\end{equation}
	that is, $g_C$ is a conical metric on the fixed cylinder $C$. This metric has holonomy $\rG_2$ and $\varphi_C = \rd r \wedge \omega + \Re(\Omega)$ with $(\omega,\Re(\Omega))$ determining the nearly K\"ahler structure on the cone. Hence, we can decompose the curvature of $\nabla_\infty$ as
	\begin{equation}
		F_{\nabla_\infty} = F_{\nabla_\infty}^7 + F_{\nabla_\infty}^{14},
	\end{equation}
	according to types with respect to $\varphi_C$.

	On the other hand, as $|F_\nabla^7|_h \lesssim e^{-2nt} \to 0$ uniformly, we conclude that any $\nabla_\infty$ must satisfy
	\begin{equation}
		F_{\nabla_\infty}^7 = 0.
	\end{equation}
	Now, write $\nabla_\infty$ in radial gauge in $C$, i.e. as $\nabla_\infty = a_\infty(r)$ with $a_\infty(\cdot)$ a $1$-parameter family of connections over $\Sigma$ parametrized by $r \in [1,e]$. Then, its curvature can be written as
	\begin{equation}
		F_{\nabla_\infty} = F_{a_\infty} + \rd r \wedge \del_r a_\infty .
	\end{equation}
	Then, the condition $F_{\nabla_\infty}^7 = 0$ can equivalently be written as $F_{\nabla_\infty} \wedge \psi_C = 0$ where $\psi_C = \tfrac{\omega^2}{2} - \rd r \wedge \Im(\Omega)$, which then yields
	\begin{align}
		\del_r a_\infty \wedge \frac{\omega^2}{2}	&= F_{a_\infty} \wedge \Im(\Omega), \\
		F_{a_\infty} \wedge \frac{\omega^2}{2}		&= 0.
	\end{align}
	The result follows from observing that $\partial_r a_\infty = 0$.
\end{proof}

\begin{lemma}[Existence of $\Phi_\infty$]
	Under the hypotheses of \Cref{thm:main3}, there is a $\nabla_\infty$-parallel section $\Phi_\infty$ of the adjoint bundle $\mathfrak{g}_{P_\infty}$ over $\Sigma$ such that $\Phi (R) = \Phi |_{\Sigma_R}$ converges uniformly to $\Phi_\infty$ as $R \to \infty$.
\end{lemma}

\begin{proof}
	Consider $\Phi_R = \Phi|_{C_R}$ where $C_R = \{ x \in X : \log (R) \leqslant t (x) \leqslant \log (R) +1 \}$ is equipped with the cylindrical metric $h$ introduced in the proof of \Cref{lem:Existence_of_A_infinity}. Then, translating the coordinate $t$, we consider $\Phi_R$ as a 1-parameter family of Higgs fields in the fixed cylinder $C = [0,1]_t \times \Sigma$ with the fixed metric $h$. Then, \Cref{thm:First_part_of_Main_Theorem_2} (ii) implies
	\begin{equation}
		|\nabla \Phi|_h^2 \lesssim e^{-2 n t},
	\end{equation}
	for $r \in [R, eR]$, i.e. $t \in [\log (R) , \log (R) +1]$. Thus, translating this back to analyze the sequence $\Phi_R$ in the cylinder $C$ we have
	\begin{equation}
		|\nabla \Phi_R|_h^2 \lesssim R^{-2(n-1)} ,
	\end{equation}
	which converges to zero as $R \to \infty$. This together with $|\Phi_R | \lesssim m$ shows that $\Phi_R\to \Phi_\infty$ uniformly over $C$ with $\nabla_\infty \Phi_\infty = 0$. In particular, $\del_t \Phi_\infty = 0$ and so $\Phi_\infty$ is independent of $t$, or $r$, and so is pulled back from $\Sigma$.
\end{proof}

\bigskip

\section{Bogomolny trick for the intermediate energy}\label{sec:Bogomolny_Trick}

In this section we use the outcome of our second main result \Cref{thm:Main_Theorem_2} to deduce a formula for the intermediate energy of critical point configurations, showing in particular that $\rG_2-$monopoles minimize such functional. Before stating the main result we recall a few basic features of our setup. Suppose $(X,\varphi)$ is an (irreducible) AC $\rG_2-$manifold. Then for all sufficiently large $r\gg 1$, the cohomology class $[\Upsilon^*\psi|_{\lbrace r \rbrace \times \Sigma}] \in H^4(\Sigma, \mathbb{R})$ is independent of $r$. In \Cref{def:AC_Cohomology} we have called this the asymptotic cohomology class of $(X, \varphi)$ and we denote it by $\Psi_\infty$.\\
Now let $(\nabla,\Phi)$ be as in \Cref{thm:main3}, with finite mass $m$ (by \Cref{thm:Finite_Mass}), and suppose that $m\neq 0$. Then $(\nabla,\Phi)$ converges along the end to $(\nabla_\infty,\Phi_\infty)$ with $\Phi_\infty \neq 0$ and $\nabla_\infty \Phi_\infty = 0$. Hence, assuming $\rG = \SU (2)$, the holonomy of $\nabla_\infty$ reduces to $\rU (1)\subseteq \SU (2)$ and $\nabla_\infty$ is reducible to a connection on a $\rU (1)$-subbundle $Q_\infty$ of $P_\infty$. It follows from $\SU (2)$ representation theory that $\mathfrak{g}_{P_\infty}\otimes\mathbb{C} \cong \underline{\mathbb{C}} \oplus L \oplus L^\ast$, for a complex line bundle $L\to\Sigma$. We now define what the higher dimensional analog of what in 3 dimensions is known as the monopole charge.

\begin{definition}\label{def:Monopole_Class}
	Given a $\rG_2$-monopole $(\nabla,\Phi)$ on an asymptotically conical manifold as above, the class $\beta = c_1(L) \in H^2(\Sigma, \mathbb{Z})$ is called a monopole class of $(\nabla,\Phi)$.
\end{definition}

\begin{remark}\label{rem:Unique_nabla}
    Given a monopole class $\beta$, there is a unique pseudo-Hermitian--Yang--Mills connection on a complex line bundle $L$ with $c_1(L) = \beta$, see  in \cite{Foscolo2017}*{Remark 3.25}. This is precisely the unique connection whose curvature is the harmonic representative of $-2\pi i \beta$. The connection $\nabla_\infty$, to which $\nabla$ is asymptotic, is induced by this unique connection.
\end{remark}

We are now in position to state our energy formula.

\begin{theorem}\label{thm:Energy_Formula}
	Let $(X^7,\varphi)$ be an irreducible AC $\rG_2$-manifold, with asymptotic cohomology class $\Psi_\infty \in H^4(\Sigma, \mathbb{R})$ (cf. \Cref{def:AC_Cohomology}), and $P\to X$ be a principal $\SU (2)$-bundle. Suppose that $(\nabla,\Phi)$ is a configuration satisfying the second order \cref{eq:2nd_Order_Eq_1,eq:2nd_Order_Eq_2} with $\nabla \Phi \in L^2 (X)$, and either $|F_\nabla|$ decays quadratically along the end or $(\nabla, \Phi)$ is a $\rG_2$-monopole such that $|F_\nabla^{14}|$ decays quadratically along the end. Let $m$ be the mass of $(\nabla, \Phi)$, assume that $m\neq 0$, and let $\beta \in H^2(\Sigma, \mathbb{Z})$ be the monopole class of $(\nabla,\Phi)$. If $(\nabla,\Phi)$ has finite intermediate energy $\cE^{\psi}(\nabla,\Phi)<\infty$, then
	\begin{equation}\label{eq:Energy_Bogomolny}
		\cE^{\psi}(\nabla, \Phi) = 4\pi m \langle \beta \cup \Psi_\infty , [\Sigma]  \rangle + \frac{1}{2}\|F_{\nabla}\wedge\psi - \ast\nabla\Phi\|_{L^2 (X)}^2.
	\end{equation}
	In particular, when $(\nabla, \Phi)$ is a $\rG_2$-monopole, then
	\begin{equation}
		\cE^{\psi}(\nabla, \Phi) = 4 \pi m \langle \beta \cup \Psi_\infty , [\Sigma]  \rangle.
	\end{equation}
\end{theorem}

\begin{proof}
	Let $U \subseteq X$ be a precompact open set with smooth boundary $\partial \overline{U}$. Then the intermediate energy of $(\nabla,\Phi)$ over $U$ can be written as
	\begin{equation}
	    \cE_U^{\psi}(\nabla,\Phi) \coloneqq \frac{1}{2}\left(\|F_{\nabla}\wedge\psi\|_{L^2 (U)}^2 + \|\nabla \Phi\|_{L^2 (U)}^2 \right) = \frac{1}{2}\|F_{\nabla}\wedge\psi - \ast\nabla\Phi\|_{L^2 (U)}^2 + \langle F_{\nabla}\wedge\psi,\ast\nabla\Phi\rangle_{L^2 (U)},
	\end{equation}
	and
	\begin{equation}
		\langle F_{\nabla}\wedge\psi,\ast\nabla\Phi\rangle_{L^2 (U)} = \int\limits_U \langle F_{\nabla}\wedge\psi\wedge\nabla\Phi\rangle = \int\limits_U \rd\left(\langle \Phi , F_\nabla \rangle \wedge \psi\right) = \int\limits_{\partial \overline{U}} \langle \Phi , F_\nabla \rangle \wedge \psi,
	\end{equation}
	where we have used the Bianchi identity $d_{\nabla}F_{\nabla} = 0$ and the fact that $\text{d}\psi = 0$ on the second equality and Stokes' theorem in the last. Thus
	\begin{equation}
	    \cE_U^{\psi}(\nabla,\Phi) = \int\limits_{\partial \overline{U}} \langle \Phi , F_\nabla \rangle \wedge \psi + \frac{1}{2}\|F_{\nabla}\wedge\psi - \ast\nabla\Phi\|_{L^2 (U)}^2.
	\end{equation}
	Now, the finiteness of the intermediate energy implies that the function
	\begin{equation}
	    f (R) \coloneqq \cE_{B_R}^\psi (\nabla, \Phi) = \int\limits_{\Sigma_R} \langle \Phi, F_\nabla \rangle \wedge \psi + \frac{1}{2}\|F_{\nabla}\wedge\psi - \ast\nabla\Phi\|_{L^2 (B_R)}^2
	\end{equation}
	is bounded, nondecreasing and converges to the intermediate energy of $(\nabla, \Phi)$. Thus, we are done if we show that
	\begin{equation}\label{eq:partial_energy_formula}
	   \lim_{R \to \infty} \int\limits_{\Sigma_R} \langle \Phi, F_\nabla \rangle \wedge \psi = 4\pi m \langle c_1(L) \cup \Psi_\infty , [\Sigma]  \rangle. 
	\end{equation}
	Let $\Psi_\infty$ denote the asymptotic cohomology class of $[\psi|_{\Sigma_R}]$ as in \Cref{def:AC_Cohomology} and $\psi_\infty$ its harmonic representative. Then, we can write $\psi|_{\Sigma_R} = \psi_\infty + \rd \gamma$ for some $|\gamma| = O(R)$ with respect to the conical metric. Using this we find,
	\begin{equation}
	    \int\limits_{\Sigma_R} \langle \Phi, F_\nabla \rangle \wedge \rd \gamma = \int\limits_{\Sigma_R} \langle \nabla\Phi \wedge F_\nabla \rangle \wedge \gamma \lesssim \Vol_{R^2 g_\Sigma}( \lbrace R \rbrace \times \Sigma)\sup_{r = R} |\nabla \Phi||F_\nabla||\gamma| \lesssim R^{-1},
	\end{equation}
	where we have used the fact that $|\nabla \Phi| = O (r^{- 6})$ by \Cref{thm:First_part_of_Main_Theorem_2}, $|F_\nabla| = O(r^{-2})$ (cf. proof of \Cref{lem:Existence_of_A_infinity}) and $|\gamma| = O(r)$. Hence, 
	\begin{align}\label{eq:Intermediate_Energy_Computation}
		\lim_{R \to \infty} \int\limits_{\Sigma_R} \langle \Phi, F_\nabla \rangle \wedge \psi = \lim_{R \to \infty} \int\limits_{\Sigma_R} \langle \Phi, F_\nabla \rangle \wedge \psi_\infty = \int\limits_\Sigma \langle \Phi_\infty, F_{\nabla_\infty} \rangle \wedge \psi_\infty.
	\end{align}
	In the case at hand we have $\rG = \SU (2)$, and as $\Phi_\infty \neq 0$ while $\nabla_\infty \Phi_\infty = 0$, the bundle $P_\infty$ reduces from $\SU (2)$ to $\rU (1) \subseteq \SU (2)$ and thus $\mathfrak{g}_{P_\infty} = \underline{\mathbb{R}} \oplus L$, for a complex line bundle, $L$, such that the complex rank $2$ vector bundle associated with the standard representation can be written as $P_\infty \times_{\SU (2)} \mathbb{C}^2 \cong L \oplus L^*$. Using this splitting one writes
	\begin{equation}
		\Phi_\infty = \begin{pmatrix}
			i m & 0 \\
			0  & - i m
		\end{pmatrix} \ , \ \text{and}  \  
		F_{\nabla_\infty} = \begin{pmatrix}
			F_L & 0 \\
			0  & - F_L
		\end{pmatrix},
	\end{equation}
	with $F_L \in -2 \pi i c_1(L) \in H^2(\Sigma, -2 \pi i \mathbb{Z})$ and $m = | \Phi_\infty |$ is the mass of $(\nabla,\Phi)$. Then, inserting into \cref{eq:Intermediate_Energy_Computation} we find \cref{eq:partial_energy_formula} as we wanted. 
\end{proof}

Under the hypothesis of \Cref{thm:Energy_Formula}, the energy formula \eqref{eq:Energy_Bogomolny} writes the intermediate energy as a sum of a topological/geometrical term, determined by the asymptotic geometry of the configuration and the manifold, and a quantity which is always greater than or equal to zero, with equality if and only if $(\nabla,\Phi)$ is a $\rG_2$-monopole. Thus showing, in particular, that keeping both the monopole class $\beta$ and the mass $m$ fixed, $\rG_2$-monopoles minimize the intermediate energy amongst such configurations.

This result also has interesting applications to the existence of $\rG_2$-monopoles on certain asymptotically conical $\rG_2$-manifolds. The following corollary is adapted from \cite{Oliveira2014}, where it was applied using a different hypothesis on the asymptotic behavior of the $\rG_2$-monopoles. Thus our next result may be regarded as a slight improvement on the aforementioned result in \cite{Oliveira2014}.

\begin{corollary}\label{cor:Vanishing}
	Let $(X^7,\varphi)$ be an irreducible AC $\rG_2$-manifold, with vanishing asymptotic cohomology class $\Psi_\infty \in H^4(\Sigma, \mathbb{R})$, e.g. if $H^2(\Sigma,\mathbb{Z})$ vanishes as is the case for the Bryant--Salamon metric on $\mathbb{R}^4 \times S^3$ for which $\Sigma = S^3 \times S^3$. Then, there are no irreducible $\rG_2$-monopoles $(\nabla,\Phi)$ on $(X^7, \varphi)$ with $|F_\nabla^{14}|$ quadratically decaying and finite nonzero intermediate energy.
\end{corollary}

\begin{proof}
	In this situation, the energy formula from \Cref{thm:Energy_Formula} applies and we find that the intermediate energy vanishes.
\end{proof}

\bigskip

\section{Decay of linearized solutions}\label{sec:Linearized_Solutions}

In this section we make further use of the methods of the previous sections, and give decay estimates to finite energy solutions to the linearization of the $\rG_2$-monopole equation, with appropriate gauge fixing.

\smallskip

\subsection{The linearized equation}\label{ss:Linearized_Eq}

First we compute the (formal) linearization of the $\rG_2$-monopole \cref{eq:Monopole}.

Let the configuration space be $\cC_P = \mathrm{Conn}_P \times \Omega_{\mathfrak{g}_P}^0$, that is, the space of (Sobolev) pairs of connections on $P$ and sections of $\mathfrak{g}_P$. Consider the (smooth) maps
\begin{align}
	\mon^\pm : \cC_P			&\rightarrow \Omega_{\mathfrak{g}_P}^1; \\
	\left( \nabla, \Phi \right)	&\mapsto \ast (F_\nabla \wedge \psi) \pm \nabla \Phi,
\end{align} 
that, in case of a $\rG_2$-monopole, satisfy
\begin{align}
	\mon^- (\nabla, \Phi)	&= 0, \\
	\mon^+ (\nabla, \Phi)	&= 2 \nabla \Phi.
\end{align}
For simplicity we set $\mon \coloneqq \mon^-$ for the rest of the paper.

Since $\cC_P$ is an affine space---in fact, an affine Hilbert manifold---the tangent spaces of $\cC_P$ at any point can canonically be identified with $\Omega_{\mathfrak{g}_P}^1 \oplus \Omega_{\mathfrak{g}_P}^0$, as Hilbert spaces. We can now compute the first derivative of $\mon$. If $X = (a, \upphi) \in T_{(\nabla, \Phi)} \cC_P$, then
\begin{equation}
	\left( \mathrm{T}_{(\nabla, \Phi)} \mon \right) (a, \upphi) = \ast \left( \rd_\nabla a \wedge \psi \right) - \nabla \upphi - [a, \Phi]. \label{eq:Tmon-}
\end{equation}
We also impose the standard, Coulomb type gauge fixing condition that we only consider tangent vectors $(a, \upphi) \in T_{(\nabla, \Phi)} \cC_P$ that are perpendicular to the gauge orbit through $(\nabla, \Phi)$, which amounts to the following PDE:
\begin{equation}
	\rd_\nabla^* a + [\upphi, \Phi] = 0. \label{eq:Coulomb}
\end{equation}
We can organize \cref{eq:Tmon-,eq:Coulomb} into a single equation. Let
\begin{subequations}
\begin{align}
	\D (a, \upphi)	&= \left( \ast (\rd_\nabla a \wedge \psi) - \nabla \upphi, - \rd_\nabla^* a \right), \label{eq:D} \\
	q (a, \upphi)	&= \left( [\Phi, a], [\Phi, \upphi] \right), \label{eq:q}
\end{align}
\end{subequations}
and let $\cL = \D + q$. Then \cref{eq:Tmon-,eq:Coulomb} are equivalent to
\begin{equation}
	\cL (a, \upphi) = 0. \label{eq:linG2mono}
\end{equation}
We call \cref{eq:linG2mono} (or \cref{eq:Tmon-,eq:Coulomb} together) the \emph{linearized $\rG_2$-monopole equation}.

\smallskip

\subsection{A priori properties of the linearized $\rG_2$-monopole equation}

First we prove two Weitzenb\"ock type formulas, one for $\cL$ and one for $\cL^*$.

\begin{lemma}
	\label{lem:2nd_order_linearized}
	Let $(\nabla, \Phi) \in \cC_P$ be any pair. For any $(a, \upphi) \in T_{(\nabla, \Phi)} \cC_P$ smooth pair let us define
	\begin{align}
		I_+ (a, \upphi)	&= \left( 2 \ast \left[ \left( \ast F_\nabla \right) \wedge a \right] - \left[ \mon^+ (\nabla, \Phi), \upphi \right], \ast \left[ \mon^+ (\nabla, \Phi) \wedge \ast a \right] \right), \\
		I_- (a, \upphi)	&= \left( - 2 \ast \left( \varphi \wedge \left[ F_\nabla \wedge a \right] \right) - \left[ \mon^- (\nabla, \Phi), \upphi \right], \ast \left[ \mon^- (\nabla, \Phi) \wedge \ast a \right] \right).
	\end{align}
	Then we have
	\begin{align}
		\cL^* \cL	&= \nabla^* \nabla + I_+ - q^2, \\
		\cL \cL^*	&= \nabla^* \nabla + I_- - q^2.
	\end{align}
\end{lemma}

\begin{proof}
	We start by noticing that $\cL = \D + q$ and $\cL^* = \D - q$, then we can write
	\begin{align}
		\cL^* \cL	&= \D^2 + [\D, q] - q^2, \\
		\cL \cL^*	&= \D^2 - [\D, q] - q^2.
	\end{align}
	Let us first rewrite $\D^2$:
	\begin{align}
		\D^2 (a, \upphi)	&= \D \left( \ast (\rd_\nabla a \wedge \psi) - \nabla \upphi, - \rd_\nabla^* a \right) \\
							&= \left( \ast (\rd_\nabla \left( \ast (\rd_\nabla a \wedge \psi) - \nabla \upphi \right) \wedge \psi) - \nabla \left( - \rd_\nabla^* a \right), - \rd_\nabla^* \left( \ast (\rd_\nabla a \wedge \psi) - \nabla \upphi \right) \right) \\
							&= \left( \rd_\nabla \rd_\nabla^* a + \ast \left( \rd_\nabla \left( \ast \left( \rd_\nabla a \wedge \psi \right) \right) \wedge \psi \right) - \ast \left( \left( \rd_\nabla^2 \upphi \right) \wedge \psi \right), \nabla^* \nabla \upphi - \ast \left( (\rd_\nabla^2 a) \wedge \psi \right) \right) \\
							&= \left( \rd_\nabla \rd_\nabla^* a + \ast \left( \rd_\nabla \left( \ast \left( \rd_\nabla a \wedge \psi \right) \right) \wedge \psi \right) - [\ast \left( F_\nabla \wedge \psi \right), \upphi], \nabla^* \nabla \upphi - \ast \left( [F_\nabla \wedge a] \wedge \psi \right) \right)
	\end{align}
	Let $\rd_\nabla^7 := \Pi_7 \circ \rd_\nabla : \Omega^1 \rightarrow \Omega_7^2$, where $\Pi_7 : \bigwedge^2 \rightarrow \bigwedge_7^2$ is the pointwise orthogonal projection. Since for 2-form, $\omega$, we have that $\ast \left( \omega \wedge \psi \right) = \omega^7$, we get that
	\begin{equation}
		\ast \left( \rd_\nabla \left( \ast \left( \rd_\nabla a \wedge \psi \right) \right) \wedge \psi \right) = 3 \rd_\nabla^* \rd_\nabla^7 a = \rd_\nabla^*\rd_\nabla a - \ast \left[ \left( F_\nabla \wedge \varphi \right) \wedge a \right].
	\end{equation}
	Using also $\ast \left( F_\nabla \wedge \psi \right) = \nabla \Phi$ and the Jacobi identity, we get that
	\begin{equation}
		\D^2 (a, \upphi) = \left( \rd_\nabla \rd_\nabla^* a + \rd_\nabla^*\rd_\nabla a - \ast \left[ \left( F_\nabla \wedge \varphi \right) \wedge a \right] - [\nabla \Phi, \upphi], \nabla^* \nabla \upphi + \ast [(\ast \nabla \Phi) \wedge a] \right).
	\end{equation}
	Now recall the Weitzenb\"ock identity for bundle-valued 1-forms on a Ricci-flat manifold:
	\begin{equation}
		\rd_\nabla \rd_\nabla^* a + \rd_\nabla^*\rd_\nabla a = \nabla^* \nabla a + \ast [(\ast F_\nabla) \wedge a].
	\end{equation}
	Combining these we get	
	\begin{equation}
		\D^2 (a, \upphi) = \left( \nabla^*\nabla a + \ast [(\ast F_\nabla - F_\nabla \wedge \varphi) \wedge a] - [\nabla \Phi, \upphi], \nabla^* \nabla \upphi + \ast [(F_\nabla \wedge \psi) \wedge a] \right).
	\end{equation}

	Let us now compute $[\D, q]$.
	\begin{align}
		[\D, q] (a, \upphi)	&= \D \left( [\Phi, a], [\Phi, \upphi] \right) - q \left( \ast (\rd_\nabla a \wedge \psi) - \nabla \upphi, - \rd_\nabla^* a \right) \\
							&= \left( \ast \left( \rd_\nabla \left( [\Phi, a] \right) \wedge \psi \right) - \nabla \left( [\Phi, \upphi] \right), - \rd_\nabla^* \left( [\Phi, a] \right) \right) \\
							& \quad - \left( [\Phi, \ast (\rd_\nabla a \wedge \psi) - \nabla \upphi], [\Phi, - \rd_\nabla^* a] \right) \\
							&= \left( \ast [(\nabla \Phi \wedge \psi) \wedge a] + [\Phi, \ast (\rd_\nabla a \wedge \psi)] - [\nabla \Phi, \upphi] - [\Phi, \nabla \upphi], \ast [\nabla \Phi \wedge \ast a] - [\Phi, \rd_\nabla^* a] \right) \\
							& \quad - \left( [\Phi, \ast (\rd_\nabla a \wedge \psi)] - [\Phi, \nabla \upphi], [\Phi, - \rd_\nabla^* a] \right) \\
							&= \left( \ast [(\nabla \Phi \wedge \psi) \wedge a] - [\nabla \Phi, \upphi], \ast [\nabla \Phi \wedge \ast a] \right).
	\end{align}
	Thus, after some simplification, we get
	\begin{multline}
		\left( \D^2 \pm [\D, q] - \nabla^* \nabla \right) (a, \upphi) = \\
		\left( \ast [(\pm \nabla \Phi \wedge \psi + \ast F_\nabla - F_\nabla \wedge \varphi) \wedge a] - [\mon^\pm (\nabla, \Phi), \upphi], \ast [\mon^\pm (\nabla, \Phi) \wedge \ast a] \right).
	\end{multline}
	Using the $\rG_2$-monopole \cref{eq:Monopole} and the following identity
	\begin{equation}
		\omega \wedge \varphi = 2 \ast \omega^7 - \ast \omega^{14}, \label{eq:omega_wedge_varphi}
	\end{equation}
	that holds for any $\omega = \omega^7 + \omega^{14} \in \Lambda^2 X = \Lambda_7^2 X \oplus \Lambda_{14}^2 X$, we can rewrite $I_\pm (a, \upphi)$ as
	\begin{equation}
		\left( \D^2 \pm [\D, q] - \nabla^* \nabla \right) (a, \upphi) = \begin{pmatrix} \ast \left[ \left( (- 1 \pm 3) \ast F_\nabla^7 + 2 \ast F_\nabla^{14} \right) \wedge a \right] - \left[ \mon^\pm (\nabla, \Phi), \upphi \right] \\ \ast \left[ \mon^\pm (\nabla, \Phi) \wedge \ast a \right)] \end{pmatrix}^T,
	\end{equation}
	which completes the proof (after using \cref{eq:omega_wedge_varphi} again in negative sign case).
\end{proof}

Using Moser iteration, we get the following immediate corollary.

\begin{corollary}
	Let $\left( \nabla, \Phi \right) \in \cC_P$ be a pair, such that $r^2 \left( |F_\nabla| + |\nabla \Phi| \right) \in L^\infty (X)$. Let $V = (a, \upphi) \in T_{\left( \nabla, \Phi \right)} \cC_P$ be in the $L^2 (X)$-kernel of either $\cL$ or $\cL^*$. Then $r^n |V|^2 \in L^\infty (X)$.
\end{corollary}

\begin{proof}
	By elliptic regularity, $V$ is smooth. Using the results of \Cref{lem:2nd_order_linearized}, we get (in both cases) and outside of $B_R$ (for $R$ large enough) that
	\begin{equation}
		\Delta |V|^2 \leqslant \frac{C}{r^2} |V|^2.
	\end{equation}
	In fact, we have faster than quadratic decay for the coefficient, but that does not change what follows. We can now use Moser iteration (cf. \Cref{prop:Moser}). Since the injectivity radius satisfies $\mathrm{inj} (x) \geqslant c r (x)$, we have (with a potentially different constant) that
	\begin{equation}
		|V (x)|^2 \leqslant \frac{C}{r (x)^n} \int\limits_{B_{r (x)/2} (x)} |V|^2 \vol_X, \label[ineq]{ineq:Moser_for_V}
	\end{equation}
	which completes the proof.
\end{proof}

\smallskip

\subsection{Decay properties of solutions to the linearized equation}

\begin{theorem}
	Let $\left( \nabla, \Phi \right) \in \cC_P$ be a pair, such that $r^2 \left( |F_\nabla| + |\nabla \Phi| \right) \in L^\infty (X)$. Let $(a, \upphi) \in T_{\left( \nabla, \Phi \right)} \cC_P$ be in the kernel of either $\cL$ or $\cL^*$, $V = (a, \upphi)$, and $|V|^2 = |a|^2 + |\upphi|^2$. Then for all $\epsilon > 0$, we have that $r^{2(n - 1)} |V|^2 \in L^\infty (X)$.
\end{theorem}

\begin{proof}
	To prove the claim, we now bound the integrals $\int_{B_{r (x)} (x)} |V|^2 \vol_X$ in \cref{ineq:Moser_for_V}, using the improved Hardy's \cref{ineq:Hardy_2} in a similar way as in \cref{ineq:Hardy_rewrite_1,ineq:Hardy_rewrite_2}. More concretely, we show that for all $\alpha < \tfrac{n - 1}{2}$:
	\begin{equation}
		\| r^{\alpha - \frac{1}{2}} V \|_{L^2 (X)}^2 \leqslant \frac{C}{n - 1 - 2 \alpha} \| V \|_{L^2 (X)}^2.  \label[ineq]{ineq:V_int_bound}
	\end{equation}
	We again use \cref{ineq:Hardy_2} together with Kato's inequality, but now with $f = r_{l, L}^\alpha r^{-1/2} |V|$, to get
	\begin{multline}
		\frac{(n-2)^2}{4} \ \| r_{l, L}^\alpha r^{-1/2} V \|_{L^2 (X - B_l)}^2 + \| r_{l, L}^\alpha \sqrt{r} \nabla^\Sigma \nabla \Phi \|_{L^2 (X - B_l)}^2 \leqslant \| \nabla (r_{l, L}^\alpha \sqrt{r} \nabla \Phi ) \|_{L^2 (X - B_l)}^2 \\
		+ C_H \|r_{l, L}^\alpha r^{-1/2 + \nu} \nabla \Phi \|_{L^2 (X - B_l)}^2.
	\end{multline}
	From here we can proceed almost identically to proof of \Cref{prop:Decay_2} (in fact, the computations are now easier as the equations are linear, and we already have the decay result for $(\nabla, \Phi)$ by \Cref{thm:Main_Theorem_2}) to get that $r^{2(n - 1) - \epsilon} |V|^2 \in L^\infty (X)$.

	Then, as in \Cref{sec:final_est}, one can show that, in fact, $r^{2(n - 1)} |V|^2 \in L^\infty (X)$.
\end{proof}

\bigskip

\section{Weighted split Sobolev spaces: Fredholm operators}\label{sec:Sobolev_Fredholm}

Let $(X,g)$ be an asymptotically conical spin manifold. In this section we construct Sobolev spaces suitable for a Fredholm theory for $\rG_2$-monopoles. In that setting, the gauge fixed linearization of the $\rG_2$-monopoles equation $\cL = \D + q$, as defined in \Cref{ss:Linearized_Eq}, is the sum of twisted Dirac type operator $\D$ and a potential $q$. Such operators are usually called Callias type operators \cite{Callias1978} and are well known to be related to the deformation theory of monopoles in other dimensions \cites{Kottke10,Kottke15}. Hence, in this section we work with a slightly more general setup than what is needed and the Sobolev spaces constructed are those suitable to analyze Callias type operators.\\
Let $\cS$ be the spinor bundle of $(X,g)$, i.e. the vector bundle associated with the standard $\Spin(n)$ representation, and $E$ an auxiliary metric vector bundle which in the monopole case should be thought of as being $\mathfrak{g}_P$. Denote the tensor product of these two bundles by $\mathcal{S}_E \coloneqq \mathcal{S} \otimes E$ and equip it with a connection $\nabla$ and a metric $\langle \cdot , \cdot \rangle$ both induced by counterparts on $E$ and $\cS$ (equipped with the metric and connection induced by $g$ and the Levi--Civita connection). This section culminates in \Cref{thm:Main_Fredholm_Theorem} with the proof that $\cL$ is Fredholm when operating between some suitable and carefully constructed Banach spaces.

\smallskip

\subsection{Nondegenerate potential and standard Sobolev spaces}\label{ss:Callias}

We start by analyzing the case when the potential $q$ is pointwise invertible at infinity and prove that, in this situation, the operator $\cL = \D + q$ is Fredholm between standard Sobolev spaces. This is the setting worked out in \cites{Ang90,Kottke10} where a formula for the index in a quite general setup is given. Here we give a short proof that 
\begin{equation}
	\cL : L^2_1 (X) \to L^2 (X),
\end{equation}
is Fredholm along the lines motivated by \cites{Taubes1983,Donaldson2002}. We begin by studying the model situation on a cone which we then extend to the AC setting. Before proceeding, we introduce a little notation. Let $\epsilon$ be a smooth positive function on a metric cone $C = \left( (1, \infty) \times \Sigma, g_C = \rd r^2 + r^2 g_\Sigma \right)$ decaying with all derivatives as $r \to \infty$, i.e. such that
\begin{equation}\label{eq:FunctionEps}
	\lim_{r \rightarrow \infty} \vert r^j \nabla^j \epsilon(r) \vert = 0.
\end{equation}

\begin{proposition}\label{prop:ModelInequality}
	In the metric cone as above and $\cL_C = \D_C + q_C$ as before with $q_C$ parallel and bounded by bellow, i.e. $\nabla ( q_C(f) ) = q_C(\nabla f)$ and $\vert q_C(f) \vert^2 \geqslant c \vert f \vert^2$ for some constant $c >0$ and all $f \in \Omega^0(C,\mathcal{S}_E)$. Furthermore, suppose there is a Weitzenb\"ock formula
	\begin{equation}
		\cL_C^* \cL_C = \nabla^* \nabla + W + q_C^* q_C,
	\end{equation}
	with $W$ satisfying $\vert W(f) \vert \leqslant \epsilon^2(r) \vert f \vert$ for some function $\epsilon(r)>0$ as in \cref{eq:FunctionEps}. Then, the following inequality holds
	\begin{equation}\label[ineq]{ineq:ModelInequality}
		\Vert f \Vert_{L^2_1 (X)}^2 \lesssim \Vert \cL_C f \Vert_{L^2 (X)}^2 + \Vert \epsilon(r) f \Vert_{L^2 (X)}^2,
	\end{equation}
	for all compactly supported $f$. 
\end{proposition}

\begin{proof}
	For compactly supported $f$ one can integrate by parts and use the Weitzenb\"ock formula in the statement
	\begin{align}
		\Vert \cL_C f \Vert^2_{L^2 (X)} &= \langle \cL_C^* \cL_C f , f \rangle_{L^2 (X)} = \Vert \nabla f \Vert^2_{L^2 (X)} + \langle W(f), f \rangle_{L^2 (X)} + \Vert q(f) \Vert^2_{L^2 (X)} \\
		&\geqslant \Vert \nabla f \Vert^2_{L^2 (X)}  - \Vert \epsilon(r) f \Vert_{L^2 (X)}^2 + c\Vert f \Vert_{L^2 (X)}^2.
	\end{align}
	and passing the term $\Vert \epsilon (r) f \Vert_{L^2 (X)}^2$ to the left hand side yields the \cref{ineq:ModelInequality}.
\end{proof}

We now go from the conical case to the asymptotically conical one. 

\begin{lemma}\label{lem:CompactEmbedding}
	Let $\epsilon : X \rightarrow \mathbb{R}_+$ be smooth function satisfying \eqref{eq:FunctionEps} and $\epsilon^{-1} L^2 (X) = \lbrace f \ | \ \epsilon f \in L^2 \rbrace$. Then the embedding $L^2_1 (X) \hookrightarrow \epsilon^{-1} L^2 (X)$ is compact. 
\end{lemma}
\begin{proof}
	The Banach space $\epsilon^{-1}L^2 (X)$ can equally be defined as the completion of the smooth compactly supported sections in the norm $\Vert f \Vert_{\epsilon^{-1}L^2} = \Vert \epsilon f \Vert_{L^2 (X)}$. Then, we consider a generic sequence $\lbrace f_i  \rbrace \subseteq L^2_1$ satisfying $\Vert f_i \Vert_{L^2_1 (X)}^2 = 1$, which we show must subsequentially converges in $	\epsilon^{-1}L^2 (X)$.\\
	Since $\Vert f_i \Vert_{L^2_1 (X)}^2 = 1$, there is a subsequence converging to a weak limit $f \in L^2_1 (X)$ satisfying $\Vert f \Vert^2_{L^2_1 (X)} \leqslant 1$. We now show that this subsequence converges strongly in $\epsilon^{-1} L^2 (X)$ to $f$. In the following computation let $B_R = B_R (x_0)$. Then,
	\begin{align}
		\Vert \epsilon (f_i - f) \Vert^2_{L^2 (X)} &= \Vert \epsilon (f_i - f) \Vert^2_{L^2 (B_R)} + \Vert \epsilon  (f_i - f) \Vert^2_{L^2 (X - B_R)} \\
		&\leqslant c_1 \Vert f_i - f \Vert^2_{L^2 (B_R)} + \epsilon^{2}(R)  \Vert f_i - f \Vert^2_{L^2 (X - B_R)} \\ \label{eq:estupidaIntIneq}
		&\leqslant c_1 \Vert f_i - f \Vert^2_{L^2 (B_R)} + 4 \epsilon^2(R)  ,
	\end{align}
	where in the last inequality we used
	\begin{equation}
		\Vert f_i - f \Vert^2_{L^2 (X - B_R)}  \leqslant \Vert f_i - f \Vert^2_{L^2_1(X - B_R)} \leqslant 2 \Vert f_i  \Vert^2_{L^2_1(X - B_R)}  + 2  \Vert f \Vert^2_{L^2_1(X - B_R)} \leqslant 4.
	\end{equation}
	The second term in \cref{eq:estupidaIntIneq} is $4 \epsilon^2(R)$ and, by increasing $R$, can be taken to be arbitrarily small. Regarding the first one $\Vert f_i - f \Vert^2_{L^2 (B_R)}$, since the embedding $L^2_1(B_R) \hookrightarrow L^2 (B_R)$ is compact, $f_i$ does converge strongly to $f$ in $L^2 (B_R)$ and the term $\Vert f_i - f \Vert^2_{L^2 (B_R)}$ can also be made arbitrarily small by increasing $i$. 
\end{proof}

\begin{lemma}\label{lem:Standard_Inequality_Fredholm}
	Let $\cL : \Omega^0 (X, \mathcal{S}_E) \rightarrow \Omega^0 (X, \mathcal{S}_E)$ be modeled on a conical operator $\cL_C$ as in \Cref{prop:ModelInequality}. Then, there is a function $\epsilon$ as in \eqref{eq:FunctionEps} such that the inequality
	\begin{equation}\label[ineq]{ineq:NotModelInequality}
		\Vert f \Vert_{L^2_1 (X)}^2 \lesssim \Vert \cL f \Vert_{L^2 (X)}^2 + \Vert f \Vert_{\epsilon^{-1} L^2 (X)}^2 ,
	\end{equation}
	holds for any $f \in L^2_1 (X)$.
\end{lemma}

\begin{proof}
	We prove this inequality by putting together two similar inequalities computed over: the interior of $X$; and its ends. As in the proof of the previous lemma, let $R \gg 1$ and $B_R \coloneqq B_R (x_0)$. The ellipticity of $L$ on $B_{R+1}$ implies that
	\begin{align}
		\Vert f \Vert_{L^2_1 (B_{R+1})}^2 &\lesssim \Vert \cL f \Vert_{L^2 (B_{R+2})}^2 +  \Vert f \Vert_{L^2 (B_{R+2})}^2 \\  \label{eq:IntermediateInequalityonB_R}
		&\lesssim \Vert \cL f \Vert_{L^2 (B_{R+2})}^2 +  \epsilon(R+2)^{-2} \Vert \epsilon f \Vert_{L^2 (B_{R+2})}^2 ,
	\end{align}
	for all $f$. This deals with the interior of $X$ and we now focus with what happens at its ends., i.e. on $X - B_R$. Notice that by possibly increasing $R$, we can assume that $X - B_R$ is quasi isometric to the metric cone and there is a model conical operator $\cL_C$ on the cone satisfying the hypothesis in \Cref{prop:ModelInequality}, and $\cL - \cL_C = O(r^{-1-\delta})$ for some $\delta >0$. Hence, it follows from \Cref{prop:ModelInequality} that
	\begin{equation}\label{eq:IntermediateInequalityonXminusB_R}
		\Vert f \Vert_{L^2_1 (X - B_R)}^2 \lesssim \Vert \cL f \Vert_{L^2 (X - B_R)}^2 +  \Vert \epsilon' f \Vert_{L^2 (X - B_R)}^2,
	\end{equation}
	for $\epsilon' = \max \lbrace r^{-1-\delta} , \epsilon \rbrace$. The last step is to put this together \eqref{eq:IntermediateInequalityonB_R} and \eqref{eq:IntermediateInequalityonXminusB_R}. For this, fix a bump function $\varphi_R$ be a supported on $B_{R+1}$ which equals $1$ on $B_R$, then
	\begin{align}
		\Vert f \Vert_{L^2_1 (X)}^2 &= \Vert f \Vert_{L^2_1 (B_R)}^2 + \Vert f \Vert_{L^2_1 (X - B_R)}^2 \\
		&\leqslant \Vert \varphi_{R+1}f \Vert_{L^2_1 (B_{R+1})}^2 + \Vert (1 - \varphi_{R})f \Vert_{L^2_1 (X - B_R)}^2 \\
		&\lesssim \Vert \cL f \Vert_{L^2 (X)}^2 +  \Vert \epsilon f \Vert_{L^2 (X)}^2,
	\end{align}
	which is simply \eqref{ineq:NotModelInequality}.
\end{proof}

\begin{corollary}\label{cor:ClosedRange}
	Let $\cL : L^2_1 (X) \rightarrow L^2 (X)$ be as in \Cref{lem:Standard_Inequality_Fredholm}. Then $\cL$ has closed range and finite dimensional kernel.
\end{corollary}

\begin{proof}
	To prove that the kernel is finite dimensional we prove that the unit ball in the kernel is compact. So let $\lbrace f_i \rbrace \subseteq \ker(D)$ be a sequence with $\Vert f_i \Vert_{L^2_1 (X)}^2 = 1$. From \Cref{lem:CompactEmbedding}, the embedding $L^2_1 (X) \hookrightarrow \epsilon^{-1} L^2 (X)$ is compact and so there is a subsequence $f_i$, which converges strongly in $\epsilon^{-1} L^2 (X)$ to some $f \in \ker(D)\cap \epsilon^{-1}L^2 (X)$. But then, the \cref{ineq:NotModelInequality} gives $\Vert f_i - f \Vert_{L^2_1 (X)}^2 \leqslant c_2\Vert \epsilon (f_i - f )\Vert_{L^2 (X)}^2 \rightarrow 0$, and so $f_i$ does converge to $f$ strongly in $L^2_1 (X)$.\\ 
	Next we prove that the image is closed, for that it is enough to prove that there is a constant $c>0$, such that for all $f \in (\ker (\cL))^\perp \cap L^2_1 (X)$
	\begin{equation}
		\Vert \cL f \Vert_{L^2 (X)} \geqslant c \Vert f \Vert^2_{L^2_1 (X)}.
	\end{equation}
	Suppose not, then there is a sequence $\lbrace f_i \rbrace \subseteq (\ker (\cL))^\perp \cap L^2_1 (X)$ with $\Vert \cL f_i \Vert^2_{L^2 (X)} \rightarrow 0$ and $\Vert f_i \Vert_{L^2_1 (X)}^2 = 1$. There is a weak limit $f \in L^2_1 (X)$ such that $\cL f = 0$ and from \Cref{lem:CompactEmbedding}, the limit $f$ is strong in $ \epsilon^{-1} L^2 (X)$. In fact $f = 0$ since by assumption it is the limit of the $f_i$'s which are in the orthogonal complement to the kernel. Then \cref{ineq:NotModelInequality} gives 
	\begin{equation}
		1 = \Vert f_i \Vert_{L^2_1 (X)}^2 \leqslant \Vert \cL f_i \Vert_{L^2 (X)}^2 + \Vert \epsilon f_i \Vert_{L^2 (X)}^2,
	\end{equation} 
	as the first term in the right hand side vanishes, while the second one converges to zero this is a contradiction.
\end{proof} 

\begin{corollary}\label{cor:Callias}
	Let $\cL : L^2_1 (X) \rightarrow L^2 (X)$ be as in \Cref{lem:Standard_Inequality_Fredholm}, then it is a Fredholm operator.
\end{corollary}

\begin{proof}
	We have shown in \Cref{cor:ClosedRange} that $\cL$ has finite dimensional kernel and closed image. Hence, the only thing missing to prove that it is Fredholm is that the cokernel is finite dimensional. As $\coker (\cL) \cong \ker (\cL^*) \cap L^2 (X)$  one just needs to prove that this later one is finite dimensional. Since $\cL^* = \D + q^*$, it is also modeled on an operator as in the hypothesis of \Cref{prop:ModelInequality} and therefore satisfies an inequality as in \eqref{ineq:NotModelInequality}. Using such an inequality, one concludes that 
	\begin{equation}
		\Vert f \Vert_{L^2_1 (X)} \lesssim \Vert \epsilon f \Vert_{L^2 (X)}  \lesssim \Vert  f \Vert_{L^2 (X)},
	\end{equation}
	for all $f \in \ker (\cL^*) \cap L^2 (X)$, and so $\ker (\cL^*) \cap L^2 (X) \hookrightarrow L^2_1 (X)$ and applying \Cref{cor:ClosedRange} to $\cL^*$ proves that its kernel in $L^2_1 (X)$ is finite dimensional.
\end{proof}

For completeness we now see that any such $\cL$ is also a Fredholm operator when considered as an operator on higher derivative Sobolev spaces.

\begin{proposition}\label{prop:Callias}
	Let $k \in \mathbb{N}$ and $\cL : \Omega^0 (X, \mathcal{S}_E) \rightarrow \Omega^0 (X, \mathcal{S}_E)$ be modeled on a conical operator $\cL_C$ as in \Cref{prop:ModelInequality}. Then $\cL : L^2_{k+1} (X) \rightarrow L^2_k (X)$ is a Fredholm operator.
\end{proposition}

\begin{proof}
	If one can prove an inequality of the form
	\begin{equation}\label{eq:ModelInequalityk}
		\Vert f \Vert_{L^2_{k+1} (X)}^2 \lesssim \Vert \cL f \Vert_{L^2_k (X)}^2 + \Vert \epsilon f \Vert_{L^2_k (X)}^2,
	\end{equation}
	for both $\cL$ and $\cL^*$ and some $\epsilon$ as in \cref{eq:FunctionEps}, then the result follows. Indeed, one can repeat all of the steps done before with $L^2 (X)$ replaced by $L^2_k (X)$ and $L^2_1 (X)$ replaced by $L^2_{k+1} (X)$. We start by noticing that the operator $\cL$ can be extended to act on sections of $T^*X \otimes \mathcal{S}_E$. Then, the Weitzenb\"ock formulas for $\cL^* \cL$ and $\cL \cL^*$ have a further contribution coming from the Riemannian curvature, which actually vanishes in the Ricci flat case. In general, the manifold is AC and this algebraic term decays and it can be bounded from above by an $\epsilon$ function as in \cref{eq:FunctionEps}, so one can assume these Weitzenb\"ock formulas are as in \Cref{prop:ModelInequality}. To establish the inequality, notice that
	\begin{equation}\label[ineq]{ineq:L2k_Inequality}
		\Vert f \Vert_{L^2_{k+1} (X)}^2 \leqslant \Vert f \Vert_{L^2_1 (X)}^2 + \Vert \nabla f \Vert_{L^2_k (X)}^2 
	\end{equation}
	and arguing by induction one can assume \cref{eq:ModelInequalityk} to be true for $k$ replaced by $j < k$, and we now prove the case $j = k$. Then, using the induction hypothesis and \cref{ineq:L2k_Inequality}
	\begin{equation}\label[ineq]{ineq:IntK}
		\Vert f \Vert_{L^2_{k+1} (X)}^2 \lesssim \left( \Vert \cL f \Vert_{L^2_1 (X)}^2 + \Vert \cL \nabla f \Vert_{L^2_{k-1} (X)}^2 \right) + \left( \Vert \epsilon f \Vert_{L^2_1 (X)}^2 + \Vert \epsilon \nabla f \Vert^2_{L^2_{k-1}} \right).
	\end{equation}
	Notice that $\epsilon \nabla f = \nabla (\epsilon f) - (\rd \epsilon) \otimes f$. Moreover, since $\epsilon$ satisfies \cref{eq:FunctionEps}, there is some other function $\epsilon_1$ still decaying as in \cref{eq:FunctionEps} and so that $\vert \epsilon \vert + \vert \rd \epsilon \vert \leqslant \epsilon_1$. So one can bound the above terms in the second pair of parentheses by $\Vert \epsilon_1 f \Vert^2_{L^2_k (X)}$. To bound from above the terms in the first bracket in \cref{ineq:IntK}, let $\lbrace e_i \rbrace$ be an orthonormal frame at $p \in X$ such that $\nabla e_i = 0$ at $p$. Then at $p$
	\begin{align}
		L(\nabla_j f) &= \D \nabla_j f + q(\nabla_j f) = \sum_i e_i \nabla_i \nabla_j f + q(\nabla_j f)\\
		&= \sum_i \left(\nabla_j ( e_i \nabla_i f )  + e_iF_\nabla(e_i, e_j)(f) \right)+ \nabla_j (q( f)) - (\nabla_j q)(f)\\
		&= \nabla_j (\cL f) + \sum_i  e_iF_\nabla(e_i, e_j)(f)- (\nabla_j q)(f),
	\end{align}
	where $F_\nabla$ is the curvature of $\nabla$. As $\cL$ is modeled in a conical operator as in \Cref{prop:ModelInequality} for which $q$ is parallel, we must have that $\nabla q$ and $F_\nabla$ decay and are therefore bounded above by some $\epsilon_2$ as in \cref{eq:FunctionEps}. From this it immediately follows that
	\begin{equation}
		\Vert \cL \nabla f \Vert_{L^2_{k-1} (X)}^2 \lesssim \Vert \nabla \cL f \Vert_{L^2_{k-1} (X)}^2 + \Vert \epsilon_2 f \Vert_{L^2_{k-1} (X)}^2,
	\end{equation}
	which together with the previous bound $\Vert \epsilon f \Vert_{L^2_1 (X)}^2 + \Vert \epsilon \nabla f \Vert^2_{L^2_{k-1}} \leqslant \Vert \epsilon_1 f \Vert_{L^2_k (X)}^2$, gives the \cref{ineq:IntK} for some $\epsilon \geqslant \epsilon_1 + \epsilon_2$.
\end{proof}

\smallskip

\subsection{Vanishing potential and weighted Sobolev spaces}\label{ss:Dirac}

A Dirac-type operator $\D$ on an AC manifold is not Fredholm for the standard Sobolev spaces. Lockhart--McOwen \cite{Lockhart1985} and Marshall \cite{Marshall02} have both constructed suitable weighted versions of the standard Sobolev spaces on which the Dirac operator is Fredholm. In this subsection we recall the definition of these spaces and state the resulting Fredholm property.\\
Let $\alpha \in \mathbb{R}$ and $p,k \in \mathbb{N}_+$, the Lockhart--McOwen \cite{Lockhart1985} weighted norms $\Vert \cdot \Vert_{L^p_{k,\alpha} (X)}$ of a smooth compactly supported twisted Dirac spinor $f \in \Gamma(X, \mathcal{S}_E)$ are inductively defined by
\begin{equation}\label{eq:Lockhart}
	\Vert f \Vert_{L^p_{k,\alpha}} = \Vert \nabla f \Vert_{L^p_{k-1,\alpha-1} (X)} + \Vert f \Vert_{L^p_{0, \alpha} (X)},
\end{equation}
and $\Vert f \Vert_{L^p_{0, \alpha} (X)}^p = \int_X \vert r^{-\alpha } f \vert^p r^{-n} \vol_X$.

\begin{definition}[Lockhart--McOwen weighted Sobolev spaces; cf. \cite{Lockhart1985}]\label{def:Lockhart}
	The \emph{Lockhart--McOwen weighted Sobolev spaces with weight $\alpha \in \mathbb{R}$}, $L^p_{k,\alpha} (X)$, are the norm-completion of the smooth compactly supported functions with respect to the norm in \cref{eq:Lockhart}.
\end{definition}

The next result states that the twisted Dirac operator $\D$ is Fredholm for these Sobolev spaces. This is a standard result, stated for example in \cites{Marshall02,Karigiannis2012}. Alternatively, this theorem follows by translating all the setup into the cylindrical setting and using the results in \cites{Donaldson2002,Lockhart1985}. In fact, the results in \cite{Donaldson2002} also prove that the model operator on a cone admits a right inverse in this case.

\begin{theorem}\label{thm:DiracOperator}
	Let $\D$ be a Dirac type operator on an AC manifold as above. Then, there is a discrete set of weights $\mathcal{K}(\D)$ such that for all $\alpha \not\in \mathcal{K}(\D)$ and $k \in \mathbb{N}$, the operator 
	\begin{equation}
		\D : L^2_{k+1, \alpha+1} (X) \rightarrow L^2_{k, \alpha} (X),
	\end{equation} 
	is Fredholm, and
	\begin{equation}\label{FredholmAlternative}
		L^2_{k, \alpha} = \D ( L^2_{k+1, \alpha+1} (X) ) \oplus W_\alpha (X),
	\end{equation}
	where $W_\alpha \cong \ker (\D)_{-\alpha - n}$. Moreover, if $\alpha \geqslant - \tfrac{n}{2}$ equality holds as $\ker(\D)_{-\alpha - n} \subseteq L^2_{k,\alpha} (X)$.
\end{theorem}

\smallskip

\subsection{Mixed situation and split Sobolev spaces}

This subsection analyses the case when $q$ is asymptotically degenerate but does not vanish identically. More precisely, we consider the following situation.

\begin{hypoth}\label{hypoth:asymp}
	Suppose that outside a compact set $K$, there is a splitting 
	\begin{equation}\label{eq:Splitting_S}
		\mathcal{S}_E = \mathcal{S}_E^{||} \oplus \mathcal{S}_E^\perp
	\end{equation}
	such that:
	\begin{itemize}
		\item $q$ vanishes on $\mathcal{S}_E^{||}$ and is nondegenerate on $\mathcal{S}_E^\perp$. 
		\item $q$ is asymptotically parallel, meaning that $|r^j \nabla^j q|$ is a function as $\epsilon$ in \cref{eq:FunctionEps}. In particular, there is $q_C$ on some $\cS_{E_\infty}$ such that $\Upsilon^*\cS_E \cong \cS_{E_\infty}$.
		\item $\nabla$ is modeled at infinity on a connection $\nabla_\infty$ on $\cS_{E_\infty}$ pulled back from the link of the asymptotic cone. In particular, $\D$ is modeled on a conical operator $\D_C$ as in \Cref{prop:ModelInequality}.
		\item $\nabla_\infty$ vanishes on $\cS_{E_\infty}^{||}$ and is irreducible on $S_{E_\infty}^\perp$.
	\end{itemize}
\end{hypoth}

From now on, we always assume that these hypotheses are satisfied.

\begin{remark}
	This is the relevant case for finite mass, irreducible, $\rG_2$-monopoles with gauge group $\rG = \SU (2)$ as we are considering. Indeed, in that case a $\rG_2$-monopole $(\nabla,\Phi)$ satisfies $|\Phi| \to m >0$ at infinite and $q = \ad_{\Phi}$. Therefore, as already mentioned in the discussion opening \Cref{sec:BW_formulas}, outside a large compact set, we can split
	\begin{equation}
		\mathfrak{g}_P = \mathfrak{g}_P^{||} \oplus \mathfrak{g}_P^\perp,
	\end{equation} 
	where $\mathfrak{g}_P^{||} = \ker (\ad_\Phi)$. This in turn induces a splitting on $\cS_{\g_P}$ as in \cref{eq:Splitting_S}.
\end{remark}

In this situation we consider a mixed type family of Sobolev spaces which along the end agree with the standard ones on $\cS_E^\perp$ and with the weighted ones on $\cS_E^{||}$. Given $s \in \Omega^0(X, \mathcal{S}_E)$ be supported along the end where \cref{eq:Splitting_S} is valid, then the respective components of $s$ be referred as $s^\Vert, s^\perp$.\\ 
With this discussion in mind, we now define the relevant function spaces.

\begin{definition}\label{def:FunctionSpaces}
	Let $\alpha \in \mathbb{R}$ and $k \in \mathbb{N}$. Define the $H_{k, \alpha}$-norm of a compactly supported $s \in \Omega^0(X,\cS_E)$ as
	\begin{equation}
		\Vert s \Vert_{H_{k, \alpha} (X)}^2 = \Vert s \Vert_{L^2_k(K)}^2 + \Vert s^\Vert \Vert_{L^2_{k,\alpha} (X - K)}^2 + \Vert s^\perp \Vert_{L^2_k (X - K)}^2,
	\end{equation}
	and the spaces $H_{k,\alpha} (X)$ as the completion of the smooth compactly supported sections in this norm.
\end{definition}

\begin{theorem}\label{thm:Main_Fredholm_Theorem}
	Let $k \in \mathbb{N}$ and $\alpha \in \mathbb{R}$. Then, there is a discrete set $\mathcal{K}(\D) \subseteq \mathbb{R}$ such that for $\alpha \notin \mathcal{K}(\D)$, the operator 
	\begin{equation}
		\cL :  H_{k+1,\alpha+1} (X) \rightarrow H_{k,\alpha} (X),
	\end{equation} 
	is Fredholm. In particular, there exist parametrices $P_L,P_R :  H_{k,\alpha} \rightarrow H_{k+1, \alpha+1}$, such that
	\begin{equation}\label{eq:Parametrix}
		\cL P_R = \Id + S_R \ , \ P_L \cL = \Id + S_L,
	\end{equation}
	with $S_R: H_{k+1,\alpha+1} (X) \rightarrow H_{k+1, \alpha+1} (X)$ and $S_L : H_{k,\alpha} (X) \rightarrow H_{k, \alpha} (X)$ compact operators.
\end{theorem}
\begin{proof}
	The statement is true if $X$ is replaced by a cone, as in that case $\cL$ consists on the direct sum of two operators as in the previous two subsections which are therefore Fredholm. As the direct sum of Fredholm operators is Fredholm the result holds on a model cone.\\
	The general case follows from a standard procedure, which constructs global parametrices by gluing those obtained for the model operators. This is illustrated below, in the construction of a global right parametrix $P_R$.\\
	Let $V_0 = X - K$ and $K \subseteq \cup_{i = 1}^N V_i$ be an open cover of $X$, such that there are local right inverses $P_i$ to the operator $\cL$, defined on some slightly larger open sets $U_i$ containing $V_i$. Moreover, suppose $K$ is big enough, so that on $U$, the operator $\cL$ is modeled on some conical operator $\cL_C = \D_C +q_C$ with $\nabla q_C = 0$. Let $\lbrace \beta_0 , \ldots , \beta_N \rbrace$ be a partition of unity subordinate to this cover.\\
	First, notice that one can change the operator $\cL$ over $V_0$ so that it is exactly conical as $\cL_C$. In fact, this amounts to subtract to $\cL$ the operator $T(s) = \beta_0 (\cL s - \cL_C (\beta_0 s)))$, which is a compact operator 
	\begin{equation}
		T: H_{k+1,\alpha+1} (X) \rightarrow H_{k,\alpha} (X),
	\end{equation} 
	and the Fredholm property is not affected by perturbations by compact operators.
	Then there is a parametrix $P_0$ constructed for $\cL_C$ satisfying the conditions in \cref{eq:Parametrix} with $\cL$ replaced by $\cL_C$. We now glue $P_0$ with the local inverses $P_i$. Following the approach in page 95 of \cite{Donaldson2002} we define the candidate for a global parametrix as 
	\begin{equation}
		P_R = \sqrt{\beta_0} P_0 \sqrt{\beta_0} + \sum\limits_{i \in I} \sqrt{\beta_i} P_i \sqrt{\beta_i},
	\end{equation} 
	and notice that even though $P_0$ and the $P_i$'s are not globally defined the expression above is. To check that $P_R$ is indeed a right parametrix for $\cL$ we compute
	\begin{equation}
		\cL P_R(s) = \sigma(\rd \sqrt{\beta_0} ) P_0 \sqrt{\beta_0} s + \sum_{i \in I} \sigma(\rd \sqrt{\beta_i}) P_i \sqrt{\beta_i} s + \sqrt{\beta_0} \cL P_0 \sqrt{\beta_0} s + \sum_{i \in I} \sqrt{\beta_i} \cL P_i \sqrt{\beta_i}s,
	\end{equation}
	where $\sigma$ denotes the higher order symbol of $\cL$ (which coincides with that of $\D$). The term in the first line is a compact operator $T': H_{k,\alpha} \rightarrow H_{k,\alpha}$ and, again, does not affect the Fredholm property. This follows from the fact that it is supported on a compact set where the derivatives of the $\beta$'s are non vanishing. Moreover, over this compact set, by elliptic regularity one can control the $L^2 (X)$ norms of the derivatives of $P_0 s$ and $P_is$ in terms of the $L^2 (X)$ norms of $s$. For the last two terms one can use $\cL P_0 = I+S_0 + T''$ for some compact operators $S_0$ and $T''$ over $V_0$; and $\cL P_i = I$ over $V_i$ to obtain
	\begin{align}
		\cL P_R(s) &= T'''(s) + \sqrt{\beta_0} (I+S_0) \sqrt{\beta_0}s + \sum\limits_{i \in I}\beta_i s\\
		&= s + T'''(s) + \sqrt{\beta_0} S_0 (\sqrt{\beta_0} s),
	\end{align}
	for $T''' = T'+T''$. Moreover since the last term is supported on the conical end where it agrees with $S_0$, which is a compact operator on these function spaces, the operator $T''' + \sqrt{\beta_0} S_0 \sqrt{\beta}$ is compact and this proves that $P_R$ is a right parametrix for $\cL$.
\end{proof}

\smallskip

\subsection{The case when $p>2$}

This subsection extends \Cref{thm:Main_Fredholm_Theorem} from $p = 2$ to $p >2$. The upshot is \Cref{thm:Main_Fredholm_Theorem_2} below. The relevant function spaces for the general situation are the ones in \Cref{def:FunctionSpaces} but constructed with $p >2$.

\begin{definition}\label{def:FunctionSpacesP}
	For $\alpha \in \mathbb{R}$, $k \in \mathbb{N}_+$ and $p \geqslant 2$ define the spaces $H^p_{k , \alpha} (X)$ to be the completion of the smooth compactly supported sections with respect to the norm given by
	\begin{equation}
		\Vert s \Vert_{H^p_{k, \alpha} (X)}^p = \Vert s \Vert^p_{L^p_k (K)} + \Vert s^\Vert \Vert_{L^p_{k,\alpha} (X - K)}^p + \Vert s^\perp \Vert_{L^p_k (X - K)}^p,
	\end{equation}
	where $K \subseteq X$ is a large compact set outside of which the splitting $\mathcal{S}_E = \mathcal{S}_E^\Vert \oplus \mathcal{S}_E^\perp$ is well defined. 
\end{definition}

\begin{remark}
	Notice that $H^2_{k, \alpha} (X) = H_{k,\alpha} (X)$ in the notation from the previous section. 
\end{remark}

The main result of this section is the following.

\begin{theorem}\label{thm:Main_Fredholm_Theorem_2}
	Let $k \in \mathbb{N}$ and $p \geqslant 2$, there is a discrete set $\mathcal{K}(\D) \subseteq \mathbb{R}$ such that for $\alpha \notin \mathcal{K}(\D)$ and $\alpha \geqslant -n/2$
	\begin{equation}\label{eq:Operatorp}
		\cL :H^p_{k+1,\alpha+1} (X) \rightarrow H^p_{k,\alpha} (X),
	\end{equation}
	is a Fredholm operator.
\end{theorem}

This is proven by showing that the parametrices $P_R,P_L$ obtained for $p = 2$ in \Cref{thm:Main_Fredholm_Theorem} extend to bounded operators with $S_R,S_L$ compact operators when regarded as operators on the relevant $H^p_{k,\alpha}$-spaces with $p>2$.

\begin{proposition}\label{prop:p2}
	Let $\alpha \geqslant -n/2$ be such that $\alpha \not\in \mathcal{K} (\D)$ and $-n-\alpha \not\in \mathcal{K}(\D)$. Then, the parametrices $P_R$ and $P_L$ for $\cL$ obtained in \Cref{thm:Main_Fredholm_Theorem}, extend to bounded operators
	\begin{equation}
		P_R, P_L : H^p_{0,\alpha} (X) \rightarrow H^p_{1,\alpha+1} (X),
	\end{equation}
	such that 
	\begin{equation}
		\cL P_R = \Id + S_R, \ \text{and} \  \ P_L \cL = \Id + S_L,
	\end{equation} 
	with $S_R: H^p_{0,\alpha} (X) \rightarrow H^p_{0,\alpha} (X)$ and $S_L : H^p_{1,\alpha+1} (X) \rightarrow H^p_{1,\alpha+1} (X)$ compact operators.
\end{proposition}

\begin{proof}
	Consider the restricted operator 
	\begin{equation}
		\cL \vert_{ (\ker (\cL) \cap H_{0, \alpha+1} (X))^{\perp_{L^2 (X)}}}: (\ker (\cL) \cap H_{0, \alpha+1} (X))^{\perp_{L^2 (X)}} \rightarrow \im (\cL \vert_{ (\ker (\cL) \cap H_{0, \alpha+1} (X))^{\perp_{L^2 (X)}}}).
	\end{equation}
	Any complement to $\im (\cL \vert_{ (\ker (\cL) \cap H_{0, \alpha+1} (X))^{\perp_{L^2 (X)}}})$ in $H_{0,\alpha} (X)$ is isomorphic to the kernel of the adjoint operator $\cL^*$ in the dual space $H_{0,-n-\alpha} (X)$. In the range when $\alpha \geqslant -n/2$ we have $H_{0,-n-\alpha} (X) \subseteq H_{0,\alpha} (X)$ and so we have
	\begin{equation}
		\im (\cL \vert_{ (\ker (\cL) \cap H_{0, \alpha+1} (X))^{\perp_{L^2 (X)}}}) = (\ker(\cL^*) \cap H_{0,-n-\alpha} (X))^{\perp_{L^2 (X)}}.
	\end{equation}
	Having this said, we now lay out how the operators $S_L$ and $S_R$ are constructed in order to then prove that they are compact. The parametrix $P_L$ in the statement is obtained by constructing a left inverse to $\cL \vert_{ (\ker(\D) \cap H_{0, \alpha+1} (X))^{\perp_{L^2 (X)}}}$, then $S_L$ is minus the projection onto $\ker (\cL) \cap H_{0, \alpha+1} (X)$, which is finite dimensional as $\cL$ is Fredholm for $p = 2$ due to \Cref{thm:Main_Fredholm_Theorem}. In the same way, $P_R$ is obtained by constructing a right inverse to $\cL$ as an operator onto $(\ker(\cL^*) \cap H_{0,-n-\alpha})^{\perp_{L^2 (X)}}$ and so $S_R$ is minus the projection onto $\ker(\cL^*) \cap H_{0,-n-\alpha} (X)$ which is finite dimensional as $\cL^* = \D - q$ is Fredholm for $p = 2$. As projections into finite dimensional subspaces are compact operators, both $S_R$ and $S_L$ are compact.
	
	Next, we turn to the proof that the parametrices $P_R, P_L$ do extend to bounded operators from $H^p_{0, \alpha} (X)$ to $H^p_{0, \alpha+1} (X)$. Let us chose a big compact geodesic ball $K$, and then we have:
	\begin{enumerate}
		\item Over the big compact set $K \subseteq X$, the spaces $H^p_{k, \alpha} (X)$ can be taken to agree with the usual $L^p_k (X)$ ones. We then fix a finite open cover $\lbrace V_i \rbrace_{i \in I}$, where the standard Calderon--Zygmund inequalities hold. These are
		\begin{align}
			\Vert \nabla g \Vert^p_{L^p (V_i)} &\lesssim \Vert \cL g \Vert^p_{L^p (V_i')}  + \Vert g \Vert^p_{L^p (V_i')}  \\ 
			\Vert \ g \Vert^p_{L^p (V_i')} &\lesssim \Vert \cL g \Vert^p_{L^p (V_i'')}  + \Vert  g \Vert^p_{L^2 (V_i'')} ,
		\end{align}
		where $V_i'' \supset V_i' \supset V_i$ are open. The reason why we chose to arrange them in this manner is that they can now be combined into
		\begin{equation}
			\Vert  g \Vert^p_{L^p_1(V_i)}  \lesssim \Vert \cL g \Vert^p_{L^p (V_i'')}  + \Vert g \Vert^p_{L^2 (V_i'')}.
		\end{equation}
		Inserting $g = P_R f$ into the above inequality and using $\cL P_R = \Id + S_R$, gives
		\begin{align}
			\Vert P_R f \Vert_{L^p (V_i)}^p &\lesssim \Vert \cL P_R f \Vert_{L^p (V_i'')}^p + \Vert P_R f \Vert_{L^2 (V_i'')}^p  \\
			&\lesssim \Vert f \Vert_{L^p (V_i'')}^p + \Vert S_R f \Vert_{L^p (V_i'')}^p + \Vert P_R f \Vert_{L^2 (V_i'')}^p  .
		\end{align}
		Then the fact that $P_R$ is bounded for $p = 2$, $L^p (V_i'') \subseteq L^2 (V_i'')$ for $p \geqslant 2$, and $S_R$ is compact and hence bounded for $p \geqslant 2$ combine to further give 
		\begin{equation}
			\Vert P_R f \Vert_{L^p (V_i)}^p \lesssim \Vert f \Vert_{L^p (V_i'')}^p,
		\end{equation}
		thus showing that $P_R:L^p (K) \to L^p (K)$ is bounded.
		
		\item In the noncompact end $X - K$, the operator $\cL = \D_C + q$ is modeled on a conical operator $\cL_C = \D_C+q_C$ as in \Cref{hypoth:asymp}. The rest of the proof requires \Cref{lem:EquivalentNorm,cor:LptoLp} below. For now assume these hold, then from \Cref{cor:LptoLp} one can use the alternative form of the $H^p_{1, \alpha+1} (X)$ norm
		\begin{equation}
			\Vert g \Vert^p_{H^p_{1, \alpha+1} (X)} = \Vert \cL g \Vert^p_{H^p_{0, \alpha} (X)} + \Vert g \Vert^p_{H^p_{0 , \alpha +1} (X)}.
		\end{equation}
		Insert into this $g = P_Rf$ with $f \in H^p_{0, \alpha} (X)$ and use $\cL P_R = \Id + S_R$, gives
		\begin{equation}
			\Vert P_R f  \Vert^p_{H^p_{1, \alpha+1} (X)} = \Vert f + S_R f \Vert^p_{H^p_{0, \alpha} (X)} + \Vert P_R f \Vert^p_{H^p_{0,\alpha+1} (X)}.
		\end{equation}
		By using the generalized Young inequality and the fact that $S_R: H^p_{0, \alpha} (X) \rightarrow H^p_{0, \alpha} (X)$ is compact, the first term can be bounded above by a term of the form $\Vert f \Vert^p_{H^p_{0, \alpha} (X)}$. As for the second term, it is guaranteed by \Cref{cor:LptoLp} that it can equally be bounded above by $\Vert f \Vert^p_{H^p_{0, \alpha} (X)}$ thus showing that the parametrix $P_R: H^p_{0, \alpha} (X) \rightarrow H^p_{1, \alpha+1} (X)$ is bounded.
	\end{enumerate}
	
	Combining these two pieces finishes the proof of \Cref{prop:p2}.
\end{proof}

The rest of this section focuses on proving \Cref{lem:EquivalentNorm,cor:LptoLp} used in the proof of \Cref{prop:p2}.

\begin{lemma}\label{lem:EquivalentNorm}
	The norm $H^p_{k+1,\alpha+1} (X)$ is equivalent to the norm $\Vert \cdot \Vert$ defined by
	\begin{equation}
		\Vert f \Vert^p = \Vert \cL f \Vert_{H^p_{k,\alpha} (X)}^p +\Vert  f \Vert_{H^p_{0,\alpha+1} (X)}^p.
	\end{equation}
\end{lemma}

\begin{proof}
	The result follows from induction in $k$. We only do the first step with $k = 1$ as the general one is similar. In this $k = 1$ case it is enough to show that $\Vert f \Vert^p$ can be bounded from above and below by fixed multiples of $\Vert f \Vert^p_{H^p_{1,\alpha+1} (X)}$.
	
	To prove the upper bound, we use $\cL f = \D f + q(f)$ and the generalized version of Young's inequality
	\begin{equation}\label[ineq]{ineq:Inter0}
		\Vert f \Vert^p \lesssim \Vert \D f \Vert^p_{H^p_{0,\alpha} (X)} +  \Vert q(f) \Vert_{H^{p}_{0,\alpha} (X)}^{p} + \Vert f \Vert_{H^p_{0,\alpha+1} (X)}^p.
	\end{equation}
	As $\vert \D \vert \lesssim \vert \nabla f \vert$, the first term above has an upper bound of the form $ \Vert \nabla f \Vert_{H^p_{0,\alpha} (X)}^p$. Furthermore, $ \vert q(f) \vert \lesssim \vert f^\perp \vert$ and given the the weights do not affect the $f^\perp$ component we also have $\Vert f^\perp \Vert_{H^{p}_{0,\alpha} (X)}^{p} = \Vert f^\perp \Vert_{H^{p}_{0,\alpha+1} (X)}^p$. Combining these, we find
	\begin{equation}
		\Vert f \Vert^p \lesssim \Vert \nabla f \Vert_{H^p_{0,\alpha} (X)}^p  + \Vert  f \Vert_{H^p_{0,\alpha+1} (X)}^p = \Vert  f \Vert_{H^p_{1,\alpha+1} (X)}^p .
	\end{equation}

	To prove the lower bound on $\Vert f \Vert^p$, we need to establish an inequality as
	\begin{equation}\label[ineq]{ineq:CZwithH}
		\Vert f \Vert^p_{H^p_{1, \alpha+1} (X)} \lesssim \Vert \cL f \Vert^p_{H^p_{0,\alpha} (X)} + \Vert f \Vert^p_{H^p_{0,\alpha+1} (X)} ,
	\end{equation}
	We prove this by splitting it into two cases when $f$ is either in $\mathcal{S}_E^\Vert$ or $\mathcal{S}_E^\perp$. In the first case, i.e. $f \in \mathcal{S}_E^\Vert$, we have $\cL f = \D f$ and $\Vert f \Vert_{H^p_{k,\alpha} (X)} = \Vert f \Vert_{L^p_{k,\alpha} (X)}$ and given that $\D$ is an elliptic asymptotically conical operator modeled on $\D_C$, there is an inequality
	\begin{equation}\label[ineq]{ineq:Inter1}
		\Vert f \Vert^p_{L^p_{1, \alpha+1} (X)} \lesssim \Vert \D f \Vert^p_{L^p_{0,\alpha} (X)} + \Vert f \Vert^p_{L^p_{0,\alpha+1} (X)} .
	\end{equation}
	This is immediate from a change of coordinates to the cylindrical setting and the result of Lockhart--McOwen in \cite{Lockhart1985}.\\
	On the other hand, if $f \in \mathcal{S}_{E}^\perp$, $\Vert f \Vert_{H^p_{k, \alpha} (X)} = \Vert f \Vert_{L^p_k (X)}$ and to bound $\Vert \nabla f \Vert^p_{L^p (X)}$ by above we use the fact that $L^p (X) = L^p_{0, -n/p} (X)$. This gives 
	\begin{equation}
		\Vert \nabla f \Vert^p_{L^p (X)} \leqslant \Vert \nabla f \Vert^p_{L^p (X)} + \Vert r^{-1} f \Vert^p_{L^p (X)} = \Vert  f \Vert^p_{L^p_{1, - n/p +1} (X)}.
	\end{equation}
	Then, using the weighted \cref{ineq:Inter1}, for the case $\alpha = - \tfrac{n}{p}$, gives
	\begin{equation}
		\Vert \nabla f \Vert^p_{L^p (X)} \lesssim \Vert \D f \Vert_{L^p (X)}^p + \Vert r^{-1} f \Vert_{L^p (X)}^p  \lesssim \Vert \cL f \Vert_{L^p (X)}^p +  \Vert f \Vert_{L^p (X)}^p + \Vert r^{-1} f \Vert_{L^p (X)}^p  \lesssim \Vert \cL f \Vert_{L^p (X)}^p + \Vert f \Vert_{L^p (X)}^p ,
	\end{equation}
	where in the second inequality we used $\D f = \cL f - q(f)$ and the fact that $q$ is bounded. The \cref{ineq:CZwithH} is now immediate from summing these two components.
\end{proof}

We use the above result in the analysis of a mixed norm designed to interpolate between the $L^p$-norm along the radial direction and the $L^2$-norm along the ``closed'' directions.

\begin{definition}
	Let $R \gg 1$ be such that the splitting $\cS_E = \cS_E^{||} \oplus \cS_E^\perp$ is well defined in the complement of $B_R$. We define the intermediate norm $\Vert \cdot \Vert_{H^{(p,2)}_{0, \alpha} (X)}$ by
	\begin{equation}
		\Vert f \Vert_{H^{(p,2)}_{0, \alpha} (X)}^p = || f||^p_{L^2 (K)} + \int\limits_R^\infty \left( r^{-\alpha p -n} \Vert f^\Vert \Vert_{L^2 (\Sigma_r)}^p + \Vert f^\perp \Vert_{L^2 (\Sigma_r)}^p \right) r^{-(n-1) \frac{p-2}{2}} \rd r,
	\end{equation}
	where the $L^2$-norms on the right hand side are with respect to the induced metric on $\Sigma_r$ defined as $\Upsilon (\lbrace r \rbrace \times \Sigma)$.
\end{definition}

\begin{lemma}\label{lem:IntermediateNormInequality}
	Let $p \geqslant 2$ and $\alpha \in \mathbb{R}$, then
	\begin{equation}
		\Vert f \Vert_{H^{(p,2)}_{0, \alpha} (X)} \lesssim \Vert f \Vert_{H^p_{0,\alpha} (X)},
	\end{equation}
	for all $f \in H^p_{0,\alpha} (X)$.
\end{lemma}

\begin{proof}
	When $f$ is supported in $K$ the inequality follows from the standard embedding $L^p (K) \hookrightarrow L^2 (K)$ and so we need only prove the case when $f$ is supported in $X - K$. Using again that for $p \geqslant 2$ and over compact sets, the $L^p$-norm is stronger than the $L^2$-norm. In fact, over $B_1 \subseteq \mathbb{R}^k$ we have $\Vert f \Vert_{L^2 (B_1)} \lesssim \Vert f \Vert_{L^p (B_1)}$, and then, by scaling, we get
	\begin{equation}
		\Vert f \Vert_{L^2 (B_r)} \lesssim r^{k \frac{p-2}{2p}} \Vert f \Vert_{L^p (B_r)},
	\end{equation} 
	for all positive $r$. Applying this scaling behavior of the $L^p$ norms, we find that 
	\begin{equation}
		\Vert f \Vert_{L^2 (\Sigma_r)}^p \lesssim r^{(n-1) \frac{p-2}{2}} \Vert f \Vert_{L^p (\Sigma_r)}^p.
	\end{equation} 
	Inserting this into the definition of the $H^{(p,2)}_{0, \alpha}$-norm above gives an upper bound with respect to the $H^{p}_{0, \alpha}$-norm.
\end{proof}

\begin{lemma}\label{lem:TwoIntInequalities}
	Let $p \geqslant 2$ and $\alpha \in \mathbb{R}$, then there is an inequality
	\begin{equation}
		\Vert P_R f \Vert_{H^{(p,2)}_{0, \alpha+1} (X)} \lesssim \Vert f \Vert_{H^{(p,2)}_{0,\alpha} (X)},
	\end{equation} 
	for all $f \in H^p_{0,\alpha} (X)$. Moreover, combining this with \Cref{lem:IntermediateNormInequality} we find
	\begin{equation}
		\Vert P_R f \Vert_{H^{(p,2)}_{0, \alpha+1} (X)} \lesssim \Vert f \Vert_{H^{p}_{0,\alpha} (X)}.
	\end{equation}
\end{lemma}

\begin{proof}
	As asserted in the statement, it is enough to prove the first inequality. Recall that $P_R$ is a bounded operator for $p = 2$, i.e. from $H_{0,\alpha} (X)$ to $H_{1, \alpha+1} (X)$ and so we need only prove the inequality for $f$ supported along the conical end. As in the previous proof, it is convenient to separately prove inequalities for the components of $f$ in $\mathcal{S}_E^\Vert$ and $\mathcal{S}_E^\perp$.
	
	We start with the case when $f \in \mathcal{S}_E^\Vert$, then the $H^p_{k, \alpha}$-norm is the standard Lockhart--McOwen weighted $L^p_{k, \alpha}$-norm, and by changing coordinates to $t = \log(r)$, the statement that $P_R$ is bounded from $H_{0, \alpha} (X) = L^2_{0,\alpha} (X)$ into $H_{0, \alpha+1} (X) = L^2_{0, \alpha+1} (X)$ reads
	\begin{equation}
		\int\limits_{\log(R)}^\infty \Vert e^{-t} P_R f \Vert^2_{L^2 (\Sigma, g_\Sigma)} e^{-2 \alpha t} \rd t \lesssim \int\limits_{\log(R)}^\infty  \Vert f \Vert^2_{L^2 (\Sigma, g_\Sigma)} e^{-2 \alpha t} \rd t.
	\end{equation}
	Equivalently, this statement can be formulated as saying that for all $T > \log(R)$, the assignment $e^{-\alpha t}f \mapsto e^{-(\alpha +1)T} (P_R f)(T)$ gives rise to a bounded map 
	\begin{equation}
		M_\alpha(T): L^2 ((\log(R), \infty) , L^2 (\Sigma, g_\Sigma)) \rightarrow L^2 (\Sigma, g_\Sigma),
	\end{equation}
	whose operator norm integrable in the parameter $T \in (\log(R), \infty)$.\\
	Still in the cylindrical setting, the fact that $f \in H^{(p,2)}_{0, \alpha} (X)$ means that 
	\begin{equation}
		e^{-\alpha t -(n-1)\frac{p-2}{2p}t}f(t) \in L^p ((\log(R), \infty), L^2 (\Sigma, g_\Sigma)).
	\end{equation} 
	Hence, we can use the fact that the family $M_\alpha(\cdot)$ has integrable operator norm and the map $L^1 \times L^p \hookrightarrow L^p$ is bounded along $ (\log(R), \infty) \times L^2 (\Sigma, g_\Sigma)$ to prove that
	\begin{equation}
		\Vert e^{-(n-1)\frac{p-2}{2p}T}( M_\alpha e^{-\alpha t} f(t))(T) \Vert^p_{L^p (\Sigma)}  \leqslant \Vert M_\alpha(T) \Vert^p \Vert  e^{-\alpha t -(n-1)\frac{p-2}{2p}t}f(t) \Vert^p_{L^p (\Sigma)}.
	\end{equation}
	Since $\Vert M_\alpha(T) \Vert_{L^1} \lesssim 1$ is bounded, we can change coordinates back to $r$ and obtain
	\begin{equation}
		\Vert P_R f \Vert^2_{H^{(p,2)}_{0, \alpha+1} (X)} \leqslant C \Vert  f \Vert^2_{H^{(p,2)}_{0, \alpha} (X)}.
	\end{equation}
	This proves that $P_R: H^{(p,2)}_{0, \alpha} (X) \rightarrow H^{(p,2)}_{0, \alpha+1} (X)$ is bounded for those components in $\mathcal{S}_E^\Vert$.
	
	We now turn to the situation when $f \in \mathcal{S}_E^\perp$ in which case the $H^p_{k, \alpha}$ norm is the standard $L^p_k$ one. The statement that $P_R$ is bounded from and into $L^2 (X)$ can equivalently be stated in the cylindrical setting, as follows. Using the measure $e^{nt}\rd t$ on $(\log(R), \infty)$, and all $T> \log(R)$, the assignment $f \mapsto (P_Rf)(T)$ gives a bounded map
	\begin{equation}
		P_R (T): L^2 ((\log(R), \infty), L^2 (\Sigma, g_\Sigma)) \rightarrow L^2 (\Sigma, g_\Sigma),
	\end{equation}
	and this $T$-parametrized family has integrable operator norm. Then, given $f \in H^{(p,2)}_{0, \alpha}$ which in the cylindrical setting means 
	\begin{equation}
		e^{ -(n-1)\frac{p-2}{2p}t}f(t) \in L^p ((\log(R), \infty), L^2 (\Sigma, g_\Sigma)),
	\end{equation} 
	using the measure $e^{nt}\rd t$ on $(\log(R) , \infty)$. Proceeding as before and combining the map $L^1 \times L^p \hookrightarrow L^p$ with the fact that the family $P_R(T)$ has integrable operator norm gives
	\begin{equation}
		\Vert e^{-(n-1)\frac{p-2}{2p}T}( P_R f(t))(T) \Vert^p_{L^p_{e^{nt}\rd t}}  \leqslant \Vert P_R(T) \Vert^p_{L^1} \Vert  e^{-(n-1)\frac{p-2}{2p}t} e^{-\alpha t}f(t) \Vert^p_{L^p_{e^{nt}\rd t}},
	\end{equation}
	with $\Vert P_R(T) \Vert^p_{L^1} \lesssim 1$. Back to the conical world this statement gets translated into
	\begin{equation}
		\Vert P_R f \Vert^2_{H^{(p,2)}_{0, \alpha+1}} \lesssim \Vert  f \Vert^2_{H^{(p,2)}_{0, \alpha}},
	\end{equation}
	proving the statement for those components in $\mathcal{S}_E^\perp$.
	
	The complete statement is obtained by putting together these two separate cases.
\end{proof}

\begin{lemma}\label{cor:LptoLp}
	There is an inequality of the form
	\begin{equation}
		\Vert P_R f \Vert_{H^{p}_{0, \alpha+1} (X)} \lesssim \Vert f \Vert_{H^p_{0,\alpha} (X)},
	\end{equation}
	which holds for all $f \in H^p_{0,\alpha} (X)$.
\end{lemma}

\begin{proof}
	Again, in this case it is enough to prove the statement for the case when $f$ is supported along the conical end. Recall the $H^p_{0, \alpha} (X)$ norm in \Cref{def:FunctionSpacesP}. It is convenient to have it rewritten as a sum of $L^p$-norms on conical annuli going off along end as follows
	\begin{align}
		\Vert g\Vert_{H^p_{0,\alpha+1} (U)}^p &= \int\limits_R^\infty \left( r^{-(\alpha +1) p -n} \Vert g^\Vert \Vert_{L^p (\Sigma_r)}^p +  \Vert g^\perp \Vert_{L^p (\Sigma_r)}^p \right) \rd r \\  \label{eq:SumNorm}
		&\cong \sum_{k \geqslant 1} \left( R^{-k((\alpha +1) p +n )} \Vert g^\Vert \Vert^p_{L^p (C_k)} + \Vert g^\perp \Vert^p_{L^p (C_k)} \right),
	\end{align}
	where $ \cong $ above denotes an equivalence of norms (which is straightforward to check) and $C_k = (R^k , R^{k+1}) \times \Sigma$ equipped with the conical metric $g_C = \rd r^2 + r^2 g_\Sigma = r^2 (\frac{\rd r^2}{r^2} + g_\Sigma)$. Notice that the conical annulus $C_{k+1}$ is obtained from $C_k$ by scaling with a factor of $R>1$. As usual, it is convenient to separate into components.
	
	First, one focuses on the components in $\mathcal{S}_E^\Vert$. Over the bounded annulus $C_1$, the standard Calderon--Zygmund inequalities give $ \Vert g \Vert_{L^p (C_1)}^p \lesssim \Vert \D g \Vert_{L^p (C_1')}^p + \Vert g \Vert_{L^2 (C_1')}^p  $, where $C_1' \supset C_1$ is a slightly larger annulus in the cone. This inequality is not scale invariant and scaling it gives
	\begin{equation}
		\Vert g \Vert_{L^p (C_k)}^p \lesssim R^{kp} \Vert \D g \Vert_{L^p (C_k')}^p + R^{-nk \frac{p-2}{2}}\Vert g \Vert_{L^2 (C_k')}^p,
	\end{equation}
	and in this component $\D = \cL$. Moreover, since $p>2$, $R^{-nk \frac{p-2}{2}} \leqslant R^{-(n-1)k \frac{p-2}{2}}$ and 
	\begin{equation}
		\Vert g \Vert_{L^2 (C_k)}^p \lesssim \int\limits_{R^k}^{R^{k+1}} \Vert g \Vert_{L^2 (\Sigma_r)}^p \rd r.
	\end{equation} 
	Then, the inequality above further gives  
	\begin{align}
		\Vert g \Vert_{L^p (C_k)}^p &\lesssim R^{kp} \Vert \cL g \Vert_{L^p (C_k')}^p + R^{-(n-1)k \frac{p-2}{2}}  \int\limits_{R^k-1}^{2R^{k+1}+1} \Vert g \Vert_{L^2 (\Sigma_r)}^p \rd r \\
		&\lesssim R^{kp} \Vert \cL g \Vert_{L^p (C_k')}^p +  \int\limits_{R^k-1}^{2R^{k+1}+1} r^{-(n-1) \frac{p-2}{2}} \Vert g \Vert_{L^2 (\Sigma_r)}^p \rd r.
	\end{align}
	Inserting this into \cref{eq:SumNorm} and recalling so that $g \in \cS_E^{||}$, gives 
	\begin{align}
		\Vert g \Vert_{H^p_{0,\alpha+1} (X)}^p &\lesssim \sum_{k \geqslant 1} R^{-k( (\alpha+1) p +n )} R^{pk} \Vert \cL g \Vert^p_{L^p (C_k)} + \int\limits_{R}^\infty r^{-(\alpha+1)p -n} \Vert g \Vert_{L^2 (\Sigma_r)}^p r^{-(n-1) \frac{p-2}{2}} \rd r \\
		&\lesssim \Vert \cL g \Vert_{H^p_{0,\alpha} (X)}^p + \Vert g \Vert_{H^{(p,2)}_{0,\alpha+1} (X)}^p .
	\end{align}
	Insert into this inequality $g = P_R f$, then by using $\cL P_R = I+ S_R$, the fact that $S_R$ is bounded from and into $H^p_{0, \alpha} (X)$ and \Cref{lem:TwoIntInequalities}, gives
	\begin{equation}
		\Vert P_R f \Vert_{H^p_{0,\alpha+1} (X)}^p \lesssim \Vert f + S_R f \Vert^p_{H^p_{0,\alpha+1} (X)} + \Vert P_R f \Vert^p_{H^{(p,2)}_{0,\alpha+1} (X)} \lesssim \Vert f \Vert^p_{H^{p}_{0,\alpha+1} (X)}.
	\end{equation}
	
	Next, we turn to those components in $\mathcal{S}_E^\perp$, recall that for these the map $q \in \Omega^0(\End(\mathcal{S}_E))$ is bounded below, i.e. $\vert q(g) \vert \geqslant c \vert g \vert$, for some $c>0$ independent of $g \in \mathcal{S}_E^\perp$. Then in any $C_k$, we have the inequalities
	\begin{equation}\label[ineq]{ineq:InterBoundedBelow}
		\Vert g \Vert^p_{L^p (C_k)} \lesssim \Vert q(g) \Vert^p_{L^p (C_k)} \lesssim \Vert \cL g \Vert^p_{L^p (C_k)} + \Vert \D g \Vert^p_{L^p (C_k)}.
	\end{equation}
	Moreover, as $\D:L^p_1(C_1) \rightarrow L^p (C_1)$ is bounded we find, after rescaling, 
	\begin{equation}
		\Vert \D g \Vert^p_{L^p (C_k)} \leqslant c_1 (\Vert \nabla g \Vert^p_{L^p (C_k')} + R^{-pk}\Vert g \Vert^p_{L^p (C_k')})
	\end{equation} 
	and given that $\D$ and $\cL$ are elliptic, the (rescaled) standard Calderon--Zygmund inequality gives
	\begin{equation}
		\Vert \nabla g \Vert^p_{L^p (C_k')} \lesssim \Vert  \D g \Vert^p_{L^p (C_k'')} + R^{-nk \frac{p-2}{2}} \Vert  g \Vert^p_{L^2 (C_k'')},
	\end{equation} 
	where $C_k'' \subseteq C_k' \supset C_k$ are nested annuli. Combining these yield
	\begin{equation}
		\Vert \D g \Vert^p_{L^p (C_k)} \lesssim \Vert \cL g \Vert^p_{L^p (C_k'')} + R^{-nk \frac{p-2}{2}} \Vert  g \Vert^p_{L^2 (C_k'')} + R^{-pk}\Vert g \Vert^p_{L^p (C_k')},
	\end{equation}
	and inserting this back into \cref{ineq:InterBoundedBelow} gives
	\begin{equation}\label[ineq]{ineq:InterBoundedBelow1}
		\Vert g \Vert^p_{L^p (C_k)} \lesssim \| \cL g \|^p_{L^p (C_k'')} +  R^{-nk \frac{p-2}{2}} \Vert  g \Vert^p_{L^2 (C_k'')} + R^{-pk}\Vert g \Vert^p_{L^p (C_k')} .
	\end{equation}
	Moreover since $p>2$ also in this case $R^{-nk \frac{p-2}{2}} \leqslant R^{-(n-1)k \frac{p-2}{2}} $ and one can dominate the second term in the right above by 
	\begin{equation}
		\int\limits_{R^k-1}^{R^{k+1}+1}  \Vert g \Vert_{L^2 (\Sigma_r)}^p r^{-(n-1) \frac{p-2}{2}} \rd r,
	\end{equation} 
	which is for components in $\mathcal{S}_E^\perp$ the $H^{(p,2)}_{0, \alpha+1}$ norm. Then, inserting \cref{ineq:InterBoundedBelow1} into the norm in \eqref{eq:SumNorm} for $g \in \mathcal{S}_E^\perp$ gives
	\begin{align}
		\Vert g \Vert^p_{H^p_{0, \alpha+1}} &\lesssim \sum_{k \geqslant 1} \Vert \cL g \Vert^p_{L^p (C_k'')} + \int\limits_R^\infty  \Vert g \Vert_{L^2 (\Sigma_r)}^p r^{-(n-1) \frac{p-2}{2}} \rd r + R^{-p} \int\limits_R^\infty \| g\|_{L^p (\Sigma_r)}^p \rd r \\
		&\lesssim \Vert \cL g \Vert^p_{H^{p}_{0, \alpha} (X)} + \Vert g \Vert^p_{H^{(p,2)}_{0, \alpha +1} (X)} + R^{-p} \|g\|_{H^p_{0,\alpha+1} (X)}^p ,
	\end{align}
	and for $R \gg 1$ we can reabsorb the last term into the first yielding
	\begin{equation}
		\Vert g \Vert^p_{H^p_{0, \alpha+1} (X)} \lesssim \Vert \cL g \Vert^p_{H^{p}_{0, \alpha} (X)} + \Vert g \Vert^p_{H^{(p,2)}_{0, \alpha +1} (X)},
	\end{equation}
	and notice that the weights $\alpha$ here are irrelevant but are introduced in order to use the appropriate notation. Following a similar strategy as in the previous case let $g = P_Rf$ in the inequality above. Then, using $\cL P_R = \Id + S_R$, that $S_R$ is bounded on $H^p_{0, \alpha} (X)$ and \Cref{lem:TwoIntInequalities} gives $\Vert P_R f \Vert^p_{H^p_{0, \alpha+1} (X)} \leqslant \Vert f \Vert^p_{H^p_{0, \alpha} (X)}$. The general result follows immediately from combining the inequalities in the two components.
\end{proof}

\begin{remark}
	Recall that restricted to the components in $\mathcal{S}_E^\Vert$, the operator $\cL$ coincides with the Dirac operator $\D$. Furthermore, along this component $H^p_{k, \alpha} (X) = L^p_{k, \alpha} (X)$ are the Lockhart--McOwen spaces. The results obtained in this subsection, when restricted to these components, also follow from the standard Lockhart--McOwen theory which could have been used instead; cf. \cite{Lockhart1985}. In fact, that was done when using \cref{ineq:Inter1} in the proof of \Cref{lem:EquivalentNorm}. However, such an inequality follows from the standard Calderon--Zygmund inequality: 
	\begin{equation}
		\Vert \nabla g \Vert^p_{L^p (C_1)} \lesssim \Vert L g \Vert^p_{L^p (C_1')} +  \Vert  g \Vert^p_{L^p (C_1')},
	\end{equation} 
	after rescaling the annulus $C_1$ to any other annulus, $C_k$, in a similar fashion to what we did in \Cref{cor:LptoLp}.
\end{remark}

\bigskip

\section{Weighted split Sobolev spaces: Sobolev Embeddings and Multiplication Maps}\label{sec:Sobolev_2}

Denote by $L^p_{k, \alpha} (X)$ the weighted spaces introduced in \Cref{def:Lockhart}. In order to develop a moduli theory for $\rG_2$-monopoles, we need to understand these spaces, in particular, to be able to handle the nonlinearities of the equations. The most relevant of these properties is the one stated in \Cref{prop:Multiplication} below. Its proof requires a combination of \Cref{lema1,lema3,lema4}.

\begin{lemma}\label{lema1}(Weighted H\"older Inequality)
	Let $\beta, \gamma \in \mathbb{R}$ and $\frac{1}{s_1} + \frac{1}{s_2} = \frac{1}{q}$, then the multiplication property $L^{s_1}_{0, \beta} (X) \times L^{s_2}_{0, \gamma} (X) \hookrightarrow L^q_{0, \gamma + \beta} (X)$ holds. In particular, if $\gamma \leqslant 0$, then $L^{s_1}_{0, \beta} (X) \times L^{s_2}_{0, \gamma} (X) \hookrightarrow L^q_{0,  \beta} (X)$.
\end{lemma}
\begin{proof}
	Let $f \in L^{s_1}_{0,\beta} (X), g \in L^{s_2}_{0,\gamma} (X)$, then using the definition of the weighted norms, rearranging the exponents and the usual H\"older inequality shows
	\begin{align}
		\Vert fg \Vert_{L^q_{0, \gamma + \beta} (X)} &= \Vert r^{-\frac{n}{q}-\gamma -\beta} f g  \Vert_{L^q (X)} \\
		&= \| ( r^{-\frac{n}{s_1}-\beta} f ) \ ( r^{-\frac{n}{s_2}-\gamma} g ) |_{L^q (X)} \\
		&\leqslant \| r^{-\frac{n}{s_1}-\beta} f  ||_{L^{s_1} (X)} \ \| r^{-\frac{n}{s_2}-\gamma} g \|_{L^q (X)} \\
		&= \Vert f \Vert_{L^{s_1}_{0, \beta} (X)} \Vert g \Vert_{L^{s_2}_{0, \gamma} (X)},
	\end{align}
	thus showing the first statement. In particular, when $\gamma \leqslant 0$, then $L^q_{0, \gamma + \beta} (X) \hookrightarrow L^q_{0, \beta} (X)$.
\end{proof}

\smallskip

\begin{lemma}\label{lema4}
	Let $p_2 > p_1$ and $\gamma_2 > \gamma_1$, then for all $s \in [p_1,p_2]$ and $\gamma \geqslant \max_{i = 1,2} \lbrace \frac{n}{p_i} - \frac{n}{s} + \gamma_i \rbrace$, there is an inclusion $L^{p_1}_{0, \gamma_1} (X) \cap L^{p_2}_{0, \gamma_2} (X) \hookrightarrow L^{s}_{0, \gamma} (X)$.
\end{lemma}

\begin{proof}
	First one notices that since $p_1 < s < p_2$, then for all $g \in L^{p_1} (X) \cap L^{p_2} (X)$, it holds that $\Vert g \Vert_{L^s (X)} \lesssim \Vert g \Vert_{L^{p_1} (X)} + \Vert g \Vert_{L^{p_2} (X)}$. Let $f \in L^s_{0, \gamma} (X)$, then
	\begin{align}
		\Vert f \Vert_{L^s_{0, \gamma} (X)} &= \Vert r^{-\gamma-\frac{n}{s}} f \Vert_{L^s (X)} \\
		&\lesssim \Vert r^{-\gamma-\frac{n}{s}} f  \Vert_{L^{p_1} (X)} + \Vert r^{-\gamma-\frac{n}{s}} f \Vert_{L^{p_2} (X)} \\
		&= \Vert f  \Vert_{L^{p_1}_{0,\gamma + \frac{n}{s}-\frac{n}{p_1}} (X)} + \Vert f  \Vert_{L^{p_2}_{0,\gamma + \frac{n}{s}-\frac{n}{p_2}} (X)} .
	\end{align}
	Since $\gamma \geqslant \max_i \lbrace \frac{n}{p_i} - \frac{n}{s} + \gamma_i \rbrace$, one has $\gamma + \frac{n}{s}-\frac{n}{p_i} \geqslant \gamma_i$ for $i = 1,2$ and 
	\begin{equation}
		\Vert f \Vert_{L^s_{0, \gamma} (X)} \lesssim \Vert f  \Vert_{L^{p_1}_{0,\gamma_1} (X)} + \Vert f  \Vert_{L^{p_2}_{0,\gamma_2} (X)}.
	\end{equation}
\end{proof}

\begin{lemma}\label{lema3}
	Let $\beta \in \mathbb{R}$, $p \in \left[ \frac{n}{2}, n \right)$ and $k \in \mathbb{N}_+$. Then, the following hold
	\begin{itemize}
		\item $L^p_k (X) = \bigcap\limits_{i = 0}^k L^p_{i, -\frac{n}{p}+i} (X)$
		\item $L^p_{k+1,\beta} (X) \hookrightarrow L^q_{k , \beta} (X)$, for $q \leqslant \frac{np}{n-p}$.
		\item $ L^p_{k+1, \loc} (X) \hookrightarrow C^{k-1}_{\loc} (X)$ and $ L^p_{k+1,\beta} (X) \hookrightarrow C^{k-1}_{\beta} (X)$,
	\end{itemize}
\end{lemma}

\begin{proof}
	The first bullet is a consequence of the definition of the $L^p_{k,\beta}$-norms in \eqref{eq:Lockhart}. The case $k = 0$ amounts to $\Vert f \Vert_{L^p_{0,-n/p} (X)} = \Vert r^{-n/p+n/p} f \Vert_{L^p (X)}$ and the general case follows from an induction argument. Here we simply do the case $k = 1$ with the remaining general following from a similar case. Write for the norm in the right hand side
	\begin{align}
		\Vert f \Vert^p_{L^p_{0,-n/p} (X)} + \Vert f \Vert^p_{L^p_{1,-n/p+1} (X)} &= \Vert f \Vert^{p}_{L^p} + \Vert r^{-1} f \Vert^{p}_{L^p} + \Vert \nabla f \Vert^{p}_{L^p} \\
		&= \Vert f \Vert^p_{L^p_1 (X)} + \Vert r^{-1} f \Vert^p_{L^p (X)} \\
		&\lesssim 2 \Vert f \Vert^p_{L^p_1 (X)},
	\end{align}
	while clearly 
	\begin{equation}
		\Vert f \Vert^p_{L^p_{0,-n/p} (X)} + \Vert f \Vert^p_{L^p_{1,-n/p+1} (X)} \gtrsim \Vert f \Vert^p_{L^p_1 (X)}.
	\end{equation} 
	The remaining two bullets are particular instances of the standard weighted Sobolev embedding theorems (Theorem $4.17$ in \cite{Marshall02}). To apply them one just needs to check that $1 - \frac{n}{p} > - \frac{n}{q}$ and $k+1 - \frac{n}{p} \geqslant k-1$.
\end{proof}

\begin{lemma}\label{lem:ThereIsALimit}
	Let $p \geqslant n/2$, $\beta \leqslant -1 $ and $\xi \in \Omega^0 (X, V)$ with $\nabla \xi \in L^p_{1, \beta} (X)$. Then $\lim_{r \rightarrow \infty }\xi(r)$ exists and is equal to a $\nabla_\infty$-parallel continuous section $\xi_\infty$ of $\cS_{E_\infty}$.
\end{lemma}

\begin{proof}
	Along the conical end $X - K$ we write
	\begin{equation}
		\nabla \xi = \frac{\partial \xi^\Vert}{\partial r} \otimes \rd r + \nabla \xi^\Vert +\frac{\partial \xi^\perp}{\partial r } \otimes \rd r + \nabla \xi^\perp,
	\end{equation}
	and as the summands are linearly independent as sections of $\Lambda^1 \otimes \cS_E$, each of them has its norm bounded by that of $\nabla \xi$. Since $\nabla_\infty$ is irreducible on $\cS_{E_\infty}^\perp$ we find that along the conical end so is $\nabla$ on $\cS_{E}^\perp$ component. Hence, there is a Poincar\'e type inequality on the slices $\Sigma_1 = \Upsilon(\lbrace 1 \rbrace \times \Sigma)$, which can be written as $\Vert \xi^\perp \Vert_{L^p (\Sigma_1)} \lesssim \Vert \nabla \xi^\perp \Vert_{L^p (\Sigma_1)}$. Scaling this inequality gives 
	\begin{equation}
		\Vert \xi^\perp \Vert_{L^p (\Sigma_r)} \lesssim r \Vert \nabla \xi^\perp \Vert_{L^p (\Sigma_r)} \lesssim r \Vert \nabla \xi \Vert_{L^p (\Sigma_r)},
	\end{equation}
	on each $\Sigma_r = \Upsilon ( \lbrace r \rbrace \times \Sigma )$. This together with the hypothesis that $\nabla \xi \in L^p_{k, \beta}$ shows that
	\begin{equation}
		\int\limits_1^\infty r^{-(\beta +1) p} \Vert \xi^\perp \Vert_{L^p (\Sigma_r)}^p \frac{\rd r}{r^n} \leqslant \int\limits_1^\infty c r^{- \beta p} \Vert \nabla \xi \Vert_{L^p (\Sigma_r)}^p \frac{\rd r}{r^n} < \infty.
	\end{equation}
	Scaling the metric on $(1, \infty)_r \times \Sigma$ to the cylindrical metric $r^{-2}g = \rd t^2 + g_\Sigma$, where $t = \log (r)$. This implies that as $t \rightarrow \infty$, all three $e^{-tp (\beta+1)}\xi^\perp$, $e^{-tp (\beta+1)}\nabla \xi^\perp$ and $e^{-tp (\beta+1)} \nabla \nabla \xi^\perp$ converge in the $L^p$-norm to zero, over the intervals $(t, t+1) \times \Sigma$, equipped with the cylindrical metric $\rd t^2 + g_\Sigma$. Since $-(\beta+1) \geqslant 0$, one concludes that as $t \rightarrow \infty$, $\xi$ converges to zero in $L^p_2 (X)$ over these intervals equipped with the fixed cylindrical metric. Using, the Sobolev embedding $L^p_2 (X) \hookrightarrow C^0 (X)$, which holds for $p\geqslant n/2$, one concludes that $\xi^\perp$ converges uniformly to zero.\\
	For the other component, i.e. $\xi^\Vert$ one has $\vert \frac{\partial \xi^\Vert}{\partial r} \vert \leqslant \vert \nabla \xi \vert$ and using this together with the H\"older inequality into
	\begin{equation}
		\int\limits_1^\infty \Big\vert \frac{\partial \xi^\Vert}{\partial r } \Big\vert \ \rd r \leqslant \int\limits_1^\infty r^{-\beta p}\vert \nabla \xi  \vert^p \rd r \int\limits_1^\infty r^{\beta p'} \rd r,
	\end{equation}
	where $p' = p/(p-1)$ is the conjugate exponent. The first integral is bounded above by $\Vert \nabla \xi \Vert_{L^p_{0, \beta} (X)}^p$. The second one is $\int_1^\infty r^{\frac{\beta p}{p-1}} \rd r$ and since $\beta \leqslant -1 < 1/p-1 = (1-p)/p$ one concludes this integral is finite. It follows that there is a limit $\xi_\infty$ to which $\xi^\Vert$ converges.
\end{proof}

\begin{remark}\label{rem:ThereIsALimit2}
	Using the same proof as in \Cref{lem:ThereIsALimit} but replacing the Sobolev embedding $L^p_2 (X) \hookrightarrow C^0 (X)$ over $\Sigma^{n-1}$ by the Sobolev embedding $L^2_k (X) \hookrightarrow C^0 (X)$, which holds for $k > \tfrac{n}{2}$, leads to the following conclusion. If $k > \frac{n}{2}$ and $\xi \in \Omega^0 (X, V)$ with $r^{j-1} \nabla^j \xi \in L^2 (X)$ for all $1 \leqslant j \leqslant k$. Then $\xi$ converges to a $\nabla_\infty$ parallel section $\xi_\infty$ of $\cS_{E_\infty}$.
\end{remark}\label{prop:ThereIsALimit2}

\begin{corollary}\label{cor:lema3}
	Let $k \in \mathbb{N}$ and $p \geqslant \frac{n}{2}$. Then, for any $\xi \in \Omega^k(X, \mathfrak{g}_P)$ with $r \nabla_0 \xi \in H^p_{k, \beta+1} (X)$ we have  $\xi^\perp \in L^p_{k+1} (X)$. Furthermore, in case $\beta<-1+ \frac{1}{p}$, then $\xi$ converges to a $\nabla_\infty$ parallel section $\xi_\infty$ of $\cS_{E_\infty}$.
\end{corollary}

\begin{proof}
	Since $r \nabla_0 \xi \in H^p_{k, \beta+1} (X)$, one knows $r \nabla_0 \xi^\perp \in L^p_k (X)$ and the same proof as that of the beginning of \Cref{lem:ThereIsALimit} shows that $\xi^\perp \in L^p_{k+1} (X)$ and converges to zero as $r \rightarrow \infty$. The other component follows from the fact that $r \nabla_0 \xi^\Vert \in L^p_{k,\beta+1} (X)$ is equivalent to $\nabla_0 \xi^\Vert \in L^p_{k, \beta} (X)$. Then, one can repeat the final part of the proof of \Cref{lem:ThereIsALimit} and notice that the argument there using H\"older's inequality works for $\beta < -1 + \frac{1}{p}$.
\end{proof}

As the conclusion of the previous lemmata we finally arrive at the main result of this subsection.

\begin{proposition}\label{prop:Multiplication}
	Let $p \in [\frac{n}{2}, n)$, $\alpha = 1-n/p$, and $N(\cdot , \cdot)$ a bilinear map satisfying
	\begin{equation}
		N( \cS_{E}^{||} , \cS_{E}^{||} ) = 0, \ \ N( \cS_{E}^{||} , \cS_{E}^\perp ) \subseteq \cS_{E}^\perp , \ \text{and} \ N( \cS_{E}^\perp , \cS_{E}^\perp ) \subseteq \cS_{E}^{||}.
	\end{equation}
	Then, $[\cdot , \cdot ]$ gives rise to a continuous multiplication map
	\begin{equation}
		N(\cdot , \cdot): H^p_{1, \alpha} (X) \times H^p_{1, \alpha} (X) \hookrightarrow H^p_{0, \alpha-1} (X).
	\end{equation}
\end{proposition}

\begin{proof}
	Let $\chi , \xi \in H^p_{1, \alpha} (X)$ and $q = \frac{np}{n-p}$, then by definition $\chi^\Vert , \xi^\Vert \in L^p_{1, \alpha} (X)$, which using the embedding in the second bullet of \Cref{lema3} lie in $L^q_{0, \alpha} (X)$. On the other hand, by the first bullet in that same Lemma we find that $\chi^\perp , \xi^\perp \in L^p_1 (X) = L^p_{0, -n/p} (X) \cap L^p_{1, -n/p+1} (X)$, and again by the second bullet $L^p_{1, -n/p+1} (X) \subseteq L^q_{0,-n/p+1} (X)$. In conclusion,
	\begin{equation}
		\chi^\Vert , \xi^\Vert \in  L^p_{0,\alpha} \cap L^q_{0, \alpha} (X) \ , \ \chi^\perp , \xi^\perp \in  L^p_{0, -n/p} (X) \cap L^q_{0,-n/p +1} (X).
	\end{equation}
	By the hypothesis, the term $N(\chi^\Vert , \xi^\Vert)$ vanishes and
	\begin{equation}
		N(\chi, \xi ) = N(\chi^\perp , \xi^\perp) + ( N(\chi^\Vert , \xi^\perp) + N(\chi^\perp , \xi^\Vert) ),
	\end{equation}
	with the first term lying in $\cS_{E}^{||}$ while both the second and third lie in $\cS_{E}^\perp$. So it is enough to show that $N(\chi^\perp , \xi^\perp) \in L^p_{0, \alpha -1} (X)$ and $N(\chi^\Vert , \xi^\perp) , N(\chi^\perp , \xi^\Vert) \in L^p (X) = L^p_{0, -n/p} (X)$.\\
	We start by analyzing the term $N(\chi^\perp , \xi^\perp)$, by using \Cref{lema1} twice in the forms 
	\begin{equation}
		L^p_{0,-n/p} (X) \times L^p_{0,-n/p} (X) \subseteq L^{p/2}_{0,-2n/p} (X), \ \text{and} \ \ L^q_{0,-n/p +1} (X) \times L^q_{0,-n/p +1} (X) \subseteq L^{q/2}_{0,-2n/p +2} (X).
	\end{equation} 
	Then, $N(\chi^\perp , \xi^\perp) \in L^{p/2}_{0,-2n/p} (X) \cap L^{q/2}_{0,-2n/p +2} (X)$ and using \Cref{lema4} with $p_1 = p/2$, $\gamma_1 = -2n/p$, $p_2 = q/2$, $\gamma_2 = -2n/p + 2$ and $s = p$ gives that $N(\chi^\perp , \xi^\perp) \in L^p_{0, \alpha-1} (X)$ for all $\alpha$ such that
	\begin{equation}
		\alpha -1 \geqslant \max \Big\lbrace \frac{2n}{p} - \frac{n}{p} - \frac{2n}{p} , \frac{2n}{q} - \frac{n}{p} - \frac{2n}{p} +2  \Big\rbrace = -\frac{n}{p}.
	\end{equation}
	Next, we turn to the terms in $\cS_{E}^\perp$. For this and apply again \Cref{lema1} twice, now in the form 
	\begin{equation}
		L^q_{0, \alpha} (X) \times L^q_{0, - n/p +1} (X) \subseteq L^{q/2}_{0, \alpha-n/p+1} (X), \ \text{and} \ \  L^q_{0, \alpha} (X) \times L^{p}_{0, - n/p} (X) \subseteq L^{\frac{np}{2n-p}}_{0, \alpha-n/p} (X).
	\end{equation} 
	Then $N (\chi^\Vert , \xi^\perp), N (\chi^\perp , \xi^\Vert) \in  L^{q/2}_{0, \alpha-n/p+1} (X) \cap L^{\frac{np}{2n-p}}_{0, \alpha-n/p} (X)$ and now we use \Cref{lema4} with $p_1 = np/(2n-p)$, $\gamma_1 = \alpha-n/p$, $p_2 = q/2$, $\gamma_2 = \alpha-n/p+1$ and $s = p$, which gives that $[\chi^\Vert , \xi^\perp], [\chi^\perp , \xi^\Vert] \in L^p (X) = L^p_{0, -n/p} (X)$. At this point we mention that the condition to apply \Cref{lema4} is that
	\begin{equation}
		\max \Big\lbrace \frac{2n-p}{np} - \frac{n}{p} + \alpha - \frac{n}{p} , \frac{2n}{q} - \frac{n}{p} + \alpha - \frac{n}{p} + 1  \Big\rbrace = -\frac{n}{p} \leqslant -\frac{n}{q},
	\end{equation}
	which holds by the condition that $p \geqslant n/2$ arises. One must remark that this condition is further required for the Sobolev embeddings in \Cref{lema3} to hold and the condition that $p < n$ is required in order for $p_1 = np/(2n-p) < p$ and lemma \Cref{lema4} to apply in the second case above.
\end{proof}

\bigskip

\section{Moduli theory}\label{sec:Moduli}

\subsection{A discussion of possible moduli spaces of $\rG_2$-monopoles}

In this section we discuss a few possible possibilities for the moduli space of $\rG_2$-monopoles on an asymptotically conical $\rG_2$-manifolds with finite mass and fixed monopole class.

Recall that a finite mass monopole $(\nabla,\Phi)$ with monopole class $\alpha \in H^2(\Sigma,\mathbb{Z})$ is modeled at infinity on a reducible pair $(\nabla_\infty,\Phi_\infty)$ on $P_\infty \rightarrow \Sigma$ as in \Cref{thm:main3} or \Cref{thm:Main_Theorem_2}. We fix a framing at infinity
\begin{equation}\label{eq:framing}
	\eta: \Upsilon^* P \vert_{X - K} \rightarrow \pi^* P_\infty,
\end{equation}
where $\Upsilon$ is as in \Cref{def:AC} and $\pi: C(\Sigma) \rightarrow \Sigma$ denotes the projection to the second factor. Let $[(\nabla_\infty, \Phi_\infty)]$ denote the gauge equivalence class of this pair and define
\begin{align}
	\Gamma_\infty &= \lbrace g \in \Aut(P_\infty) \ \ \vert \ \ g\cdot (\nabla_\infty, \Phi_\infty) = (\nabla_\infty, \Phi_\infty) \rbrace, \\ 
	\gamma_\infty &= \lbrace \xi \in \Gamma( \mathfrak{g}_{P_\infty} ) \ \ \vert \ \ \nabla_\infty \xi = 0 = [\xi , \Phi_\infty] \rbrace.
\end{align}
Then $\Gamma_\infty$ are the gauge transformations of $P_\infty$ which preserve the boundary data and $\gamma_\infty$ its Lie algebra. Furthermore, we define $\mathcal{G}$ to be the group of continuous gauge transformations, which have a limit $g_\infty = \lim_{r \rightarrow \infty} g(r) \in \mathcal{G}_\infty$. It comes equipped with an evaluation map $\ev: \mathcal{G} \rightarrow \mathcal{G}_\infty$ which associates to $g \in \cG$ its limit at infinity. Using the framing \eqref{eq:framing}, we can define 
\begin{equation}
	\mathcal{G}(0) \coloneqq \ker(\ev) \subseteq \Gamma \coloneqq \ev^{-1}(\Gamma_\infty) \subseteq \cG.
\end{equation}

There are two possible approaches to setting up the moduli theory:

\begin{enumerate}
	\item Consider gauge equivalence classes of pairs $(\nabla, \Phi)$ on $P$ that are asymptotic to a pair in $[(\nabla_\infty, \Phi_\infty)]$.
	\item Fix the representative $(\nabla_\infty, \Phi_\infty) \in [(\nabla_\infty, \Phi_\infty)]$ and consider pairs $(\nabla,\Phi)$ asymptotic to this representative modulo the action of $\Gamma \subseteq \mathcal{G}$. 
\end{enumerate}

The automorphism group of the boundary data $\Gamma_\infty $ is isomorphic to a subgroup $ H \subseteq \SU (2)$, i.e. it is either trivial or isomorphic to $\rU (1)$.

\begin{remark}
	Recall that $\mathfrak{g}_{P_\infty}\cong \underline{\mathbb{R}} \oplus L^2 (X)$, where $L$ is a line bundle over $\Sigma$. However, if $(\nabla,\Phi)$ is irreducible, then $||\nabla \Phi ||_{ L^2} \neq 0$  and the energy formula in \Cref{thm:Energy_Formula} shows that $\alpha = -2 \pi i c_1(L)$ must be nontrivial in which case $H \cong \rU (1)$. Tracing through the definitions, this isomorphism can be seen more explicitly as follows.
	\begin{itemize}
		\item If $ g \in \Aut(P_\infty)$ and $ g\cdot \Phi_\infty = \Phi_\infty $, then one can write $ g = e^{if\Phi_\infty}$, for some $f \in C^\infty(\Sigma, \mathbb{R/Z})$. Moreover, if $g$ is further supposed to preserve the connection then it must be constant, this gives an isomorphism $\Gamma_\infty \cong \rU (1)$.
		\item If $\xi \in \mathfrak{g}_{P_\infty}$ and $[\xi, \Phi_\infty] = 0$, then $ \xi = f \Phi_\infty $ for $f \in C^\infty(\Sigma, \mathbb{R})$ and if $\nabla_\infty\xi = 0$ then $f$ must be constant. This gives an isomorphism $\gamma_\infty \cong \mathbb{R}$.
	\end{itemize}
	
\end{remark}

It is also useful to consider a slightly larger moduli space which fibers over these ones with fibre $\Gamma_\infty$. Then, consider the moduli space of configurations to be those pairs $(\nabla, \Phi)$ which are asymptotic to $(\nabla_\infty, \Phi_\infty)$ modulo the action of $\mathcal{G}(0)$. Any implementation of this idea gives a moduli space of configurations, which fibers over the previous ones with fiber $\Gamma_\infty \cong H$.

\begin{remark}
	There is also one other way of constructing such a moduli space which comes with the framing $\eta$ incorporated in the definition at the expense of considering a slightly larger gauge group. Consider triples $(\nabla,\Phi,\eta)$ of configurations and a framing $\eta$ modulo the action of $\Gamma$. Here $\Gamma$ acts on the framing in a nontrivial way and this is what accounts for increasing the gauge group from $\mathcal{G}(0)$ to $\Gamma$.
\end{remark}

Let $(\nabla_0, \Phi_0)$ be a connection and an Higgs field on $P$ which along the conical end, $X - K$, is asymptotic to pullbacks of $(\nabla_\infty, \Phi_\infty)$ via the framing $\eta$ as in \eqref{eq:framing}. The adjoint action of $\Phi_0$ induces a splitting of $\mathfrak{g}_P$ along the conical end as in the beginning of \Cref{sec:BW_formulas}
\begin{equation}\label{Adsplitting}
	\mathfrak{g}_{P} \vert_{X - K} \cong \mathfrak{g}_{P}^\Vert \oplus \mathfrak{g}_{P}^\perp,
\end{equation}
where $\mathfrak{g}_{P}^\Vert = \ker( \ad_{\Phi_0})$ and $\mathfrak{g}_{P}^\perp$ its orthogonal. So one can uniquely split sections $ \chi \in \Omega^k(X - K, \mathfrak{g}_P)$ as $\chi = \chi^\Vert + \chi^\perp$, for $\chi^\Vert \in \Omega^k(X - K , \mathfrak{g}_{P}^\Vert)$ and $\chi^\perp \in \Omega^k (X - K, \mathfrak{g}_{P}^\perp)$.

\smallskip

\subsection{Moduli of Configurations}

This subsection defines and constructs moduli spaces of configurations $(\nabla, \Phi)$ witha fixed mass and monopole class. So we fix $(\nabla_0,\Phi_0)$ and construct $H^p_{k,\alpha}$-spaces as in \Cref{def:FunctionSpaces} using this pair and
\begin{equation}
	\cS_{E} = (\Lambda^1 \oplus \Lambda^0) \otimes \mathfrak{g}_P,
\end{equation}
with the decomposition into $||$ and $\perp$ being that induced on $\mathfrak{g}_P$ along the conical end. The upshot of this subsection is \Cref{thm:ModuliConf} which gives the moduli space of configurations the structure of a smooth Banach manifold. Then the boundary conditions defined by a finite mass monopole are preserved in
\begin{equation}
	\mathcal{A}^p_{k, \alpha} = \lbrace \nabla = \nabla_0 + a \ \ \vert \ \ a \in H^{p}_{k, \alpha} (X) \rbrace \ , \ \mathcal{H}^p_{k, \alpha} = \lbrace \Phi = \Phi_0 + \phi \ \ \vert \ \ \phi \in H^{p}_{k, \alpha} (X) \rbrace.
\end{equation}

Now we define the configuration space of $\rG_2$-monopoles.

\begin{definition}\label{def:Configuration_Space}
	Let $p >0$, $k \in \mathbb{N}$ and $\alpha \in \mathbb{R}$. Then, we define
	\begin{equation}
		\mathcal{C}^p_{k, \alpha} \coloneqq \mathcal{A}^p_{k, \alpha} \times \mathcal{H}^p_{k, \alpha},
	\end{equation} 
	which we refer to as the space of configurations.
\end{definition}

The induced topology of these spaces, in principle, depends on the background configuration $(\nabla_0, \Phi_0)$ and on $p, k$, and $\alpha$. A gauge transformation $g \in \cG \cap L^p_{k+1, \loc} (X)$ acts on a configuration $(\nabla_0 + a , \Phi_0 + \phi)$ via
\begin{equation}\label{gaugecomplicated}
	\left( \nabla_0 + g (\nabla_0 g^{-1}) + g a g^{-1} , \Phi_0 + (g\Phi_0 g^{-1} - \Phi_0) + g\phi g^{-1} \right),
\end{equation}
and two configurations in $\mathcal{C}^p_{k, \alpha} $ be considered equivalent if related by such a $g \in \mathcal{G} \cap L^p_{k+1, \loc} (X)$. We now describe the necessary setup in order to view this equivalence relation as generated by the action of a Banach Lie Group. Set
\begin{align}
	\mathcal{G}^p_{k, \alpha} &\coloneqq \lbrace g \in L^p_{k+1, \loc} (X) \ \ \vert \ \  r \nabla_0 g \in H^p_{k, \alpha+1} (X) \rbrace, \\
	\Lie (\mathcal{G}^p_{k, \alpha}) &\coloneqq \lbrace \xi \in \Omega^0(X, \mathfrak{g}_P) \ \ \vert \ \ r \nabla_0 \xi \in H^p_{k, \alpha+1} (X) \rbrace,
\end{align}
which we topologize as follows: The pointwise exponential defines a map, $\exp$, is a now map from $\Lie (\mathcal{G}^p_{k, \alpha})$ to $\mathcal{G}^p_{k, \alpha}$. For each $\epsilon > 0$ define 
\begin{equation}
	V_\epsilon = \lbrace \xi \in \Lie (\mathcal{G}^p_{k, \alpha}) \ \ \vert \ \ \Vert r \nabla_0 \xi \Vert_{H^p_{k, \alpha+1} (X)} \leqslant \epsilon \rbrace,
\end{equation}
and let the topology on $\mathcal{G}^p_{k, \alpha}$ be generated by the image under the exponential of the open sets $V_\epsilon \subseteq \Lie (\mathcal{G}^p_{k, \alpha})$ together with their translations.

\begin{proposition}
	Let $p \in [\tfrac{n}{2},n)$, $\alpha = - \tfrac{n}{p} + 1$, then the following hold
	\begin{enumerate}
		\item With the topology defined above $\mathcal{G}^p_{1, \alpha} $ is a Banach Lie group with Lie algebra $\Lie (\mathcal{G}^p_{1, \alpha})$.
		
		\item If one further supposes that $p < \frac{n+1}{2}$, then there is a surjective evaluation homomorphism $\ev: \mathcal{G}^p_{1, \alpha} \rightarrow \Gamma_\infty$, with derivative $\rd \ev :  \Lie (\mathcal{G}^p_{1, \alpha}) \rightarrow \gamma_\infty$.
		
		\item $\mathcal{G}^p_{1, \alpha} $ acts smoothly in $\mathcal{C}^{p}_{1, \alpha}$.
	\end{enumerate}
\end{proposition}

\begin{proof}
	Start by noticing that if $g \in \mathcal{G}^p_{1,\alpha}$, then $g \in L^p_{2, \loc} (X)$ and since one is working in a range where $p \geqslant \tfrac{n}{2}$, the third bullet in \Cref{lema3} applies and $g \in C^0_{\loc} (X)$, i.e. these gauge transformations are continuous.
	\begin{enumerate}
		
		\item To prove that $\mathcal{G}^p_{1, \alpha}$ is a Banach Lie group we must show that pointwise multiplication and inversion are well defined. Then, with the above topology $\Lie (\mathcal{G}^p_{1, \alpha})$ is its Lie algebra.
		
		\begin{enumerate}
			
			\item To prove that multiplication is well defined let $g,h \in  \mathcal{G}^{p}_{1, \alpha}$, i.e. $r \nabla_0 g, r \nabla_0 h \in H^p_{1, \alpha+1} (X)$ and one needs to show that 
			\begin{equation}
				r \nabla_0 (gh) = r (\nabla_0 g ) h + r g \nabla_0 h \in H^p_{1, \alpha +1} (X),
			\end{equation}
			for all $l \leqslant k$. The gauge transformations are continuous and $r \nabla_0 h \in  H^p_{1, \alpha +1} (X)$, so it follows that  $r g \nabla_0 h \in H^p_{1, \alpha +1} (X)$ and the same applies to $r (\nabla_0 g ) h$. Alternatively one uses the Sobolev embedding in the second bullet of \Cref{lema3}, which gives
			\begin{align}
				r \nabla_0 h^\Vert, \: r \nabla_0 g^\Vert &\in L^p_{1, \alpha+1} (X) \subseteq L^q_{0, \alpha+1} (X), \\
				r \nabla_0 h^\perp, \: r \nabla_0 g^\perp &\in L^p_{1, -n/p+1} (X) \subseteq L^q_{0, -n/p+1} (X),
			\end{align}
			and using $\alpha = 1-n/p$, 
			\begin{equation}
				r \nabla_0 h,  r \nabla_0 g \in L^p_{0, - n/p+1} (X) \cap L^q_{0,-n/p+1} (X).
			\end{equation} 
			and the multiplication map in \Cref{lema1} guarantees that 
			\begin{equation}
				r \nabla_0 g \nabla_0 h \in  L^p (X) \subseteq H^p_{0,\alpha} (X).
			\end{equation}
			
			\item To prove $g^{-1} \in \mathcal{G}^p_{1, \alpha}$ notice that $\nabla_0 g^{-1} = - g^{-1} (\nabla_0 g) g^{-1}$. Then proceeding as before, separating terms and using $g,g^{-1}\in C^0_{\loc} (X)$ and \Cref{lema1,lema4} yield $r \nabla_0 g^{-1} \in H^p_{1, \alpha+1} (X)$.
			
		\end{enumerate}
		
		\item Let $g \in \mathcal{G}^p_{1, \alpha}$, then $r \nabla_0 g \in H^p_{1, \alpha+1} (X)$, i.e. $(\nabla_0 g)^\Vert \in L^p_{1, \alpha} (X)$ and $L^p_{0,1-n/p} (X)$ and $(\nabla_0 g)^\perp \in L^p_{1, -n/p} (X)$. Then, \Cref{lem:ThereIsALimit} shows that $(\nabla_0 g)^\perp \rightarrow 0$, but the same does not apply to the other component.  However, the last part of the argument in that Lemma can be used and we now repeat it here. Notice that $\nabla_0 g \in L^p_{0,-n/p+1} (X)$, then H\"older's inequality gives
		\begin{equation}
			\int\limits_1^\infty \big\vert \frac{\partial g}{\partial r} \big\vert \rd r \leqslant \int\limits_1^\infty \vert r^{n/p-1} \nabla_0 g \vert^p \rd r \int\limits_1^\infty r^{\frac{p}{p-1}(1-n/p)} \rd r.
		\end{equation}
		The first integral is bounded above by $\Vert \nabla_0 g \Vert^p_{L^p_{0,-n/p+1} (X)}$ while the second is finite if and only if $p < \frac{n+1}{2}$. Hence in this case this proves the existence of $g_\infty \in \mathcal{G}_\infty$ such that $g \rightarrow g_\infty$ and $\nabla_\infty g_\infty = 0$ (i.e. $g_\infty \in \Gamma_\infty$). Using a bump function it is straightforward to check that the evaluation maps, given by taking the limits, are surjective.
		
		\item To check the action of $\mathcal{G}^p_{1,\alpha}$ on $\mathcal{C}^p_{1,\alpha}$ is well defined, one needs to prove that
		\begin{equation}
			g (\nabla_0 g^{-1}) + g a g^{-1}, \qandq (g\Phi_0 g^{-1} - \Phi_0) + g\phi g^{-1}
		\end{equation}
		are in $H^p_{1, \alpha} (X)$. For the terms $gag^{-1}$, $g \phi g^{-1}$ and $g \nabla_0 g^{-1} = -(\nabla_0 g) g^{-1}$ notice that $(a, \phi) \in H^p_{1 , \alpha} (X)$, $g \in C^0$ as it is in $L^p_{2, \loc} (X)$, and $r \nabla_0 g \in H^p_{1, \alpha+1} (X)$. Then, repeating the arguments in the proof of the first item proves that these are $H^p_{1, \alpha} (X)$. One is now left with analyzing $(g\Phi_0 g^{-1} - \Phi_0)$, for which one requires again the conclusion of the second item, above. Namely that if $g \in \cG$ and $\xi \in L(\cG)$ are such that $r \nabla_0 g, r \nabla_0 \xi \in H^p_{1, \alpha+1} (X)$, then $g, \xi$ converge to limits $g_\infty \in \Gamma_\infty$ and $\xi_\infty \in \gamma_\infty$. Moreover, one has $\xi^\perp \in L^p_2 (X)$ by \Cref{cor:lema3} and writing $g = e^{\xi}$ gives
		\begin{equation}
			g\Phi_0 g^{-1} - \Phi_0 = \sum\limits_{k = 1}^\infty \frac{1}{k!} \ad (\xi)^k (\Phi_0),
		\end{equation}
		and the multiplication map in \Cref{lema3}, used in the same way as before, show that the $k \geqslant 2$ terms are in $H^p_{1, \alpha} (X)$ if and only if $[\xi, \Phi_0] \in H^p_{1, \alpha} (X)$. 
		Along the conical end, we can write $[\xi, \Phi_0] = [\xi^\perp, \Phi_0]$ and since $\Phi_0$ is smooth and bounded and $\xi^\perp \in L^p_2 (X)$ it is indeed true that $ [\xi, \Phi_0] \in H^p_{1, \alpha} (X)$. First, the convergence of the series above is follows from $\vert [\xi, \Phi_0] \vert \leqslant \vert \xi^\perp \vert$ which converges to zero as $r \rightarrow \infty$. It is therefore bounded and the convergence of the series is guaranteed by the term $\tfrac{1}{k!}$.
		
		Conversely, if $(\nabla,\Phi)$ and $g \cdot (\nabla, \Phi)$ are both in $\mathcal{C}^p_{1, \alpha}$ are related by an $L^p_{2, \loc} (X)$ gauge transformation $g = e^{\xi}$, then actually $e^{\xi} \in \mathcal{G}^p_{1, \alpha}$ one rewinds the previous arguments. First, the fact that $[\xi, \Phi] = [\xi^\perp, \Phi_0] + \ldots \in L^p_2 (X) \subseteq L^p_1 (X)$ implies $r \nabla_0 \xi^\perp \in L^p_1 (X)$. Second, the fact that $g^{-1} (\nabla_0 g) = \nabla_0 \xi \in H^p_{1,\alpha} (X)$ implies that $r \nabla_0 \xi^\Vert \in L^p_{1, \alpha +1} (X)$. Put these two together to conclude that $r \nabla_0 \xi \in H^p_{1, \alpha+1} (X)$ and so $g \in \mathcal{G}^p_{1, \alpha}$.
	\end{enumerate}
\end{proof}

Using the second item in this proposition, we define
\begin{equation}\label{def:LieGroup}
	\mathcal{G}^p_{k, \alpha} (0) \coloneqq \ker (\ev),
\end{equation}
which is a Banach Lie subgroup of $\mathcal{G}^p_{k, \alpha}$ consisting of the gauge transformations converging to the identity along the end. For $p \in [n/2, n)$ and $\alpha = 1-n/p$ its Lie algebra is the Lie subalgebra of $\Lie (\mathcal{G}^p_{k, \alpha} )$ consisting of those sections which decay, i.e 
\begin{equation}
	\Lie (\mathcal{G}^p_{k, \alpha} (0)) = H^p_{k+1,\alpha +1} (X , \mathfrak{g}_P).
\end{equation}

\begin{lemma}\label{lem:HodgeLemma}
	Let $\beta \in \mathbb{R} - \lbrace -1 \rbrace$ and $(\nabla, \Phi) \in \mathcal{C}^p_{k, \beta}$. Denote by $\rd_{\nabla}^*$ the formal $L^2 (X)$ adjoint of the operator $\nabla$ and for all $\beta$ extend $\nabla,\nabla^\ast$ to operators
	\begin{equation}
		\nabla, \: (\nabla^\beta)^* : L^p_{k+1, \beta+1}(X, \mathfrak{g}_P) \rightarrow L^p_{k, \beta}(X, T^*X \otimes \mathfrak{g}_P).
	\end{equation}
	Then, for $\beta \neq -1$, there is a constant $c>0$ and an inequality $\Vert \nabla \eta \Vert_{L^p_{0,\beta} (X)} \geqslant c \Vert \eta \Vert_{L^p_{0,\beta+1} (X)}$, and so a decomposition
		\begin{equation}\label{eq:DiractSumDecomposition}
			L^p_{k, \beta} (X, T^*X \otimes \mathfrak{g}_P) = \ker(\rd_\nabla^*) \cap L^p_{k,\beta} (X) \oplus \im(\nabla).
		\end{equation}
\end{lemma}

\begin{proof}
	For all $p,k,\beta$ the map $r^{-\beta-1} : L^p_{k, \beta+1+l} (X) \rightarrow L^p_{k,l} (X)$, which multiplies a section by $r^{-\beta-1}$, is an isomorphism of Banach spaces. Conjugation with it gives then an equivalence of linear operators
	\begin{equation}
		\begin{matrix}
			L^p_{k+1, \beta+1} & \xrightarrow{\nabla} & L^p_{k, \beta} \\
			\downarrow & & \downarrow \\
			L^p_{k+1,0} & \xrightarrow{\nabla^\beta} & L^p_{k,-1}
		\end{matrix},
	\end{equation}
	with 
	\begin{equation}
		\nabla^{\beta} = r^{-(\beta+1)} \circ \nabla \circ r^{\beta+1} = (\beta +1 )\frac{\rd r}{r} + \nabla.
	\end{equation}
	For simplicity we only present the proof of the $p = 2$ case, as in this case it is easy to complete squares. As $K$ is compact and $\nabla$ is irreducible on $K$, one can combine inequalities of Kato and Poincar\'e to get 
	\begin{equation}
		\Vert \nabla \eta \Vert_{L^2 (K)} \geqslant c_1 \Vert \eta \Vert_{L^2 (K)},
	\end{equation} 
	for some $c_1 > 0$ and all $\eta$ compactly supported in the interior of $K$. Moreover, as $r$ is bounded on $K$, this holds equally well for $\nabla^\beta = \nabla$. Then one needs to prove a similar inequality for a section $\eta$ which is  supported on the conical end $X - B_R$ one writes $\eta = \eta^\Vert + \eta^\perp \in L^2_{1,0} (X)$ and splitting $\nabla^{\beta} \eta$ into orthogonal components we compute
	\begin{align}
		\Vert \nabla^{\beta} \eta \Vert_{L^2_{0,-1}(X - K)}^2 &= \int\limits_R^\infty \frac{\rd r}{r} \int\limits_\Sigma \vert r \nabla^{\beta} \eta \vert^2 \ \vol_\Sigma \\
		&= \int\limits_R^\infty \rd r \int\limits_\Sigma \left( r \big\vert \frac{ \partial \eta}{\partial r} + \frac{\beta +1}{r} \eta \big\vert^2  + r \vert \nabla_0 \eta^\Vert \vert^2 + r \vert \nabla_0 \eta^\perp \vert^2 \right)  \ \vol_\Sigma 
	\end{align}
	In computing a lower bound for this we ignore the terms $ r \vert \nabla_0 \eta^\Vert \vert^2$ and the term $r \vert \frac{\partial \eta}{\partial r} \vert^2$ which appears when one expands the square. Also, when one expand the square there is a mixed term appearing, however as this is $2(\beta+1) \langle \eta , \frac{\partial \eta}{\partial \rho} \rangle = (\beta+1)\frac{\partial \vert \eta \vert^2}{\partial \rho}$ and since $\eta$ is compactly supported on $X - K$, one can integrate by parts and this term vanishes. One is left with
	\begin{equation}
		\Vert \nabla^{\beta} \eta \Vert_{L^2_{0,-1}(X - K)}^2 \geqslant \int\limits_1^\infty \rd r \int\limits_\Sigma \left( \frac{(\beta +1)^2}{r} \vert \eta  \vert^2 + r \vert \nabla_0 \eta^\perp \vert^2 \right) \ \vol_\Sigma.
	\end{equation}
	To handle this let $\Sigma_r = \Upsilon (\lbrace r \rbrace \times \Sigma)$, then the irreducibility of the connection $\nabla_\infty$ on $\mathfrak{g}_P^\perp$, gives a Poincar\'e type inequality, which after scaling is $\Vert \nabla_\infty \eta^\perp \Vert_{L^2 (\Sigma_{\rho})}^2 \geqslant c_2 r^{-2} \Vert \eta^\perp \Vert_{L^2 (\Sigma_{\rho})}^2$ for some constant $c_2 >0$. Moreover, as the connection $\nabla_0$ is asymptotic to $\nabla_\infty$ one can assume the same inequality holds for $\nabla_0$ and large $r$. Inserting this into the previous inequality yields
	\begin{equation}
		\Vert \nabla^{\beta} \eta \Vert_{L^2_{0,-1}(X - K)}^2 \geqslant \int\limits_1^\infty \rd r \int\limits_\Sigma \left( \frac{(\beta +1)^2}{r} \vert \eta^\Vert \vert^2 + \frac{c_2 + (\beta +1)^2}{r} \vert \eta ^\perp \vert^2 \right) \ \vol_\Sigma \gtrsim (1+ \beta)^2 \Vert \eta \Vert^2_{L^2_{0,0}(X - K)}.
	\end{equation}
	Combining this with the similar inequality one has on $K$, gives the inequality in the first item of the statement.
	
	As a consequence of this Poincar\'e type inequality $\nabla^\beta$ has closed image and the decomposition in the theorem follows. Recall that the operator $\nabla^\beta$ above is equivalent to $\nabla : L^2_{1,\beta+1} (X) \rightarrow L^2_{0,\beta} (X)$, so this one has  closed image. Then the same is true for $\nabla : L^p_{k+1,\beta+1} (X) \rightarrow L^p_{k,\beta} (X)$, which gives the decomposition \eqref{eq:DiractSumDecomposition}. Using the weighted inner product $\langle \cdot , \cdot \rangle_{L^2_{0,\beta} (X)}$ one can identify a copy of cokernel of $\rd_A$ with the orthogonal complement, i.e. the kernel of its adjoint 
	\begin{equation}
		(\nabla^\beta)^* = r^{2(\beta+1)+n} \circ \nabla^* \circ r^{-2\beta -n} = (2\beta +n ) \iota_{r \frac{\partial}{\partial r}} + \nabla^*.
	\end{equation}
\end{proof}

\begin{remark}
	The proof above gives a bound $\Vert \nabla \eta \Vert_{L^p_{0,\beta} (X)} \geqslant c \Vert \eta \Vert_{L^p_{0,\beta+1} (X)}$ with an explicit constant $c = \vert 1 + \beta \vert$. For $\beta = -n/2$ this gives back Hardy's inequality
	\begin{equation}
		\Vert \nabla \eta \Vert^2_{L^2 (X)} \geqslant \left( \frac{n-2}{2} \right)^2 \Vert r^{-1} \eta \Vert^2_{L^2 (X)}.
	\end{equation}
	Actually this gives the best possible constant on any asymptotically Euclidean manifold.
\end{remark}

\begin{corollary}\label{cor:HodgeLemma2}
	For $\beta \neq -1$, the operator 
	\begin{align}\label{eq:da}
		\cL_0 :H^p_{k+1,\beta+1} (X,\mathfrak{g}_P) & \rightarrow H^p_{k,\beta} (X, (\Lambda^0 \oplus \Lambda^1) \otimes \mathfrak{g}_P) . \\
		\xi & \mapsto (-\nabla \xi , [\xi, \Phi]),
	\end{align}
	has closed image. Using the notation $H^p_{k, \beta} (X)$ for the right hand side in \cref{eq:da}, there is an orthogonal decomposition
	\begin{equation}
		H^p_{k, \beta} (X) = \ker(\cL_0^{*}) \oplus \im(\cL_0).
	\end{equation}
	Where $\cL_0^{*}(a, \phi) = -\nabla^{*} a + [ \Phi, \phi]$.
\end{corollary}

\begin{proof}
	This proof combines the inequality in the previous \Cref{lem:HodgeLemma} with $|[\Phi,\xi]| \geqslant c |\xi^\perp|$, which holds sufficiently far out along the end. As $\Vert \cL_0 (\xi) \Vert_{H^2_{0,\beta} (X)} = \Vert \nabla \xi \Vert_{H^2_{0,\beta} (X)}^2 + \Vert [\Phi, \xi] \Vert_{H^2_{0,\alpha} (X)}^2$ we immediately conclude that $\cL_0$ has closed image.
\end{proof}

\begin{definition}
	A configuration $(\nabla, \Phi)$ is said to be irreducible if $\ker(\cL_0) = 0$.
\end{definition}

\begin{theorem}\label{thm:ModuliConf}
	Let $p \in [n/2, n)$ and $\alpha = 1-n/p$. Then, the quotient spaces 
	\begin{equation}
		\tilde{\mathcal{B}}^p_{1,\alpha} = \mathcal{C}^p_{1, \alpha} / \mathcal{G}^p_{1,\alpha}(0), \ \text{and} \ \ \mathcal{B}^p_{1,\alpha} = \mathcal{C}^p_{1, \alpha} / \mathcal{G}^p_{1,\alpha},
	\end{equation}
	inherit the structure of Banach manifolds with the property that
	\begin{equation}
		\mathcal{B}^p_{1,\alpha} = \tilde{\mathcal{B}}^p_{1,\alpha} / \Gamma_\infty.
	\end{equation}
	Moreover, the subset $\left( \mathcal{B}^p_{1,\alpha} \right)^* \subseteq \mathcal{B}^p_{1,\alpha}$ consisting of the image of the irreducible configurations is a smooth Banach manifold.
\end{theorem}

\begin{proof}
	To prove that $\tilde{\mathcal{B}}^p_{1,\alpha} = \mathcal{C}^p_{1, \alpha} / \mathcal{G}^p_{1,\alpha}(0)$ is a Banach manifold one constructs local slices to the action of $\mathcal{G}^p_{1,\alpha}(0)$ using the inverse function theorem. Then these slices can be used as charts for $\tilde{\mathcal{B}}^p_{1,\alpha}$. Let $\epsilon > 0$ and define the slice candidates as
	\begin{equation}
		T_{(\nabla, \Phi), \epsilon} = \lbrace (a,\phi) \in H^p_{1,\alpha} (X) \ \ \vert \ \  \nabla^* a -[\Phi_0, \phi] = 0 \ , \  \Vert (a, \phi) \Vert_{H^p_{1,\alpha} (X)} < \epsilon \rbrace.
	\end{equation}
	Then, in order to prove that these are actual slices one needs to show that the map
	\begin{equation}
		h : T_{(\nabla, \Phi), \epsilon} \times \mathcal{G}^p_{1,\alpha}(0) \rightarrow \mathcal{C}^p_{1, \alpha},
	\end{equation}
	which for $g = e^{\xi}$ sufficiently close to the identity, sends $\left( (a, \phi), g \right)$ to the gauge equivalent configuration 
	\begin{equation}
		h ((a,\phi),g) = g \cdot (\nabla + a , \Phi + \phi),
	\end{equation} 
	is an isomorphism onto an open set around $(A,\Phi)$. This can be proved using the inverse function theorem, by simply showing that the derivative
	\begin{align}
		\rd h = \id \oplus \cL_0 : \left( \ker(\cL_0^{ *}) \cap H^{p}_{1,\alpha} (X) \right) \oplus H^p_{1, \alpha} (X) & \rightarrow H^p_{1,\alpha} (X) \\
		\left( (a, \phi), \xi \right) & \mapsto  (- \nabla \xi + a , \left[ \xi , \Phi \right] + \phi),
	\end{align}
	is an isomorphism. This is a direct consequence of \Cref{cor:HodgeLemma2}. There is still the extra action of $\Gamma_\infty$ on $\mathcal{C}^p_{1, \alpha}$ and one can quotient out by its action to obtain the full quotient $\mathcal{B}^p_{1,\alpha} = \mathcal{C}^p_{1,\alpha}/ \mathcal{G}^p_{1, \alpha}$. Moreover, away from reducible configurations the action of $\mathcal{G}^p_{1 , \alpha}$ is free and so $\left(\tilde{\mathcal{B}}^p_{1,\alpha}\right)^*$ is smooth.
\end{proof}

\smallskip

\subsection{Moduli of Monopoles}

In this short subsection we finally prove the theorem establishing the main Fredholm setup describing the moduli space of $\rG_2$-monopoles; see \Cref{thm:Moduli}.\\
Fix $p \in [n/2, n)$ and $\alpha = 1-n/p \not\in \mathcal{K}(D)$, then \Cref{thm:Main_Fredholm_Theorem_2} applies to the linear map $\cL = \D + q$ defined in \cref{eq:linG2mono} as the gauge fixed linearized monopole equation. Using this, we now show that the moduli space of $\rG_2$-monopoles can be described as a quotient of the zero set of a $\Gamma_\infty$-invariant Fredholm section of the Banach space bundle 
\begin{equation}\label{BundleFredholm}
	\mathcal{F}^p_{1, \alpha} = \mathcal{C}^p_{1, \alpha} \times_{\mathcal{G}^p_{1, \alpha}(0)} H^p_{0, \alpha-1}(X, \Lambda^*X \otimes \mathfrak{g}_P),
\end{equation}
over the Banach manifold $\tilde{\mathcal{B}}^p_{1,\alpha}$. Notice that sections of this bundle are in one-to-one correspondence with $\mathcal{G}^p_{1, \alpha}(0)$-equivariant maps from $\mathcal{C}^p_{1, \alpha} \rightarrow H^p_{0, \alpha-1}(X, \Lambda^*X \otimes \mathfrak{g}_P)$ and, as we see in the proof of the next result, the monopole equation is precisely given by the map 
\begin{equation}
	\mon : \mathcal{C}^p_{1, \alpha} \rightarrow H^p_{0, \alpha-1} (X, \Lambda^*X \otimes \mathfrak{g}_P),
\end{equation} 
defined by
\begin{equation}
	\mon (\nabla,\Phi) = \ast (F_{\nabla} \wedge \psi)  - \nabla \Phi,
\end{equation}
which is invariant by the action of the gauge transformations $\mathcal{G}^p_{1, \alpha} \supset \mathcal{G}^p_{1, \alpha}(0)$. We have now everything in place to prove the main theorem in this section.

\begin{theorem}\label{thm:Moduli}
	Let $\rG = \SU (2)$ and $p \in [n/2, n)$ such that $\alpha = 1 - n/p \not\in \mathcal{K} (D)$. Then, there is a $\Gamma_\infty$-invariant Fredholm section 
	\begin{equation}
		\mon:  \tilde{\mathcal{B}}^p_{1,\alpha} \to \mathcal{F}^p_{1, \alpha},
	\end{equation} 
	of the bundle $\mathcal{F}^p_{1, \alpha} \rightarrow \tilde{\mathcal{B}}^p_{1,\alpha}$ such that the moduli space of $\rG_2$-monopoles is in bijection with
	\begin{equation}
		\mon^{-1} (0) / \Gamma_\infty \subseteq \mathcal{B}^p_{1,\alpha}.
	\end{equation}
\end{theorem}

\begin{proof}
	The monopole moduli space is the zero locus of the section $\mon$ and we must now show this is a section of $\mathcal{F}^p_{1, \alpha} \rightarrow \tilde{\mathcal{B}}^p_{1,\alpha}$. To achieve this we write $(\nabla,\Phi) \in \cC^p_{1,\alpha}$ as $(\nabla_0 + a, \Phi_0 + \phi)$ with $(a,\phi) \in H^p_{1,\alpha} (X)$. Then, using $\mon (\nabla_0,\Phi_0) = 0$
	\begin{equation}\label{eq:mon_map_Right_Space}
		\mon (\nabla,\Phi) = \mon (\nabla_0,\Phi_0) + \left( \mathrm{T}_{(\nabla, \Phi)} \mon \right) (a, \upphi) + N((a,\phi),(a,\phi))
	\end{equation}
	where $	\left( \mathrm{T}_{(\nabla, \Phi)} \mon \right) (a, \upphi)$ is as in \cref{eq:Tmon-} and $N((a,\phi),(a,\phi))$ is a quadratic term in $(a,\phi)$ which can be written as
	\begin{equation}
		N((a,\phi),(a,\phi)) = \frac{1}{2} \ast \left( [a \wedge a] \wedge \psi \right) - [a, \phi].
	\end{equation}
	This satisfies the conditions of \Cref{prop:Multiplication} and so the right hand side of \cref{eq:mon_map_Right_Space} lies in $H^p_{0,\alpha-1} (X)$. We can define $\mon^{-1} (0) $ inside $\tilde{\mathcal{B}}^p_{1,\alpha}$ using the local slices constructed in \Cref{thm:ModuliConf} which modeled on the kernel of $\cL_0^*$. Equivalently, we may instead construct such a local model for $\mon^{-1} (0)$ in the quotient $\tilde{\mathcal{B}}^p_{1,\alpha}$ by considering instead the zero locus of the joint map $\mon + \cL_0^*$. Its linearization
	\begin{equation}
		\left( \mathrm{T}_{(\nabla, \Phi)} \mon \right) \oplus \cL_0^*,
	\end{equation}
	is precisely the map $\cL = \D + q$ defined in \cref{eq:linG2mono}. For these $p,k,\alpha$ \Cref{thm:Main_Fredholm_Theorem_2} applies and the map $\cL : H^p_{1,\alpha} (X) \rightarrow H^p_{0,\alpha-1} (X)$ is therefore Fredholm.
\end{proof}

\bibliography{references}
\bibliographystyle{abstract}

\end{document}